\documentclass[11pt]{amsart}
\usepackage{comment}
\usepackage{amsfonts}
\usepackage{amssymb}
\usepackage{graphicx}
\usepackage{amsmath,amsthm}
\usepackage{latexsym}
\usepackage{amscd}
\usepackage{xypic}
\usepackage{esint}
\usepackage{extarrows}
\usepackage{mathrsfs}
\usepackage{dsfont}
\usepackage{setspace}
\usepackage{enumitem}
\usepackage[margin=1.25in]{geometry}

\usepackage{hyperref}

\allowdisplaybreaks

\numberwithin{equation}{section}

\newtheorem{prop}{Proposition}[section]
\newtheorem{thm}[prop]{Theorem}
\newtheorem{lemm}[prop]{Lemma}
\newtheorem{coro}[prop]{Corollary}

\newtheorem{claim}[prop]{Claim}
\newtheorem*{claim*}{Claim}

\theoremstyle{definition}
\newtheorem{defi}[prop]{Definition}

\newtheorem{rmk}[prop]{Remark}

\newcommand{\CC}{\mathbb{C}}
\newcommand{\DD}{\mathbb{D}}

\newcommand{\NN}{\mathbb{N}}

\newcommand{\RR}{\mathbb{R}}

\newcommand{\cA}{\mathcal A}

\renewcommand{\cD}{\mathcal D}

\newcommand{\cF}{\mathcal F}

\renewcommand{\cH}{\mathcal H}

\renewcommand{\cL}{\mathcal L}
\newcommand{\cM}{\mathcal M}

\def\ft{\mathfrak{t}}
\def\fn{\mathfrak{n}}

\DeclareMathOperator{\tr}{tr}

\DeclareMathOperator{\Span}{span}

\DeclareMathOperator{\supp}{supp}

\newcommand{\pa}[2]{\frac{\partial #1}{\partial #2}}

\newcommand{\paop}[1]{\pa{}{#1}}

\DeclareMathOperator{\Div}{div}
\DeclareMathOperator{\loc}{loc}

\newcommand{\lambdabar}{\bar{\lambda}}
\newcommand{\Lambdabar}{\bar{\Lambda}}
\newcommand{\ep}{\epsilon}

\let\oldmarginpar\marginpar
\renewcommand\marginpar[1]{\-\oldmarginpar[\raggedleft\footnotesize #1]%
{\raggedright\footnotesize #1}}

\DeclareMathOperator{\dist}{dist}

\DeclareMathOperator{\dvol}{dvol}

\setlist[enumerate]{leftmargin = 2em}

\usepackage{graphicx}

\DeclareMathOperator{\Ric}{Ric}

\title[Ginzburg-Landau solutions on manifolds with boundary]{Asymptotics for the Ginzburg-Landau equation on manifolds with boundary under homogeneous Neumann condition}
\author{Da Rong Cheng}

\begin{document}
\maketitle 

\begin{abstract} On a compact manifold $M^{n}$ ($n\geq 3$) with boundary, we study the asymptotic behavior as $\ep$ tends to zero of solutions $u_{\ep}: M \to \CC$ to the equation $\Delta u_{\ep} + \ep^{-2}(1 - |u_{\ep}|^{2})u_{\ep} = 0$ with the boundary condition $\partial_{\nu}u_{\ep} = 0$ on $\partial M$. Assuming an energy upper bound on the solutions and a convexity condition on $\partial M$, we show that along a subsequence, the energy of $\{u_{\ep}\}$ breaks into two parts: one captured by a harmonic $1$-form $\psi$ on $M$, and the other concentrating on the support of a rectifiable $(n-2)$-varifold $V$ which is stationary with respect to deformations preserving $\partial M$. Examples are given which shows that $V$ could vanish altogether, or be non-zero but supported only on $\partial M$.
\end{abstract}

\tableofcontents
\section{Introduction}
\subsection{Background and statement of main result}
Let $M^{n}\ (n \geq 3)$ be an oriented, smooth, compct Riemannian manifold, with $\partial M \neq 0$. The main goal of the present paper is to study the limiting behavior as $\ep \to 0$ of solutions to the following boundary value problem for the Ginzburg-Landau equation on $M$. 
\begin{equation}\label{GLNeumann}
\left\{
\begin{array}{cl}
- \Delta u &= \ep^{-2}(1 - |u|^{2})u \text{ in }M,\\
\partial_{\nu}u &= 0 \text{ on }\partial M,
\end{array}
\right.
\end{equation}
where the functions $u$ are complex-valued. Notice that solutions to \eqref{GLNeumann} correspond to critical points of the Ginzburg-Landau functional, defined to be
\begin{equation}
E_{\ep}(u) = \int_{M} e_{\ep}(u) \dvol,\ \text{where } e_{\ep}(u) =  \frac{|\nabla u|^{2}}{2} + \frac{(1 - |u|^{2})^{2}}{4\ep^{2}}.
\end{equation}
Specifically, $u$ is a solution to \eqref{GLNeumann} if and only if 
\begin{equation}\label{GLvariation}
\delta E_{\ep}(u)(\zeta) \equiv \int_{M} \langle \nabla u, \nabla \zeta \rangle + \frac{|u|^{2} - 1}{\ep^{2}}u \cdot \zeta \dvol = 0,
\end{equation}
for all $\zeta \in C^{1}(\overline{M}; \CC)$. Basic elliptic theory shows that if $u \in W^{1, 2}\cap L^{4}(M;\CC)$ satisfies \eqref{GLvariation} for all $\zeta \in C^{1}(\overline{M}; \CC)$, then $u$ is in fact smooth on $\overline{M}$. 

The Ginzburg-Landau functional in dimensions three or higher ($n \geq 3$) is known to be closely related to the $(n-2)$-volume functional as $\ep \to 0$. Among the large number of works along this line, we mention the result of Bethuel, Brezis and Orlandi (\cite{bbo}), which concerns the case where $M = \Omega$ is a simply-connected domain in $\RR^{n}$ and the boundary condition in \eqref{GLNeumann} is replaced by a sequence of Dirichlet boundary data $g_{\ep}: \partial \Omega \to \CC$ arranged to blow up on a smooth $(n-3)$-dimensional submanifold $S$ of $\partial \Omega$. What they showed is that if for each $\ep$ there is a solution $u_{\ep}$ to 
\begin{equation}\label{GLDirichlet}
\left\{
\begin{array}{cl}
- \Delta u_{\ep} &= \ep^{-2}(1 - |u_{\ep}|^{2})u_{\ep} \text{ in }\Omega,\\
u_{\ep} &= g_{\ep} \text{ on }\partial \Omega,
\end{array}
\right.
\end{equation}
and if $\left| \log\ep \right|^{-1}E_{\ep}(u_{\ep})$ is bounded uniformly in $\ep$, then, after possibly taking a subsequence, the measures
\begin{equation}
\mu_{\ep} = \frac{e_{\ep}(u_{\ep})}{\left| \log\ep \right|} dx
\end{equation}
converges weakly to the volume measure of a rectifiable $(n-2)$-varifold which is stationary in the interior of $\Omega$. We remark here that prior to \cite{bbo}, Lin and Rivi\`ere (\cite{lr}) have studied the case where the solutions $u_{\ep}$ are assumed to minimize $E_{\ep}$ subject to the Dirichlet boundary condition in \eqref{GLDirichlet}. They proved that in this case, the limit measure supports an integral $(n-2)$-current which minimizes volume among integral currents with boundary $S$. 


Recently, Stern has succeeded in adapting the Bethuel-Brezis-Orlandi result to closed manifolds and combining it with a min-max construction to prove the existence of a non-zero stationary rectifiable $(n-2)$-varifold (\cite{stern1}, \cite{stern2}). His work can roughly be divided into two parts: First, he showed that if for each $\ep > 0$ there is a solution $u_{\ep}$ to
\begin{equation}\label{GLclosed}
-\Delta u_{\ep} = \ep^{-2}(1 - |u_{\ep}|^{2})u_{\ep} \text{ on }M,
\end{equation}
and if $E_{\ep}(u_{\ep}) = O(\left| \log\ep \right|)$ as $\ep \to 0$, then there exists a sequence of $S^{1}$-valued harmonic maps $\varphi_{\ep}$ on $M$ and a subsequence $\ep_{k}$ such that, denoting $\tilde{u}_{\ep} = \varphi_{\ep}^{-1} u_{\ep}$, the measures $\nu_{k} = \left| \log\ep_{k} \right|^{-1} e_{\ep_{k}}(\tilde{u}_{k})\dvol$ converge weakly to $\|V\|$, the volume measure of a stationary rectifiable $(n-2)$-varifold $V$. Moreover, the mass of $V$ is given by
\begin{equation}
\|V\|(M) = \lim\limits_{k \to \infty} \left| \log\ep_{k} \right|^{-1}\left( E_{\ep_{k}}(u_{k}) - \frac{\| h_{\ep_{k}} \|^{2}_{2; M}}{2}\right),
\end{equation} 
where $h_{\ep_{k}}$ is the harmonic part in the Hodge decomposition of the $1$-form $u_{\ep_{k}} \times du_{\ep_{k}}$. Then, via a two-parameter min-max construction, he produced a sequence of solutions $\{u_{\ep}\}$ to which the above convergence result can be applied to get a limit varifold and showed that this varifold is non-zero. We refer the reader to \cite{stern1} and \cite{stern2} for the details.

In light of Stern's work, it's natural ask whether one can follow a similar strategy when $\partial M \neq \emptyset$ to prove the existence of a rectifiable $(n-2)$-varifold which is stationary with free boundary, i.e. stationary with respect to deformations preserving $\partial M$. From a variational point of view, the correct boundary value problem to study for the above purpose is the homogeneous Neumann problem \eqref{GLNeumann}, and the principal result of this paper, stated below as Theorem \ref{GLgeneralconvergence}, establishes a version of the Bethuel-Brezis-Orlandi result for solutions to \eqref{GLNeumann} under a convexity assumption on $\partial M$. This extends the first part of Stern's work to manifolds with boundary.

Below, we will assume that our manifold $M^{n}\ (n \geq 3)$ is isometrically embedded in a closed Riemannian manifold $\widetilde{M}$ of the same dimension, and that the latter is isometrically embedded in an Euclidean space $\RR^{N}$. Furthermore, we assume that $\partial M$ is convex in the sense that if $\nu$ denotes the unit normal of $\partial M$ pointing into $M$, then we have
\begin{equation}\label{convexboundary}
\langle \nabla_{\xi}\nu(p), \xi \rangle \leq 0\text{ for all } \xi \in T_{p}\partial M \text{ and for all } p \in \partial M.
\end{equation}
We now state our main theorem.

\begin{thm}\label{GLgeneralconvergence}
Under the above assumptions, take a sequence $\{\ep_{k}\}$ of parameters converging to zero and suppose for each $k$ there is a solution $u_{k}$ to \eqref{GLNeumann} with $\ep = \ep_{k}$. Moreover, assume that there exists a constant $K_{0}$ independent of $k$ such that 
\begin{equation}\label{GLenergyupperbound}
E_{\ep_{k}}(u_{k}) \leq K_{0}\left| \log\ep_{k} \right| \text{ for all }k,
\end{equation}
and define the following Radon measures 
\begin{equation}\label{energymeasure}
\mu_{k} = \left| \log\ep_{k} \right|^{-1} e_{\ep_{k}}(u_{k})\dvol.
\end{equation}
Then, the exists a closed, countably $(n-2)$-rectifiable set $\Sigma \subset \overline{M}$ with $H^{n-2}(\Sigma) < \infty$, and a harmonic $1$-form $\psi$ on $M$ with $\psi(\nu) \equiv 0$ on $\partial M$, such that, passing to a subsequence if necessary, the following statements hold.
\begin{enumerate}
\item[(1)] In the sense of measures on $\overline{M}$, 
\begin{equation}\label{GLlimitdecomp}
\mu \equiv \lim\limits_{k \to \infty}\mu_{k} = \frac{|\psi|^{2}}{2}\dvol+ \mu_{s}.
\end{equation}
where $\mu_{s}$ is a nonnegative measure with $\supp(\mu_{s}) = \Sigma$.
\item[(2)] Denoting by $B^{N}_{r}(p)$ the ball in $\RR^{N}$ centered at $p$ with radius $r$, the density $\Theta(\mu_{s}, p) \equiv \lim_{r \to 0}r^{2-n}\mu_{s}(B^{N}_{r}(p) \cap \overline{M})$ exists for all $p \in \Sigma$. Moreover, the rectifiable $(n-2)$-varifold $V \equiv (\Sigma, \Theta(\mu_{s}, \cdot))$ satisfies
\begin{equation}\label{GLlimitstationary}
\delta V(X) = \int_{\Sigma} \Div_{T_{y}\Sigma} X(y) \Theta(\mu_{s}, y)dH^{n-2}(y) = 0,
\end{equation}
for all $C^{1}$-vector fields $X$ on $\widetilde{M}$ satisfying $\langle X(y), \nu(y) \rangle = 0$ for $y \in \partial M$.
\end{enumerate}
\end{thm}
\begin{defi}\label{stationaryterminology} We take this opportunity to introduce the following terminology.
\begin{enumerate}
\item[(1)] The class of vector fields considered in conclusion (2) will be referred to as the class of \textit{admissible vector fields}. Note that the deformations of $\widetilde{M}$ generated by an admissible vector field always preserve $\partial M$, and hence $\overline{M}$. 
\item[(2)] Borrowing the terminology from \cite{martinxin}, we say that a varifold is stationary \textit{with free boundary} if it satisfies \eqref{GLlimitstationary} for all admissible vector fields.
\end{enumerate}
\end{defi}

\begin{rmk}\label{simplyconnected} 
Regarding the conclusions of Theorem \ref{GLgeneralconvergence}, we note the following.
\begin{enumerate}
\item[(1)] Unless further topological constraints are put on $M$, we cannot rule out the absolutely continuous part on the right-hand side of \eqref{GLlimitdecomp}. In fact, one may take $M = S^{1} \times S^{2}_{+}$ with the standard product metric $g = ds^{2} + dr^{2} + \sin^{2}r d\theta$, $\ep_{k} = e^{-k^{2}}$ and $u_{k}(s, r, \theta) = (1 - k^{2}\ep_{k}^{2})^{1/2}e^{iks}$. Then, similar to Remark 1.3 of \cite{stern1} one verifies that $u_{k}$ is a solution to \eqref{GLNeumann} with $\ep = \ep_{k}$, but the measures $\mu_{k} =\left| \log\ep_{k} \right|^{-1}e_{\ep_{k}}(u_{k})\dvol$ converge to a constant (non-zero) multiple of the volume measure of $M$. In particular, the singular part $\mu_{s} = \|V\|$ in \eqref{GLlimitdecomp} vanishes altogether in this case.

\item[(2)] As in \cite{bbo} and \cite{stern1}, \cite{stern2}, we do not know if the density $\Theta(\mu_{s}, \cdot)$ is an integer (up to a constant multiple) at $H^{n-2}$-a.e. $p \in \Sigma$. However, by the work of Lin and Rivi\`ere (\cite{lr}), this would be true at $H^{n-2}$-a.e. interior points $p \in \Sigma \cap M$ provided we assume in addition that each $u_{\ep}$ minimizes $E_{\ep}$ on compact subsets of $M$.

\item[(3)] Note that any $(n-2)$-varifold supported on $\partial M$ and stationary with respect to tangential deformations is also stationary with free boundary. Thus \eqref{GLlimitstationary} alone does not rule out the possibility that $\|V\|(\partial M) > 0$. In fact, in Section 8 we construct a sequence of solutions such that the varifold $V$ yielded by Theorem \ref{GLgeneralconvergence} satisfies $V \neq 0$ but $\|V\|(M) = 0$. In other words, all the energy is concentrating on the boundary.

\end{enumerate}
\end{rmk}
\begin{rmk}\label{bpwoverlap}
We mention a few previous or recent results besides \cite{bbo}, \cite{stern1} and \cite{stern2} that are related to Theorem \ref{GLgeneralconvergence}. All the results below assume an energy upper bound similar to \eqref{GLenergyupperbound}.
\begin{enumerate}
\item[(1)] Bethuel, Orlandi and Smets (\cite{bosring}) considered solutions to a slightly more general class of equations on a domain $\Omega\subset \RR^{n}$, but without imposing a boundary condition. They established a version of \eqref{GLlimitdecomp} for the limit measure restricted to $\Omega$. Moreover, the singular part was shown to define a rectifiable $(n-2)$-varifold which is stationary in the interior, i.e. \eqref{GLlimitstationary} holds for all $X$ compactly supported in $\Omega$. The same authors also studied the parabolic Ginzburg-Landau equation on the whole space (\cite{bospara}).

\item[(2)] The Neumann problem \eqref{GLNeumann} was previously studied by Chiron (\cite{chiron}) on a domain $\Omega \subset \RR^{n}$. (In fact, he worked with the same type of equations as in \cite{bosring}.) A principal result there was a boundary $\eta$-ellipticity theorem (cf. Sections 1.2 and 5 below). In addition, by locally reflecting across $\partial \Omega$, he invoked the interior estimates in \cite{bosring} to get a rectifiable varifold $V$ in $\overline{\Omega}$ and inferred that the sum of $V$ and its reflection is stationary. However, unless $\partial\Omega$ is totally geodesic, the reflected metric is generally not $C^{1}$, and thus the results in \cite{bosring} may not directly apply. For this reason, below we have taken slightly more care when using reflection to deduce finer properties of the solutions near the boundary.

\item[(3)] Very recently, Bauman, Phillips and Wang (\cite{bpw}) studied solutions on a simply-connected domain $\Omega \subset \RR^{n}$ under the "weak anchoring" boundary condition:
\begin{equation*}
\partial_{\nu}u + \lambda_{\ep}(u_{\ep} - g_{\ep}) = 0 \text{ on }\partial \Omega,
\end{equation*}
where $\lambda_{\ep} = K\ep^{-\alpha}$ for some $K > 0$, $\alpha \in [0, 1)$, and $\{g_{\ep}\}$ is a prescribed bounded sequence in $C^{2}(\partial \Omega; S^{1} )$. They derived interior and boundary estimates to show that the limit measure is rectifiable in $\overline{\Omega}$ and defines a stationary varifold in the interior. Moreover, as in \cite{bbo}, the solutions were shown to converge to an $S^{1}$-valued harmonic map away from the support of the varifold.
\end{enumerate}
\end{rmk}
\subsection{Outline of proof}
The proof of Theorem \ref{GLgeneralconvergence} is modeled on the arguments in \cite{bbo}, with some devices adapted from \cite{bbh}, \cite{bos} and \cite{bospara}. Our main contribution would be the analysis of $\{u_{k}\}$ near boundary points of $M$, since most of the necessary interior estimates were already done in \cite{bbo} or \cite{stern2}. Note that by the assumption \eqref{GLenergyupperbound}, passing to a subsequence if necessary, we may assume that $\mu_{k}$ converges in the sense of measures on $\overline{M}$ to a limit measure $\mu$. The proof then goes through the following steps.
\vskip 2mm
\noindent\textbf{Step 1: $\eta$-ellipticity and local estimates.}
We begin with a version of the so-called "$\eta$-ellipticity" theorem that applies to boundary points of $M$. Roughly speaking, we show that there exists a threshold $\eta_{0}$ depending only on $\overline{M}$ and the isometric embedding $\widetilde{M} \to \RR^{N}$ such that for small enough balls and for sufficiently large $k$, if
\begin{equation}\label{GLenergyeta}
r^{2-n}\int_{B_{r}^{N}(p) \cap \overline{M}} e_{\ep_{k}}(u_{k}) \dvol \leq \eta \left| \log(r^{-1}\ep_{k}) \right|,
\end{equation}
for some $\eta < \eta_{0}$, then in fact $|u|$ is close to $1$ in a smaller ball. Consequently, we may locally write $u_{k} = |u_{k}| e^{i\varphi_{k}}$, with $\varphi_{k}$ single-valued, and estimate $|u_{k}|$ and $\varphi_{k}$ separately as in \cite{bbh} and \cite{bos} to show that, decreasing the threshold $\eta_{0}$ above if necessary, we obtain estimates on the derivatives of $\{u_{k}\}$ within balls verifying \eqref{GLenergyeta}. (It is in order to obtain boundary estimates at this stage that we use the convexity assumption on $\partial M$.)  

\vskip 2mm
\noindent\textbf{Step 2: Definitions and properties of $\Sigma$ and $\psi$.}
We define the set $\Sigma$ in the conclusion of Theorem \ref{GLgeneralconvergence} to be 
\begin{equation}\label{GLsigmadefinition}
\Sigma = \{ p \in \overline{M}| \ \Theta(\mu, p) \equiv \lim\limits_{r \to 0}r^{2-n}\mu(B_{r}^{N}(p) \cap \overline{M}) > 0 \},
\end{equation}
and show that in fact 
\begin{equation}\label{GLdensitybounds}
CK_{0} \geq \Theta(\mu, p) \geq \underline{\eta} > 0, \text{ everywhere on }\Sigma,
\end{equation}
where $\underline{\eta}$ is a threshold determined by the previous step, and $C$ is a universal constant. The existence of the limit in \eqref{GLsigmadefinition} as well as the upper bound in \eqref{GLdensitybounds} are consequences of a monotonicity formula, whereas the lower bound follows from the estimates in the previous step. Note that the density lower bound and the assumption \eqref{GLenergyupperbound} imply that $H^{n-2}(\Sigma) < \infty$. 

Next, again using the estimates in Step 1, we show that up to taking a further subsequence, the $1$-forms $\left| \log\ep_{k} \right|^{-1}u_{k} \times du_{k} \equiv \left| \log\ep_{k} \right|^{-1}\left(u^{1}_{k}du^{2}_{k} - u^{2}_{k}du^{1}_{k}\right)$ converge in $C^{m}_{\loc}(M \setminus \Sigma)$ for all $m$ and in $C^{1}_{\loc}(\overline{M} \setminus \Sigma)$. Moreover, the limit 1-form $\psi$ is actually smooth across $\Sigma$ (here we need to use the fact that $H^{n-2}(\Sigma) < \infty$) and has all the asserted properties. A consequence of the convergence of $\left| \log\ep_{k} \right|^{-1}u_{k} \times du_{k}$ to $\psi$ is that $\mu  = (|\psi|^{2}/2)\dvol$ when restricted to $\overline{M} \setminus \Sigma$.

\vskip 2mm
\noindent\textbf{Step 3: Rectifiability and stationarity of $\mu\ \llcorner \Sigma$.}
Defining $\mu_{s} \equiv \mu - (|\psi|^{2}/2)\dvol = \mu\ \llcorner \Sigma$, then clearly conclusion (1) of Theorem \ref{GLgeneralconvergence} holds. Moreover, the previous step implies that $\Theta(\mu_{s}, p) = \Theta(\mu, p) \in [\eta_{0}, CK_{0}]$ everywhere on $\Sigma$. We then invoke Preiss' theorem (\cite{preiss}) to conclude that $\Sigma$ is an $(n-2)$-rectifiable set in $\RR^{N}$, and that $\mu_{s} = \Theta(\mu_{s}, \cdot) H^{n-2}\ \llcorner \Sigma$.

To obtain the identity \eqref{GLlimitstationary}, we choose an admissible vector field $X$ and substitute $\zeta = \langle X, \nabla u_{k} \rangle$ in \eqref{GLvariation} to obtain
\begin{equation}
\int_{M} e_{\ep_{k}}(u_{k}) \Div_{g}X - \langle \nabla_{\nabla u_{k}}X, \nabla u_{k} \rangle \dvol_{g} = 0.
\end{equation}
Introducing a local orthonormal frame $\{e_{i}\}_{i\leq n}$ on $\overline{M}$, we shall see that as $k \to \infty$, the above identity yields
\begin{equation}\label{GLfreelystationary}
\int_{\Sigma} (\delta_{ij} - A_{ij}(p)) \langle \nabla_{e_{i}}X, e_{j} \rangle(p) \Theta(\mu_{s}, p)dH^{n-2}(p) = 0,
\end{equation}
for some $H^{n-2}$-measurable functions $(A_{ij})_{i, j \leq n}$ on $\Sigma$. Finally, we show that for $H^{n-2}$-a.e. $p \in \Sigma$, the matrix $\delta_{ij} - A_{ij}(p)$ represents orthogonal projection onto the approximate tangent plane to $\Sigma$ at $p$, and hence by \eqref{GLfreelystationary}, the rectifiable $(n-2)$-varifold $(\Sigma, \Theta(\mu_{s}, \cdot))$ is stationary with free boundary. Our approach to Step 3 is similar to that in \cite{bospara}, Part II, Section 5.
\subsection{Organization of the paper} The remainder of this paper is organized as follows: 

In Section 2, we set up various notations and introduce some terminology. In Section 3, we recall some basic properties of solutions to \eqref{GLNeumann}, among which is a reflection lemma (Lemma \ref{GLreflection}) which will be used repeatedly to help us obtain estimates near the boundary. Then, Section 4 is devoted to the monotonicity formula, which relate the Ginzburg-Landau energy of a solution at different scales. 

In Sections 5 and 6, we carry out Step 1 and Step 2 in the proof of Theorem \ref{GLgeneralconvergence}. Specifically, Section 5 establishes the $\eta$-ellipticity theorem and some of its consequences. Sections 6.1 and 6.2 derives higher-order estimates. In Section 6.3, we define $\Sigma$ and obtain $\psi$ as a limit away from $\Sigma$. We then verify that $\psi$ has all the asserted properties. After that, Section 7 is devoted to Step 3 of the proof of Theorem \ref{GLgeneralconvergence}. Finally, in Section 8 we give an example where the limit varifold $V$ is non-zero but supported on $\partial M$.

Certain details have been relegated to the appendices. In Appendix A we adapt some devices from \cite{bbh} to our setting. In Appendix B, we prove an upper bound for the Green matrix of a class of second order elliptic systems, which is important for the boundary estimates in Section 5. In Appendix C, we gather some known results concerning the Hodge Laplacian with respect to a Lipschitz metric, and also explain how the results in Appendix B apply in this setting. In Appendix D, we give a proof of the fact that, near a boundary point, the metric components in terms of Fermi coordinates (introduced in Section 2.2) can be estimated depending only on the geometry of $M$ and $\partial M$.
\vskip 2mm
\noindent\textbf{Acknowledgements.} This paper is an extension of my thesis work, and I'd like to thank my advisor Richard Schoen for his guidance and encouragements along the way. Thanks also go to Andre Neves for several enlightening conversations, and to Changyou Wang for having kindly responded to my questions concerning \cite{bpw}. I also thank the Forschungsinstitut f\"ur Mathematik at ETH Z\"urich, where part of the work was carried out.
\section{Notation and terminology}
\subsection{Basic conventions}
Below we summarize the general terms and notations to be used in this paper. Some additional notation and terminology are introduced in Sections 2.2 through 2.4.
\begin{center}
\begin{tabular}{ll}
$\DD$ & The open unit disk in $\RR^{2} \simeq \CC$.\\
$B^{k}_{r}(p)$ & The open ball in $\RR^{k}$ centered at $p$ with radius $r$.\\
\ & (We often drop $k$ if $k = n = \dim(M)$, $r$ if $r = 1$ and $p$ if $p = 0$.)\\
$B^{+}_{r}(p)$ & The p[en upper half ball $B_{r}(p) \cap \{x^{n} > 0\}$ for $p \in \{x^{n} = 0\}$.\\
$T_{r}(p)$ & The set $B_{r}(p) \cap \{x^{n} = 0\}$ for $p \in \{x^{n} = 0\}$.\\
$D, \partial$ & Usual partial derivative.\\
$\nabla^{g}$ & Covariant derivative with respect to the metric $g$.\\
$\Delta_{g}$ & Laplace-Beltrami operator (or the Hodge Laplace operator) with respect to $g$.\\
\ & (In the two previous notations, we drop $g$ when this is not misleading.)\\
$|f|_{0; A}$ & The $C^{0}$-norm of $f$ on $A$.\\
$[f]_{0, \mu; A}$ & The $C^{0, \mu}$ H\"older semi-norm of $f$ on $A \subset \RR^{n}$.\\
$|f|_{k, 0; A}$ & The $C^{k}$-norm $\sum_{i = 0}^{k}|D^{i}f|_{0; A}$ for $A \subset \RR^{n}$.\\
$|f|_{k, \mu; A}$ & The $C^{k, \mu}$-norm $|f|_{k, 0; A} + [D^{k}f]_{0, \mu; A}$ for $A\subset \RR^{n}$.\\
$\|f\|_{p; A}$ & The $L^{p}$-norm of $f$ on $A$.\\
$\|f\|_{k, p; A}$ & The $W^{k, p}$-norm $\sum_{i =0}^{k}\|D^{i}f\|_{p; A}$ for $A \subset \RR^{n}$.\\
$H^{s}(E)$ & The $s$-dimensional Hausdorff measure of the set $E \subset \RR^{N}$.\\
$\mu\ \llcorner K$ & The restriction of the measure $\mu$ to $K$, defined by $(\mu\ \llcorner K)(E) = \mu(E \cap K)$.\\
$\dvol$ or $\dvol_{g}$ & The volume measure with respect to the metric $g$.\\
$d_{g}(\cdot, \cdot)$ or $d(\cdot, \cdot)$ & The distance on $\widetilde{M}$ with respect to $g$.\\
$d_{0}(\cdot, \cdot)$ & Euclidean distance on $\RR^{N}$.\\
$\langle \cdot,\ \cdot \rangle$ & An innerproduct induced by $g$.\\
$\cdot$ & Standard innerproduct on Euclidean space.\\
\end{tabular}
\end{center}
In addition, for two complex numbers $u = u^{1} + iu^{2}$ and $v = v^{1} + iv^{2}$, we let 
\begin{equation*}
u\times v = u^{1}v^{2} - u^{2}v^{1} = -\text{Im}(u\overline{v}).
\end{equation*}
Finally, in what follows, a constant is said to be \textit{universal} if it depends only on $M, \widetilde{M}$, the metric on $\widetilde{M}$, and the isometric embedding $\widetilde{M} \to \RR^{N}$.

\subsection{Classes of metrics}
A good portion of the results we discuss in this paper are local, and hence we will often be working on a ball or half ball in $\RR^{n}$ and equipped with a metric $g$ in one of the classes defined below. 
\begin{defi}\label{interiorclass}
Fix $\mu \in (0, 1]$ and $r, \lambda, \Lambda > 0$, we define $\cM_{\lambda, \Lambda; r}$ to be the class of metrics $g = (g_{ij})$ with Lipschitz continuous components on $B_{r}$, satisfying the following estimates.
\begin{enumerate}
\item[(l1)]  $g_{ij}(0) = \delta_{ij}$ and $r\left[g_{ij}\right]_{0, 1; B_{r}} \leq \Lambda$ for all $i, j$.
\item[(l2)] $\lambda |\xi|^{2} \leq g_{ij}(x)\xi^{i}\xi^{j} \leq \lambda^{-1} |\xi|^{2},\text{ for all }x\in B_{r},\ \xi \in \RR^{n}.$
\end{enumerate}
Similarly, we define $\cM_{\mu, \lambda, \Lambda; r}$ to consist of metrics $g = (g_{ij})$ with components in $C^{1, \mu}(B_{r})$, such that
\begin{enumerate}
\item[(h1)]  $g_{ij}(0) = \delta_{ij}$ and $r|\partial_{k}g_{ij}|_{0; B_{r}} + r^{1 + \mu}\left[\partial_{k}g_{ij}\right]_{0, \mu; B_{r}} \leq \Lambda$ for all $i, j, k$.
\item[(h2)] $\lambda |\xi|^{2} \leq g_{ij}(x)\xi^{i}\xi^{j} \leq \lambda^{-1} |\xi|^{2},\text{ for all }x\in B_{r},\ \xi \in \RR^{n}.$
\end{enumerate}
\end{defi}

\begin{defi}\label{boundaryclass}
For $\mu \in (0, 1]$ and $r, \lambda, \Lambda > 0$, we let $\cM^{+}_{\lambda, \Lambda; r}$ (resp., $\cM^{+}_{\mu, \lambda, \Lambda; r}$) consist of metrics $g = (g_{ij}) \in \cM_{\lambda, \Lambda; r}$ satisfying conditions (l1) and (l2) (resp., (h1) and (h2)) in the previous definition, but with $B_{r}$ replaced by $B_{r}^{+}$. In addition, we require that 
\begin{enumerate}
\item[(1)] $g_{in}(x) = \delta_{in}$ for all $x \in T_{r}$ and for all $i$. 
\item[(2)] $g_{nn}$ and $g_{ij}$ for $i, j \leq n-1$ are all even functions in the $x^{n}$-variable, while $g_{in}$ for $i \leq n-1$ are odd functions.
\end{enumerate}
\end{defi}
\begin{rmk}
\begin{enumerate}
\item[(1)] Conditions (1) and (2) in Definition \ref{boundaryclass} implies that $\tau^{\ast}g = g$, where $\tau:(x', x^{n}) \to (x', -x^{n})$ denotes reflection across $T$. Consequently the inward unit normal to $T$ with respect to $g$ coincides with $\partial_{n}$. 

\item[(2)]We clearly have $\cM_{\mu, \lambda, \Lambda; r} \subset \cM_{\lambda, \Lambda; r}$ and $\cM^{+}_{\mu, \lambda, \Lambda; r} \subset \cM_{\lambda, \Lambda; r}^{+}$ for all $\mu \in (0, 1]$. 

\item[(3)]When $r = 1$, we drop the subscript $r$ from the notation for the classes.
\end{enumerate}
\end{rmk}
\begin{rmk}\label{singularmetric}
The motivation for introducing these classes is that we want to derive boundary estimates by reflection, but the metric is not $C^{1}$ across the boundary unless the boundary is totally geodesic.
\end{rmk}
\subsection{Local coordinates}
For $p \in \widetilde{M}$, the open geodesic ball on $\widetilde{M}$ centered at $p$ with radius $r$ is denoted by $\tilde{B}_{r}(p)$. On the other hand, if $p \in \partial M$, then we take geodesic normal coordinates $(x^{1}, \cdots, x^{n-1})$ of $\partial M$ centered at $p$ and let $x^{n} = t = \dist(\cdot, \partial M)$. These constitute a system of coordinates near $p$, sometimes referred to as Fermi coordinates. Borrowing the terminology in \cite{martinxin}, we define the Fermi distance $\tilde{r}$ to be $\sqrt{|x_{1}|^{2} + \cdots + |x_{n-1}|^{2} + t^{2}}$ and define the Fermi ball $\tilde{B}_{r}^{+}(p)$ to be $\{ q \in M|\ \tilde{r} < r \}$, which corresponds to the half ball $B_{r}^{+} \subset \RR^{n}$ under Fermi coordinates.

The metric components and their derivatives in Fermi coordinates can be estimated through the geometry of $M$ and $\partial M$. Namely, letting $R$ denote the curvature tensor of $M$ and $h$ the second fundamental form of $\partial M$ in $M$, then we have the following proposition, which is a well-known fact. Nonetheless, we give a proof in Appendix D. We also note here that similar estimates have been derived in \cite{martinxin}.

\begin{prop}[cf. \cite{martinxin}, Appendix A]
\label{fermiproperties}
Letting $B_{0} = \sup_{M}\left| R \right|_{g}$ and $A_{0} = \sup_{\partial M}|h|$. Then there exists a small radius $r_{0} < 1$ depending only on $n, B_{0}$ and $A_{0}$ such that for all $p \in \partial M$, the following holds.
\begin{equation}\label{fermicomparable}
\frac{1}{2}\delta_{ij}\xi^{i}\xi^{j} \leq g_{ij}(x)\xi^{i}\xi^{j} \leq 2 \delta_{ij}\xi^{i}\xi^{j}, \text{ for all }x \in B_{r_{0}}^{+}, \xi \in \RR^{n},
\end{equation}
where $g_{ij}$ are the components of the metric in terms of Fermi coordinates $(x^{1}, \cdots, x^{n-1}, t)$ centered at $p$.
Moreover, for each $k$ there is a constant $C$ depending only on $n, k, \{B_{j} \equiv \sup_{M}\left| \nabla^{j}R \right|_{g}\}_{1\leq j \leq k}$ and $\{A_{j} \equiv \sup_{\partial M} \left| \nabla^{j}h \right|_{g} \}_{1 \leq j \leq k}$, such that 
\begin{equation}\label{fermiderivative}
\left| \paop{x^{i_{1}}}\cdots \paop{x^{i_{k}}} g_{ij} \right| \leq C_{k} \text{ on }B_{r_{0}}^{+}.
\end{equation}
\end{prop}
\begin{rmk}\label{fermirmk}
\begin{enumerate}
\item[(1)] Let $(g_{ij})$ be as in Proposition \ref{fermiproperties} and reflect each component evenly across $T$. Then it follows from Proposition \ref{fermiproperties} that the resulting metric lies in $\cM^{+}_{1, \lambda, \Lambda; r_{0}}$, where $\Lambda$ and $\lambda$ depend only on the parameters listed before \eqref{fermiderivative} (with $k = 1$). 

\item[(2)] By Proposition \ref{fermiproperties} one infers that there exists a threshold $\tilde{r}_{0}$ and an increasing function $c:(0, \tilde{r}_{0}) \to (0, 1/2)$ with $\lim_{r \to 0}c(r) = 0$, both universally determined, such that
\begin{equation*}
B^{N}_{(1 - c(r))r}(p) \cap \overline{M} \subset \tilde{B}^{+}_{r}(p) \subset B^{N}_{(1 + c(r))r}(p)\cap \overline{M},
\end{equation*}
for all $p \in \partial M$ and $r < \tilde{r}_{0}$. A similar relation holds between geodesic balls $\tilde{B}_{r}(p)$ and the balls $B^{N}_{r}(p) \cap \widetilde{M}$.
\end{enumerate}
\end{rmk}
\subsection{Differential forms}
The reference for the material summarized below is Chapter 7 of \cite{morreybook} or Chapter 5 of \cite{gms}, although our notation will sometimes be different. 

Given a Riemannian manifold $(M, g)$, possibly with non-empty boundary, we let $L^{p}_{r}(M)$ denote the space of $r$-forms $\omega$ such that
\begin{equation*}
\|\omega\|_{p; M}^{p} \equiv \int_{M} |\omega|_{g}^{p}\dvol_{g} < \infty.
\end{equation*}
We say that $\omega \in W^{m, p}_{r}(M)$ if and only if 
\begin{equation*}
\|\omega\|_{m, p; M}^{p}  \int_{M} \sum\limits_{0 \leq k \leq m}|\nabla^{k}\omega|_{g}^{p}\dvol_{g} < \infty.
\end{equation*}
For each $r \geq 1$ and $x \in M$, the innerproduct induced by $g$ on $\bigwedge^{r}T_{x}^{\ast}M$ will be denoted simply by $\langle \cdot,\ \cdot \rangle$. Integrating over $M$, we get a pairing between differential forms on $M$, which we denote by $(\cdot, \ \cdot)$. Letting $d$ denote the exterior differential, we denote by $d^{\ast}$ its formal adjoint with respect to $(\cdot,\ \cdot)$. Thus $d$ and $d^{\ast}$ are related by the fact that for all $\alpha \in C^{\infty}(\overline{M}; \bigwedge^{r}T^{\ast}M )$ and $\beta \in C^{\infty}_{c}(M; \bigwedge^{r}T^{\ast}M)$, we have
\begin{equation}\label{exteriorduality}
(d^{\ast}\alpha, \beta) = (\alpha, d\beta).
\end{equation}

Given a form $\omega \in C^{\infty}(M; \bigwedge^{r}T^{\ast}M)$, its tangential and normal traces on $\partial M$, denoted $\ft \omega$ and $\fn \omega$, respectively, are defined as in Section 7.5 of \cite{morreybook}. Both $\ft$ and $\fn$ maps from $W^{1, p}_{r}(M)$ to $L^{p}_{r}(\partial M)$, and we define $W^{1, p}_{r, \ft}(M)$ to be the subspace of $W^{1, p}_{r}(M)$  consisting of those forms $\omega$ with $\ft \omega = 0$. Next we let $\cH_{r, \ft}(M)$ denote the space of forms $\omega$ in $W^{1, 2}_{r, \ft}(M)$ such that 
\begin{equation}\label{harmonicdefi}
\cD_{g}(\omega, \zeta; M) \equiv \int_{M}\langle d\omega, d\zeta \rangle + \langle d^{\ast}\omega, d^{\ast}\zeta \rangle \dvol_{g} = 0, \text{ for all } \zeta \in W^{1, 2}_{r, \ft}(M).
\end{equation}
The Hodge decomposition theorem (cf. \cite{morreybook}, Theorem 7.7.7) then says that each form $\omega \in L^{2}_{r}(M)$ can be uniquely decomposed as
\begin{equation*}
\omega = h + d\alpha + d^{\ast}\beta,
\end{equation*}
where $h \in \cH_{r, \ft}(M)$, $\alpha \in W^{1, 2}_{r-1, \ft}(M)$ and $\beta \in W^{1, 2}_{r + 1, \ft}(M)$, with $d^{\ast}\alpha = 0$ and $d\beta = 0$.
The spaces $W^{1, p}_{r, \fn}(M)$ and $\cH_{r, \fn}(M)$ are defined similarly, and the above decomposition theorem also holds with $\ft$ replaced by $\fn$. Also, the $\|\cdot\|_{1, 2; M}$-closure of smooth, compactly supported $r$-forms in $W^{1, 2}_{r}(M)$ will be denoted $W^{1, 2}_{r, 0}(M)$. In the case $M = B_{1}$, we introduce one more class of differential forms and define
\begin{equation*}
W^{1, 2, +}_{r, \ft}(B_{1}) \equiv \{ \varphi \in W^{1, 2}_{r, \ft}(B_{1})|\ \tau^{\ast}\varphi = \varphi \},
\end{equation*}
where $\tau: (x', x^{n}) \to (x', -x^{n})$ denotes reflection across $\{x^{n} = 0\}$. 

\begin{rmk}\label{lipschitzhodge}
Most of the results in Chapter 7 of \cite{morreybook}, including the definitions of $\ft$ and $\fn$, and the decomposition Theorem 7.7.7, require only a Lipschitz metric, and hence apply when we equip $B_{1}$ with a metric from $\cM_{\lambda, \Lambda}$ or $\cM^{+}_{\lambda, \Lambda}$. In particular, this implies that $\cH_{r, \ft}(B_{1}) = \cH_{r, \fn}(B_{1}) = \{0\}$ for any metric $g \in \cM_{\lambda, \Lambda}$, as is the case when $g$ is smooth.
\end{rmk}

\section{Basic properties of solutions}

In this short section, we recall a few basic local and global properties of solutions to \eqref{GLNeumann}, beginning with the following standard estimates.
\begin{lemm}\label{GLbasicbounds}
Suppose $u$ solves \eqref{GLNeumann}. Then the following estimates hold.
\begin{enumerate} 
\item[(1)] $|u| \leq 1$ on $M$.
\item[(2)] $|\nabla u| \leq C\ep^{-1}$ on $M$, with $C$ a universal constant.
\end{enumerate}
\end{lemm}
\begin{proof}
To prove assertion $(1)$, we observe that the function $\varphi \equiv |u|^{2} - 1$ satisfies
\begin{equation}\label{modulussquareeq}
\Delta_{g} \varphi  - 2\frac{|u|^{2}}{\ep^{2}}\varphi \geq 0.
\end{equation}
Moreover, we clearly have $\partial_{\nu}\varphi = 0$ on $\partial M$. Thus by the strong maximum principle and the Hopf lemma, either $\varphi$ is a constant (which must then be non-positive by \eqref{modulussquareeq}) or $\varphi$ has a non-positive maximum. In either case, we conclude that $|u| \leq 1$.

For $(2)$ we cover $M$ by geodesic or Fermi balls and use $(1)$ in conjunction with the multiplicative Schauder estimates in Lemmas \ref{interiormulti} and \ref{boundarymulti}. 
\end{proof}
The gradient estimate in Lemma \ref{GLbasicbounds}(2) can of course be localized, and we get the following.
\begin{lemm}\label{localbasicGLestimate}
Let $g \in \cM_{\lambda, \Lambda}$ and suppose $u : (B_{1}, g) \to \DD$ is a solution to 
\begin{equation}\label{GLequation}
- \Delta_{g} u = \ep^{-2}(1 - |u|^{2})u,
\end{equation}
where $\DD$ is the unit disk in $\CC$. Then
\begin{equation*}
|\nabla u|_{0; B_{3/4}} \leq C\ep^{-1},
\end{equation*}
where $C$ depends only on $n, \lambda$ and $\Lambda$.
\end{lemm}
\begin{proof}
Since $|\Delta_{g}u|_{0; B_{1}} \leq \ep^{-2}$, we apply Lemma \ref{interiormulti} with $s = 3/4$ to see that 
\begin{equation*}
|Du|_{0; B_{3/4}} \leq C\left( 1 + \ep^{-1} \right),
\end{equation*}
which gives the result since $\nabla u = g^{ij}\partial_{j}u \partial_{i}$.
\end{proof}

The following reflection lemma was already used in \cite{chiron} and will also be important to us. 
\begin{lemm}\label{GLreflection}
Let $g \in \cM^{+}_{\lambda, \Lambda}$ and suppose $u \in W^{1, 2}_{\text{loc}}(B_{1}^{+} \cup T; \DD)$ solves
\begin{equation}\label{halfGL}
\left\{
\begin{array}{cl}
-\Delta_{g} u &= \ep^{-2}(1 - |u|^{2})u \text{ in }B_{1}^{+},\\
\partial_{n} u &= 0 \text{ in }T.
\end{array}
\right.
\end{equation}
 Then, letting 
\begin{equation}
\tilde{u}(x', x^{n}) = \left\{
\begin{array}{cl}
u(x', x^{n}) & \ x^{n} > 0,\\
u(x', -x^{n}) & \ x^{n} < 0,
\end{array}
\right.
\end{equation}
we have that $\tilde{u}$ lies in $W^{1, 2}_{\text{loc}}(B_{1}; \DD)$ and solves 
\begin{equation}\label{wholeGL}
-\Delta_{g} \tilde{u} = \ep^{-1}(1 - |\tilde{u}|^{2})\tilde{u} \text{ on }B_{1}.
\end{equation}
\end{lemm}
\begin{proof}
The first assertion is standard. To prove that $\tilde{u}$ solves \eqref{wholeGL} we take $\zeta \in C^{1}_{c}(B_{1})$ and use the symmetry of $\tilde{u}$ and $g$ to compute
\begin{align*}
&\int_{B_{1}} \langle \nabla\tilde{u}, \nabla\zeta \rangle + \ep^{-2}(|\tilde{u}|^{2}-1)\tilde{u}\cdot \zeta \dvol_{g}\\
&= \int_{B_{1}^{+}} \langle \nabla u, \nabla\left( \zeta + \tau^{\ast}\zeta \right) \rangle + \ep^{-2}(|u|^{2} - 1)u\cdot (\zeta + \tau^{\ast}\zeta)  \dvol_{g} = 0.
\end{align*}
The second inequality follows because $u$ solves \eqref{halfGL}.
\end{proof}
\section{Energy monotonicity formulae}


In this section, we show that the energy of solutions to \eqref{GLequation} possesses certain monotonicity properties which will be essential in all the subsequent sections of this paper. For convenience, we formulate the results in local terms.  Thus, we will fix $\lambda, \Lambda > 0$ and work with solutions to \eqref{GLequation} on $(B_{1}, g)$, where $g\in \cM_{\lambda, \Lambda}$. Note that we allow for Lipschitz metrics here so that the results can be applied to boundary points after reflection (cf. Remark \ref{singularmetric}). Finally, the idea of the proofs given below has been widely used in many other contexts (e.g. \cite{bbo}, \cite{chiron}, \cite{bpw}, etc.) and we do not claim originality over it.

\begin{prop}[cf. Proposition \ref{globalmono} below]\label{lipschitzmonotonicity}
There exist constants $\chi$ and $C$, depending only on $n, \lambda$ and $\Lambda$ such that if $g \in \cM_{\lambda, \Lambda}$ and $u: (B_{1}, g)\to \CC$ is a solution to \eqref{GLequation}, then the following hold for $\rho \in (0, 1)$.
\begin{align}\label{lipschitzmonotonicity1}
& \paop{\rho}\left( e^{\chi\rho}\rho^{2-n}\int_{B_{\rho}(x)}e_{\ep}(u) \dvol \right)\\
\nonumber \geq &\ \frac{1}{1 + C\rho}\left[  e^{\chi\rho}\rho^{1-n}\int_{B_{\rho}} \frac{(1 - |u|^{2})^{2}}{2\ep^{2}} \dvol + e^{\chi\rho}\rho^{2-n}\int_{\partial B_{\rho}}|\nabla_{r}u|^{2} \sqrt{\det(g)}dH^{n-1} \right].
\end{align}
\begin{align}\label{lipschitzmonotonicity2}
& \paop{\rho}\left( e^{\chi\rho}\rho^{2-n}\int_{B_{\rho}(x)}\frac{|\nabla u|^{2}}{2} + \left( 1 + \frac{2}{(1 + C)(n-2)} \right)\frac{(1 - |u|^{2})^{2}}{4\ep^{2}} \dvol \right)\\
\nonumber \geq &\ \frac{1}{1 + C} e^{\chi\rho}\rho^{2-n}\int_{\partial B_{\rho}} \frac{1}{(n-2)}\frac{(1 - |u|^{2})^{2}}{2\ep^{2}} + |\nabla_{r}u|^{2} \sqrt{\det(g)}dH^{n-1} .
\end{align}
where $B_{\rho}$ denotes the Euclidean ball of radius $\rho$, and $H^{n-1}$ is the $(n-1)$-dimensional Hausdorff measure.
\end{prop}
\begin{proof}
Throughout this proof, $C, C_{1}, C_{2}, \cdots$ denote any constants depending only on $n, \lambda, \Lambda$, and $O(r^{a})$ denotes any term bounded in norm by $Cr^{a}$.

Taking a compactly supported $C^{1}$-vector field $X \in C^{1}_{0}(B_{1}; \RR^{n})$ and substituting $\zeta = \langle X, \nabla u \rangle$ into the first variation formula \eqref{GLvariation}, we get, after integrating by parts, 
\begin{equation}\label{innervariation}
\int_{B_{1}} \left[ e_{\ep}(u)\Div X - \langle  \nabla_{\nabla u}X, \nabla u \rangle \right] \dvol = 0.
\end{equation}
Next we let $r(x) = (x_{1}^{2} + \cdots + x_{n}^{2})^{1/2}$ and choose $X$ to be $\xi(r)r\nabla r$, where $\xi(r) = \zeta(r/\rho)$ and $\zeta$ is a decreasing cut-off function with $\zeta(t) = 1$ for $t \leq 1/2$ and $\zeta(t) = 0$ for $t \geq 1$. Then, by a direct computation we have
\begin{equation}\label{nablaX}
\nabla X = (r\xi' + \xi)\nabla r \otimes \nabla r + \xi r\nabla^{2}r;\ \Div X = (r\xi' + \xi)|\nabla r|^{2} + \xi r \Delta r.
\end{equation}
(The two identities above are only guaranteed to hold almost everywhere on $B_{1}$ since $g$ is only Lipschitz. Nonetheless, this does not affect the subsequent arguments.) Plugging \eqref{nablaX} into \eqref{innervariation} yields
\begin{align}\label{radialvariationlipschitz}
\int_{B_{1}}e_{\ep}(u)\left( (r\xi' + \xi)|\nabla r|^{2} \right.&\left.+ \xi r\Delta r \right)\\
\nonumber& -\left[ (r\xi' + \xi)|\nabla_{r}u|^{2} + \xi r\nabla^{2}r(\nabla u, \nabla u) \right] \dvol = 0.
\end{align}
We next compute the terms $|\nabla r|^{2}$, $r \Delta r$ and $r\nabla r^{2}(\nabla u, \nabla u)$. Recalling the assumptions on $g$, we have
\begin{equation}\label{nablar}
|\nabla r|^{2} = g^{ij}\partial_{i}r \partial_{j}r = 1 + (g^{ij} - \delta^{ij})\partial_{i}r\partial_{j}r = 1 + O(r).
\end{equation}
\begin{equation}\label{hessianr}
r\nabla^{2}_{i. j}r = r\partial_{ij}r + r\Gamma_{ij}^{k}\partial_{k}r = \left( \delta_{ij} - \frac{x_{i}x_{j}}{r^{2}} \right) + O(r).
\end{equation}
From \eqref{hessianr} we get
\begin{equation}\label{hessianrnablau}
r\nabla^{2}r(\nabla u, \nabla u) = |\nabla u|^{2} - |\nabla_{r}u|^{2} + O(r)|\nabla u|^{2},
\end{equation}
and that
\begin{equation}\label{laplacer}
r\Delta r = rg^{ij}\nabla^{2}_{i, j}r = n - |\nabla r|^{2} + O(r) = (n-1) + O(r).
\end{equation}
Plugging \eqref{nablar}, \eqref{hessianrnablau} and \eqref{laplacer} back into \eqref{radialvariationlipschitz} and simplifying, we get
\begin{align}\label{pohozaevwithO}
0 = &\int_{B_{1}}(1 + O(r))r\xi' e_{\ep}(u)\dvol + \int_{B_{1}}\xi\left( (n-2)e_{\ep}(u) + \frac{(1 - |u|^{2})^{2}}{2\ep^{2}} \right)\dvol-  \\
\nonumber & - \int_{B_{1}}r\xi' |\nabla_{r}u|^{2} \dvol + \int_{B_{1}}O(r)\xi e_{\ep}(u)\dvol.
\end{align}
Recalling the definition of $O(r^{a})$ and the fact that $\xi' \leq 0$, we get
\begin{align}\label{Oremoved}
&-\int_{B_{1}}r\xi' |\nabla_{r}u|^{2}\dvol + \int_{B_{1}}\xi \frac{(1 - |u|^{2})^{2}}{2\ep^{2}} \dvol \\
\nonumber \leq \ &(2-n) \int_{B_{1}}\xi e_{\ep}(u)\dvol - \int_{B_{1}} \left( 1 + C_{1}r \right)r\xi' e_{\ep}(u) \dvol + \int_{B_{1}}C_{2}r\xi e_{\ep}(u)\dvol.
\end{align}
Letting $\zeta$ increase to the characteristic function of $B_{1}$, we arrive at
\begin{align}
&\rho\int_{\partial B_{\rho}}|\nabla_{r}u|^{2}\sqrt{\det(g)}dH^{n-1} + \int_{B_{\rho}}\frac{(1 - |u|^{2})^{2}}{2\ep^{2}}\dvol\\
\nonumber  \leq&\  (2-n)\int_{B_{\rho}}e_{\ep}(u) + (1 + C_{1}\rho)\rho\int_{\partial B_{\rho}}e_{\ep}(u)\sqrt{\det(g)}dH^{n-1} + C_{2}\rho \int_{B_{\rho}} e_{\ep}(u)\dvol \\
\nonumber  =&\left( 1 + C_{1}\rho \right) \left[ (2-n)\int_{B_{\rho}}e_{\ep}(u)\dvol \right.\\
\nonumber &\left.+\ \rho\int_{\partial B_{\rho}}e_{\ep}(u)\sqrt{\det(g)}dH^{n-1} + \left( \frac{(n-2)C_{1} + C_{2}}{1 + C_{1}\rho}\rho \int_{B_{\rho}}e_{\ep}(u)\dvol \right)  \right].
\end{align}
Now we set $\chi = (n-2)C_{1} + C_{2}$ and multiply both ends of the above inequality by $(1 + C_{1}\rho)^{-1}e^{\chi\rho}\rho^{1-n}$ to get
\begin{align}
&\frac{1}{1 + C_{1}\rho}\left[ e^{\chi\rho}\rho^{1-n}\int_{B_{\rho}}\frac{(1 - |u|^{2})^{2}}{2\ep^{2}}\dvol +  e^{\chi\rho}\rho^{2-n}\int_{\partial B_{\rho}} |\nabla_{r}u|^{2}\sqrt{\det(g)}dH^{n-1} \right]\\
\nonumber \leq &\ e^{\chi\rho}(2-n)\rho^{1-n}\int_{B_{\rho}}e_{\ep}(u)\dvol\\
\nonumber & + e^{\chi\rho}\rho^{2-n}\int_{\partial B_{\rho}}e_{\ep}(u)\sqrt{\det(g)}dH^{n-1} + \chi e^{\chi\rho}\rho^{2-n}\int_{B_{\rho}}e_{\ep}(u)\dvol\\
\nonumber =&\  \paop{\rho}\left( e^{\chi\rho}\rho^{2-n}\int_{B_{\rho}}e_{\ep}(u)\dvol \right).
\end{align}
Thus we get \eqref{lipschitzmonotonicity1}. The inequality \eqref{lipschitzmonotonicity2} follows from adding the following term to both sides of \eqref{lipschitzmonotonicity1}.
\begin{equation}
\paop{\rho}\left( \frac{1}{(1 + C_{1})(n-2)}e^{\chi\rho}\rho^{2-n}\int_{B_{\rho}}\frac{(1 - |u|^{2})^{2}}{2\ep^{2}}\dvol \right).
\end{equation}
\end{proof}

Proposition \ref{lipschitzmonotonicity} controls the Ginzburg-Landau energy at all smaller scales in terms of the energy at a fixed scale. We will also need a formula that has the opposite effect. 
\begin{prop}\label{reversemonotonicity}
Under the assumptions of Proposition \ref{lipschitzmonotonicity}, there exists constants $\chi', C$ and $\rho_{0}$, all depending only on $n, \lambda$ and $\Lambda$, such that for $\rho < \rho_{0}$ there holds
\begin{align}\label{reversemonotonicity1}
& \paop{\rho}\left( e^{-\chi'\rho}\rho^{2-n}\int_{B_{\rho}(x)}\frac{|\nabla u|^{2}}{2} + \frac{(1 - |u|^{2})^{2}}{4\ep^{2}} \dvol \right)\\
\nonumber \leq &\frac{1}{1 - C\rho}\left[  e^{-\chi'\rho}\rho^{1-n}\int_{B_{\rho}} \frac{(1 - |u|^{2})^{2}}{2\ep^{2}} \dvol + e^{-\chi'\rho}\rho^{2-n}\int_{\partial B_{\rho}}|\nabla_{r}u|^{2} \sqrt{\det(g)}dH^{n-1} \right].
\end{align}
\end{prop}
\begin{proof}
In the proof of Proposition \ref{lipschitzmonotonicity}, since the definition of $O(r^{a})$ involves a two-sided bound, we can use \eqref{pohozaevwithO} to derive the following inequality.
\begin{align}
&-\int_{B_{1}}r\xi' |\nabla_{r}u|^{2}\dvol + \int_{B_{1}}\xi \frac{(1 - |u|^{2})^{2}}{2\ep^{2}} \dvol \\
\nonumber \geq \ &(2-n) \int_{B_{1}}\xi e_{\ep}(u)\dvol - \int_{B_{1}} \left( 1 - C_{1}r \right)r\xi' e_{\ep}(u) \dvol - \int_{B_{1}}C_{2}r\xi e_{\ep}(u)\dvol.
\end{align}
Again let $\zeta$ increase to the characteristic function of $B_{1}$, and set $\rho_{0} = 1/2C_{1}$, we get, for $\rho < \rho_{0}$, that
\begin{align}
&\rho\int_{\partial B_{\rho}}|\nabla_{r}u|^{2}\sqrt{\det(g)}dH^{n-1} + \int_{B_{\rho}}\frac{(1 - |u|^{2})^{2}}{2\ep^{2}}\dvol\\
\nonumber  \geq&\  (2-n)\int_{B_{\rho}}e_{\ep}(u) + (1 - C_{1}\rho)\rho\int_{\partial B_{\rho}}e_{\ep}(u)\sqrt{\det(g)}dH^{n-1} - C_{2}\rho \int_{B_{\rho}} e_{\ep}(u)\dvol \\
\nonumber  = &\left( 1 - C_{1}\rho \right) \left[ (2-n)\int_{B_{\rho}}e_{\ep}(u)\dvol \right.\\
\nonumber &\left.+\ \rho\int_{\partial B_{\rho}}e_{\ep}(u)\sqrt{\det(g)}dH^{n-1} - \left( \frac{(n-2)C_{1} + C_{2}}{1 - C_{1}\rho}\rho \int_{B_{\rho}}e_{\ep}(u)\dvol \right)  \right].
\end{align}
We next set $\chi' = 2[(n-2)C_{1} + C_{2}]$ and multiply both ends of the above inequality by $(1 - C_{1}\rho)^{-1}e^{-\chi' \rho}\rho^{1-n}$ to obtain, for $\rho < \rho_{0}$, 
\begin{align}
& \frac{1}{1 - C_{1}\rho}\left[ e^{-\chi'\rho}\rho^{1-n}\int_{B_{\rho}}\frac{(1 - |u|^{2})^{2}}{2\ep^{2}}\dvol + e^{-\chi' \rho}\rho^{2-n}\int_{B_{\rho}}|\nabla_{r}u|^{2}\sqrt{\det(g)}dH^{n-1} \right]\\
\nonumber\geq\ & \paop{\rho}\left( e^{-\chi' \rho}\rho^{2-n}\int_{B_{\rho}}e_{\ep}(u)\dvol \right).
\end{align}
\end{proof}

For later use, we record below two corollaries of Proposition \ref{lipschitzmonotonicity} which apply to balls not centered at the origin.
\begin{coro}\label{monotonicitydecentralized}
Let $u$ and $g$ be as in Proposition \ref{lipschitzmonotonicity}. Then for all $x_{0} \in B_{1}$ and $\sigma \leq \rho < 1 - |x_{0}|$, the following inequalities hold.
\begin{equation}\label{monovarying}
\sigma^{2-n}\int_{B_{\sigma}(x_{0})}e_{\ep}(u)\dvol \leq C\rho^{2-n}\int_{B_{\rho}(x_{0})}e_{\ep}(u)\dvol.
\end{equation}
\begin{equation}\label{potentialdensitybound}
\int_{B_{\rho}(x_{0})} |x - x_{0}|^{2-n} \frac{(1 - |u(x)|^{2})^{2}}{\ep^{2}} \dvol \leq C\left( 1 - |x_{0}| \right)^{2-n}\int_{B_{1}}e_{\ep}(u)\dvol.
\end{equation}
In both inequalities, the constants $C$ on the right-hand side depend only on $n, \lambda, \Lambda$.
\end{coro}
\begin{rmk}
In particular, it follows from \eqref{monovarying} that for all $x_{0} \in B_{3/4}$ and $\rho < 1/4$, there holds
\begin{equation}\label{energydensitybound}
\rho^{2-n}\int_{B_{\rho}(x_{0})}e_{\ep}(u)\dvol \leq C\cdot 4^{n-2}\int_{B_{1/4}(x_{0})} e_{\ep}(u)\dvol \leq C\int_{B_{1}}e_{\ep}(u)\dvol.
\end{equation}
\end{rmk}
\begin{proof}[Proof of Corollary \ref{monotonicitydecentralized}]
It suffices to establish \eqref{monovarying} in the case $\lambda^{-1} \sigma \leq \rho < \left( 1 - |x_{0}| \right)$ and \eqref{potentialdensitybound} in the case $\rho < \lambda\left( 1 - |x_{0}| \right)$, since otherwise the conclusions are easily seen to hold.

If $x_{0} = 0$, then \eqref{monovarying} and \eqref{potentialdensitybound} follow from integrating \eqref{lipschitzmonotonicity1} and \eqref{lipschitzmonotonicity2}, respectively. If $x_{0} \neq 0$, we can write $(g_{ij}(x_{0})) = (A^{-1})^{t}A^{-1}$ for some $A$ invertible and consider the coordinate transform $x = \psi(y) \equiv Ay + x_{0}$. Then since $g \in \cM_{\lambda, \Lambda}$ we easily see that
\begin{equation}
B_{\sqrt{\lambda}\rho}(x_{0}) \subset \psi(B_{\rho}) \subset B_{\sqrt{\lambda}^{-1}\rho}(x_{0}) \subset B_{1},
\end{equation}
for all $\rho < \sqrt{\lambda}\rho_{0}$, where $\rho_{0} = 1 - |x_{0}|$. Moreover, letting $\tilde{u}(y) = u(\psi(y))$ and $\tilde{g} = \psi^{\ast}g$, then clearly $\tilde{u}$ solves \eqref{GLequation} on $B_{\sqrt{\lambda}\rho_{0}}$ with respect to $\tilde{g}$. Moreover, since 
\begin{equation}\label{grelation}
(\tilde{g}_{ij}(y)) = A^{t} \left(g_{ij}(\psi(y)) \right)A,
\end{equation}
we see that $\tilde{g}_{ij}(0) = \delta_{ij}$, and that 
\begin{enumerate}
\item[(1)] $[\tilde{g}_{ij}]_{0, 1; B_{\sqrt{\lambda}\rho_{0}}} \leq C_{n, \lambda, \Lambda}$.
\item[(2)] $\lambda^{2}|\xi|^{2} \leq \tilde{g}_{ij}(y)\xi^{i}\xi^{j} \leq \lambda^{-2}|\xi|^{2}$, for all $y \in B_{\sqrt{\lambda}\rho_{0}}$ and $\xi \in \RR^{n}$.
\end{enumerate}
The fact that $\tilde{u}, \tilde{g}$ are not defined on the unit ball certainly does not affect the arguments in the proof of Proposition \ref{lipschitzmonotonicity}, and we infer that \eqref{lipschitzmonotonicity1} holds for $\tilde{u}$ and $\tilde{g}$ for $\rho \in (0, \sqrt{\lambda}\rho_{0})$ and for some constants $\chi$ and $C$ depending only on $n, \lambda$ and $\Lambda$. Integrating \eqref{lipschitzmonotonicity1} we obtain, for $\sigma \leq \rho < \sqrt{\lambda}\rho_{0}$
\begin{align*}
&e^{\chi' \sigma} \sigma^{2-n} \int_{B_{\sigma}} e_{\ep}(\tilde{u}) \dvol_{\tilde{g}} \leq e^{\chi' \rho} \rho^{2-n} \int_{B_{\rho}} e_{\ep}(\tilde{u}) \dvol_{\tilde{g}}\\
\Longrightarrow\ & e^{\chi' \sigma} \sigma^{2-n} \int_{\psi(B_{\sigma})} e_{\ep}(u) \dvol_{g} \leq e^{\chi' \rho} \rho^{2-n} \int_{B_{\rho}} e_{\ep}(u) \dvol_{g}\\
\Longrightarrow\ & e^{\chi' \sigma} \sigma^{2-n} \int_{B_{\sqrt{\lambda}\sigma}} e_{\ep}(u) \dvol_{g} \leq e^{\chi' \rho} \rho^{2-n} \int_{B_{\sqrt{\lambda}^{-1}\rho}} e_{\ep}(u) \dvol_{g}.
\end{align*}
The last inequality holds for all $\sigma \leq \rho < \sqrt{\lambda}\rho_{0}$, and thus we get \eqref{monovarying} whenever $\lambda^{-1}\sigma \leq \rho < \rho_{0}$.

The inequality \eqref{potentialdensitybound} is proved similarly, except that, instead of \eqref{lipschitzmonotonicity1}, we integrate \eqref{lipschitzmonotonicity2} applied to $\tilde{u}$ and $\tilde{g}$ and then undo the affine transformation. The details will be omitted.

\end{proof}

A second corollary of Proposition \ref{lipschitzmonotonicity} that we mention here is a Courant-Lebesgue type result. We will omit its proof since one can follow exactly the arguments of Proposition II.2 in \cite{bbo}. Note that the inequality (II.6) used in the proof there would now follow from \eqref{lipschitzmonotonicity1} instead of Lemma II.3 in \cite{bbo}. 
\begin{prop}[\cite{bbo}, Proposition II.2]
\label{courantlebesgue}
Under the assumptions of Proposition \ref{lipschitzmonotonicity}, there exists a radius $r \in (\ep^{1/2}, \ep^{1/4})$ such that
\begin{align}
&r^{3-n}\int_{\partial B_{r}}|\nabla_{\nu}u|^{2}\sqrt{\det(g)}dH^{n-1} + r^{2-n}\int_{B_{r}} \frac{(1 - |u|^{2})}{2\ep^{2}}\dvol_{g}\\
\nonumber \leq\ & \frac{C}{\left| \log\ep \right|}(\ep^{1/4})^{2-n}\int_{B_{\ep^{1/4}}}e_{\ep}(u)\dvol_{g} \leq \frac{C}{\left| \log\ep \right|}\int_{B_{1}}e_{\ep}(u)\dvol_{g},
\end{align}
where the constant $C$ depends on the same parameters as in Proposition \ref{lipschitzmonotonicity}.
\end{prop}

\section{The $\eta$-ellipticity theorem}


This section is devoted to the proof of the $\eta$-ellipticity theorem at boundary points, but for completeness we also state the interior version. The theorem roughly says that the modulus $|u|$ of a solution to the Ginzburg-Landau equation is close to $1$ at places where the scale invariant Ginzburg-Landau energy is small compared to $|\log\ep|$. The precise statements are given as follows. 

\begin{thm}[Interior version]
\label{interioreta}
There exist constants $\eta_{0}, \ep_{0}, s$ and $C$, depending on $n, \mu, \lambda$ and $\Lambda$, such that if $g \in \cM_{\mu, \lambda, \Lambda}$ and $u: (B_{1}, g) \to \DD$ solves \eqref{GLequation} with $ \ep < \ep_{0}$ and
\begin{equation}\label{smallenergyscaled}
\int_{B_{1}} e_{\ep}(u) \dvol \leq \eta \left| \log\ep \right|,
\end{equation}
for some $\eta < \eta_{0}$, then there exists a radius $r_{1} \in (\sqrt{\ep}/2, 1/4)$ such that
\begin{equation}
|u(x)| \geq 1 - C\eta^{s}, \text{ for all } x \in B_{r_{1}}.
\end{equation}
\end{thm}

\begin{thm}[Boundary version]
\label{boundaryeta}
There exist constants $\eta_{0}', \ep_{0}', s'$ and $C'$, depending on $n, \mu, \lambda$ and $\Lambda$, such that if $g \in \cM^{+}_{\mu, \lambda, \Lambda}$ and $u: (B_{1}^{+}, g) \to \DD$ solves \eqref{halfGL} with $ \ep < \ep_{0}'$ and
\begin{equation}\label{smallenergyscaledbdy}
\int_{B_{1}^{+}} e_{\ep}(u) \dvol \leq \eta \left| \log\ep \right|,
\end{equation}
for some $\eta < \eta_{0}'$, then there exists a radius $r_{1} \in (\sqrt{\ep}/2, 1/4)$ such that
\begin{equation}\label{etaellipticityconclusion}
|u(x)| \geq 1 - C'\eta^{s'}, \text{ for all } x \in B^{+}_{r_{1}}.
\end{equation}
\end{thm}

The interior version of the $\eta$-ellipticity theorem for Euclidean metrics appears as Theorem 2 of \cite{bbo}. For the case of non-flat metrics, see \cite{stern2}. On the other hand, the boundary version for Euclidean metrics was previously established in \cite{bbo} for Dirichlet boundary conditions, in \cite{chiron} for the homogeneous Neumann condition, and in \cite{bpw} for the weak anchoring condition. We note that when proving the boundary version, both \cite{bbo} and \cite{chiron} made the simplifying assumption that after locally flattening the boundary, the resulting metric on the upper half plane is flat. Here we do not make this assumption.

The structure of the proof of Theorem \ref{boundaryeta} closely follows that of Theorem 3 in \cite{bbo}, the main difference being the estimate of $\varphi_{1}$ in Step 1 of the proof of Proposition \ref{energydecay} below, for which we need an upper bound on the Green matrix with non-smooth coefficients. This estimate is derived in Appendix B. 

\subsection{Energy decay estimate}
The main result of this section is the following.
\begin{prop}
\label{energydecay}
Let $g \in \cM^{+}_{\mu, \lambda, \Lambda}$ and suppose $u: (B_{1}^{+}, g) \to \DD$ is a solution to \eqref{halfGL}. Then there exists $\delta_{0} < 1/4$, depending on $n, \mu, \lambda$, and $\Lambda$, such that for $\delta < \delta_{0}$ there holds
\begin{align}\label{energydecayestimate}
\int_{B_{\delta}^{+}}e_{\ep}(u) \dvol_{g} \leq C &\left[ \left( \int_{B_{1}^{+}}\frac{(1 - |u|^{2})^{2}}{\ep^{2}} \dvol_{g} \right)^{1/3}\left( \int_{B^{+}_{1}}e_{\ep}(u) \dvol_{g} \right) \right.\\
\nonumber &+ \left.\left( \int_{B^{+}_{1}}\frac{(1 - |u|^{2})^{2}}{\ep^{2}}\dvol_{g} \right)^{2/3} + \delta^{n}\int_{B^{+}_{1}}e_{\ep}(u)\dvol_{g} \right],
\end{align}
where $C$ depends only on $n, \mu, \lambda$ and $\Lambda$. 
\end{prop}
\begin{proof}
We begin by treating the gradient term in $e_{\ep}(u)$, which can be decomposed as follows.
\begin{equation}\label{gradientdecomposition}
|\nabla u|^{2} = \left( 1 - |u|^{2} \right) |\nabla u|^{2} + \frac{1}{4}\left| \nabla |u|^{2} \right|^{2} + |u \times \nabla u|^{2},
\end{equation}
where $u \times \nabla u \equiv u^{1}\nabla u^{2} - u^{2}\nabla u^{1}$ in terms of the components of $u = (u^{1}, u^{2})$. We will estimate these terms one by one, starting with the last term.
\vskip 2mm
\noindent\textbf{Step 1: Decomposition of $u \times \nabla u$.}
\vskip 1mm
We first use Lemma \ref{GLreflection} to extend $u$ to a solution on all of $B_{1}$ by even reflection across $T$. The reflected function will still be denoted by $u$. Next, for some $R < 3/4$ to be fixed later depending only on $n, \lambda, \Lambda$ and $\mu$, we choose a radius $r_{1} \in [R/2, R]$ satisfying
\begin{align}\label{goodshell}
\int_{S_{r_{1}}}|\nabla u|^{2}\sqrt{\det(g)}dH^{n-1} &\leq C \int_{B_{1}}|\nabla u|^{2}\dvol,\\
\nonumber\int_{S_{r_{1}}}\left( 1 - |u|^{2} \right)^{2} \sqrt{\det(g)}dH^{n-1} &\leq C\int_{B_{1}}\left( 1 - |u|^{2} \right)^{2}\dvol.
\end{align}
Note that since $R$ will later be chosen to depend only on predetermined parameters, we will not keep track of the $R$-dependences of the constants $C$. We now want to perform a Hodge decomposition on the one-form $u \times du$ restricted to $B_{r_{1}}$. To that end, we let $\xi$ be the solution to 
\begin{equation}\label{xiBVP}
\left\{
\begin{array}{crl}
\Delta \xi &= & 0 \text{ on $B_{r_{1}}$},\\
\nabla_{\nu}\xi & = & u \times \nabla_{\nu}u \text{ on $\partial B_{r_{1}}$},\\
\int_{B_{r_{1}}} \xi \dvol &= & 0.
\end{array}
\right.
\end{equation}
We see immediately by standard elliptic theory that
\begin{equation}\label{xiw12}
|\nabla \xi|_{0; B_{R/4}} \leq C\|\nabla \xi\|_{2; B_{r_{1}}}\leq C\| \nabla u \|_{2; S_{r_{1}}} \leq C\| \nabla u \|_{2; B_{1}},
\end{equation}
where we used \eqref{goodshell} in the last inequality. Consequently, for $\delta < R/4$,
\begin{equation}\label{xiestimate}
\int_{B_{\delta}} |\nabla \xi|^{2}\dvol \leq C\delta^{n} \int_{B_{1}}|\nabla u|^{2}\dvol.
\end{equation}
Next, consider the one-form $\chi_{B_{r_{1}}} (u\times du - d\xi)$ and compute
\begin{align*}
\int_{B_{r_{1}}} \langle u \times du - d\xi, d\zeta \rangle \dvol = &\int_{\partial B_{r_{1}}} \zeta (u \times \nabla_{\nu} u - \nabla_{\nu}\xi) \sqrt{\det(g)}dH^{n-1}\\
 &+ \int_{B_{r_{1}}} \zeta d^{\ast}(u \times du - d\xi) \dvol = 0,
\end{align*}
for all $\zeta \in W^{1, 2}_{0}(B_{1}; \RR)$. Therefore, by Theorem 7.7.7 of \cite{morreybook} and inequality \eqref{hodgecoercive}, there exists a two-form $\varphi \in W^{1, 2}_{2, \ft}(B_{1})$ such that
\begin{align}\label{phibasic}
d^{\ast}\varphi = (u \times du - d\xi)\chi_{B_{r_{1}}}&;\ d\varphi = 0.\\
\nonumber\| \varphi\|_{1, 2; B_{1}} \leq C (\|du\|_{2; B_{r_{1}}} + \|d\xi\|_{2; B_{r_{1}}}) &\leq C\|\nabla u\|_{2; B_{1}}.
\end{align}
Therefore, on $B_{r_{1}}$ we have
\begin{equation}\label{hodgedecomposition}
u \times du = d^{\ast}\varphi + d\xi.
\end{equation}
We also note that $\varphi$ belongs to the class $W^{1, 2, +}_{2, \ft}(B_{1})$, since both $u$ and $\xi$ are even functions in $x^{n}$ (cf. Remark \ref{symmetry}).

\vskip 2mm
\noindent\textbf{Step 2: Estimates for $d^{\ast}\varphi$.}
\vskip 1mm
To derive estimates on $\varphi$, we follow \cite{bbo} and define the cut-off function
\begin{equation*}
\alpha(x) =
\left\{
\begin{array}{cl}
f(|u(x)|)^{2} & \text{ in }B_{r_{1}}\\
1 & \text{ elsewhere.}
\end{array}
\right.,
\end{equation*}
where the function $f$ satisfies $|f'| \leq 4$ and is given as follows for some constant $\beta < 1/4$ to be determined later.
\begin{equation*}
f(t) = 
\left\{
\begin{array}{cl}
\frac{1}{t} & \text{ if }t\geq 1-\beta\\
1 & \text{ if }t \leq 1-2\beta.\\
\end{array}
\right.
\end{equation*}
Later in the proof, we will need to use the simple observation that $1 - \alpha \leq 4\beta$. Note that on $B_{1}$, we certainly have
\begin{equation*}
\Delta \varphi = d((1 - \alpha)d^{\ast}\varphi) + d(\alpha d^{\ast}\varphi).
\end{equation*}
Moreover, the second term on the right can be computed as in \cite{bbo}. Specifically, for any two-form $\zeta \in W^{1, 2}_{2, \ft}(B_{1})$, we have
\begin{equation}\label{laplaciancomputation}
\int_{B_{1}}\langle \alpha d^{\ast}\varphi, d^{\ast}\zeta \rangle= \int_{B_{r_{1}}}\langle \alpha u\times du, d^{\ast} \zeta \rangle - \int_{B_{r_{1}}} \langle \alpha d\xi, d^{\ast} \zeta \rangle  = (I) + (II)
\end{equation}
For $(I)$, we have
\begin{equation*}
(I)  = \int_{\partial B_{r_{1}}} \ft\alpha u \times du  \wedge \ft\ast\zeta + \int_{B_{r_{1}}}\langle d(\alpha u \times du), \zeta \rangle ,
\end{equation*}
While for $(II)$, we have
\begin{align*}
(II) &= \int_{B_{r_{1}}} \langle (1 - \alpha)d\xi, d^{\ast}\zeta \rangle - \int_{B_{r_{1}}}\langle d\xi, d^{\ast}\zeta \rangle\\
&= \int_{B_{r_{1}}} \langle (1 - \alpha)d\xi, d^{\ast}\zeta \rangle - \int_{\partial B_{r_{1}}}\ft d\xi \wedge \ft \ast \zeta.
\end{align*}
Therefore, putting everything together, $\varphi$ satisfies
\begin{align}\label{phiequation}
\int_{B_{1}}\langle d^{\ast}\varphi, d^{\ast}\zeta \rangle =& \int_{B_{r_{1}}}\langle d(\alpha u \times du), \zeta \rangle +  \int_{\partial B_{r_{1}}} \ft\alpha u \times du  \wedge \ft\ast\zeta\\
\nonumber & +\int_{B_{r_{1}}} \langle (1 - \alpha)d\xi, d^{\ast}\zeta \rangle - \int_{B_{r_{1}}}\langle d\xi, d^{\ast}\zeta \rangle
+ \int_{B_{1}}\langle (1 - \alpha)d^{\ast}\varphi, d^{\ast}\zeta \rangle\\
\nonumber \equiv&\ w_{1}(\zeta) + w_{2}(\zeta)\\
\nonumber & + w_{3}(\zeta) + w_{4}(\zeta) + w_{5}(\zeta),
\end{align}
for all $\zeta \in W^{1, 2}_{2, \ft}(B_{1})$, where the five $w_{i}$'s are bounded linear functionals on $W^{1, 2}_{2, \ft}(B_{1})$ defined by the five terms in the two lines above them, respectively. Given \eqref{phiequation}, we use Lemma \ref{hodgeinvertible} to decompose $\varphi$ into $\varphi_{1} + \cdots \varphi_{5}$, with each $\varphi_{i}$ being the unique solution in $W^{1, 2}_{2, \ft}(B_{1})$ to 
\begin{equation}\label{phiiequation}
\cD(\varphi_{i}, \zeta) = w_{i}(\zeta), \text{ for all }\zeta \in W^{1, 2}_{2, \ft}(B_{1}).
\end{equation}
Note also that $\varphi_{i} \in W^{1, 2, +}_{2, \ft}(B_{1})$ for $i = 1, \cdots, 5$. Below we estimate these $2$-forms one by one, with $\varphi_{1}$ requiring the most work. On the other hand, $\varphi_{2}$ to $\varphi_{5}$ are estimated in essentially the same way as in \cite{bbo} or \cite{bpw}. Nonetheless, we indicate the arguments for the reader's convenience.

We first handle $\varphi_{3}$ by taking itself as a test form in \eqref{phiiequation} to get
\begin{align*}
\cD(\varphi_{3}, \varphi_{3}) &= w_{3}(\varphi_{3}) = \int_{B_{r_{1}}} \langle (1 - \alpha)d\xi, d^{\ast}\varphi_{3} \rangle\\
&\leq \|(1 - \alpha) \nabla \xi\|_{2; B_{r_{1}}} \| d^{\ast}\varphi_{3} \|_{2; B_{1}} \leq C\beta \|\nabla u\|_{2; B_{1}}\| d^{\ast}\varphi_{3} \|_{2; B_{1}},
\end{align*}
where in the last inequality we used the estimate \eqref{xiw12} and the fact that $|1 - \alpha| \leq 4\beta$. Hence, since $\cD$ controls the $W^{1, 2}$-norm, we get 
\begin{equation}\label{phi3estimate}
\| \varphi_{3} \|_{1, 2; B_{1}} \leq C\beta \|\nabla u\|_{2; B_{1}}.
\end{equation}
A similar argument using the $W^{1, 2}$-estimate for $\varphi$ in \eqref{phibasic} in place of \eqref{xiw12} shows that, 
\begin{equation}\label{phi5estimate}
\| \varphi_{5}\|_{1, 2; B_{1}} \leq C\beta \| \nabla u \|_{2; B_{1}},
\end{equation}

Next, for $\varphi_{2}$, we notice that there holds
\begin{align*}
\cD(\varphi_{2}, \varphi_{2}) &= w_{2}(\varphi_{2}) = \int_{\partial B_{r_{1}}} \ft \alpha u\times du \wedge \ft \ast \varphi_{2}\\
&\leq \| \nabla u \|_{2; \partial B_{r_{1}}} \|\varphi_{2}\|_{2; \partial B_{r_{1}}} \leq C\| \nabla u\|_{2; B_{1}} \| \varphi_{2} \|_{1, 2; B_{r_{1}}}.
\end{align*}
Note that the last inequality follows from our choice of $r_{1}$ and the trace inequality. We therefore obtain 
\begin{equation}\label{phi2w12}
\| \varphi_{2}\|_{1, 2; B_{1}} \leq C\|\nabla u\|_{2; B_{1}}.
\end{equation}
To proceed, we note that since $w_{2}$ is defined by a boundary integral on $\partial B_{r_{1}}$, we have, in particular,
\begin{equation*}
\cD_{g}(\varphi_{2}, \zeta) = 0,\text{ for all }\zeta \in W^{1, 2, +}_{2, 0}(B_{r_{1}}).
\end{equation*}
Since we are assuming that $g \in \cM_{\mu, \lambda, \Lambda}^{+}$, we infer with the help of Lemma \ref{ellipticform} and Remark \ref{fitfundamentalsolution} that $\varphi_{2} \in C^{1, \mu}_{\text{loc}}(B_{r_{1}}^{+} \cup T_{r_{1}})$, with the following estimate
\begin{equation}
|\varphi_{2}|_{1, 0; B_{R/4}^{+}} \leq C\|\varphi_{2}\|_{1, 2; B_{1}}.
\end{equation}
Combining this with \eqref{phi2w12}, we infer that for $\delta < R/4$,
\begin{equation}
\| \varphi_{2}\|^{2}_{1, 2; B_{\delta}^{+}} \leq C \delta^{n} \|\nabla u\|^{2}_{2; B_{1}}.
\end{equation}
Next, we note that
\begin{equation*}
w_{4}(\zeta) = - \int_{B_{r_{1}}}\langle d\xi, d^{\ast}\zeta \rangle = - \int_{\partial B_{r_{1}}} \ft d\xi \wedge \ft \ast \zeta,
\end{equation*}
and therefore we can follow the arguments above to obtain
\begin{equation}
\|\varphi_{4}\|^{2}_{1, 2; B_{\delta}^{+}} \leq C\delta^{n}\|\nabla u\|_{2; B_{1}}^{2}, \text{ for }\delta < R/4.
\end{equation}

Finally we estimate $\varphi_{1}$. We first observe that, using \eqref{hodgecoercive} and recalling how we handled (I) in \eqref{laplaciancomputation}, we have
\begin{align*}
c_{0}\|\varphi_{1}\|_{1, 2; B_{1}}^{2} & \leq \cD_{g}(\varphi_{1}, \varphi_{1}) = w_{1}(\varphi_{1})\\
 &= -\int_{\partial B_{r_{1}}}\ft \alpha u \times du \wedge \ft \ast \varphi_{1} + \int_{B_{r_{1}}}\langle \alpha u \times du, d^{\ast}\varphi_{1} \rangle.
\end{align*}
From this and our choice of $r_{1}$ we infer that 
\begin{equation}\label{phi1w12bound}
\|\varphi_{1}\|_{1, 2; B_{1}} \leq C\|\nabla u\|_{2; B_{1}}.
\end{equation}
Next, since $u$ takes values in $\CC$, we have the following pointwise bound for $\omega_{1}$ on $B_{r_{1}}$ (see \cite{bbo}, equation (III.28)):
\begin{equation}\label{w1good}
|d(\alpha u \times du)| = \left| d\left( f(|u|)u \right) \times d\left( f(|u|)u \right) \right| \leq C\frac{(1 - |u|^{2})^{2}}{\ep^{2}\beta^{2}}.
\end{equation}
We now let $w = d(\alpha u \times du)\chi_{B_{r_{1}}}$ and let $\tilde{\varphi}$ be the unique solution in $W^{1, 2}_{0}(B_{1}; \RR^{N})$ to 
\begin{equation}\label{tildephipde}
\cD_{g}(\tilde{\varphi}, \zeta) = \int_{B_{r_{1}}}\langle w,  \zeta\rangle, \text{ for all } \zeta \in W^{1, 2}_{0}(B_{1}; \RR^{N}).
\end{equation}
Then, by the first part of Lemma \ref{hodgegreengrowth} we infer that, for $H^{n}$-a.e. $x \in B_{1}$, there holds
\begin{equation}
\tilde{\varphi}_{\gamma} (x) = \int_{B_{r_{1}}} w_{\alpha}(y)G_{\alpha\gamma}(x, y) dy.
\end{equation}
We now choose the radius $R$ introduced at the beginning of Step 1 to be 
\begin{equation}
R = \{ R_{1}, \lambda^{2}/6\},
\end{equation}
where $R_{1}$ is as in Lemma \ref{hodgegreengrowth}.
The bound \eqref{w1good} and the second part of Lemma \ref{hodgegreengrowth} then imply that for $x \in B_{r_{1}} \subset B_{R}$ we have
\begin{align*}
\left| \tilde{\varphi}_{\gamma}(x) \right| &\leq \int_{B_{r_{1}}}|w_{\alpha}(y)| |G_{\alpha\gamma}(x, y)|dy\\
&\leq C\int_{B_{r_{1}}} |x - y|^{2-n}\frac{(1 - |u(y)|^{2})^{2}}{\ep^{2}\beta^{2}}dy\\
&\leq C\int_{B_{r_{1}}} |x - y|^{2-n}\frac{(1 - |u(y)|^{2})^{2}}{\ep^{2}\beta^{2}}\dvol_{g}.
\end{align*}
We now use \eqref{potentialdensitybound} to estimate
\begin{align*}
\int_{B_{r_{1}}} |x - y|^{2-n}\frac{(1 - |u(y)|^{2})^{2}}{\ep^{2}\beta^{2}}\dvol_{g} & \leq \int_{B_{3r_{1}}(x)} |x - y|^{2-n}\frac{(1 - |u(y)|^{2})^{2}}{\ep^{2}\beta^{2}}\dvol_{g}\\
&\leq C\beta^{-2}\int_{B_{1}}e_{\ep}(u)\dvol_{g}.
\end{align*}
Combining the above two strings of inequalities gives 
\begin{equation}\label{tildephisup}
|\tilde{\varphi}|_{0; B_{r_{1}}} \leq C\beta^{-2}\int_{B_{1}}e_{\ep}(u)\dvol_{g}.
\end{equation}
Testing the system \eqref{tildephipde} against $\tilde{\varphi}$ itself, we obtain with the help of \eqref{hodgecoercive} and the estimates \eqref{w1good}, \eqref{tildephisup} that
\begin{align}
\| \tilde{\varphi} \|_{1, 2; B_{1}}^{2} &\leq C \|w_{1}\|_{1; B_{r_{1}}} |\tilde{\varphi}|_{0; B_{r_{1}}}\\
&\leq C\beta^{-4} \left( \int_{B_{1}} \frac{(1 - |u|^{2})^{2}}{\ep^{2}}\dvol_{g} \right)  \left( \int_{B_{1}}e_{\ep}(u) \dvol_{g}\right).
\end{align}
Next, since $u$ is an even function across $T$, we have by uniqueness that $r^{\ast}\tilde{\varphi} = \tilde{\varphi}$, and therefore the difference $\varphi_{1} - \tilde{\varphi}$ lies in $W^{1, 2, +}_{2, \ft}(B_{1})$. Moreover, it satisfies
\begin{equation*}
\cD_{g}(\varphi_{1} - \tilde{\varphi}, \zeta) = 0, \text{ for all }\zeta \in W^{1, 2, +}_{2, 0}(B_{1}).
\end{equation*}
Using Lemma \ref{ellipticform}, Remark \ref{fitfundamentalsolution}, \eqref{phi1w12bound} and interior elliptic estimates, we infer that
\begin{align*}
|\varphi_{1} - \tilde{\varphi}|^{2}_{0; B_{1/2}^{+}} &\leq C\| \varphi_{1} - \tilde{\varphi} \|^{2}_{1, 2; B_{1}^{+}} \leq C\|\varphi_{1}\|^{2}_{1, 2; B_{1}^{+}} + C\|\tilde{\varphi}\|^{2}_{1, 2; B_{1}^{+}}\\
&\leq C\|\nabla u\|_{2; B_{1}}^{2} +C\beta^{-4} \left( \int_{B_{1}} \frac{(1 - |u|^{2})^{2}}{\ep^{2}} \dvol_{g}\right)  \left( \int_{B_{1}}e_{\ep}(u) \dvol_{g}\right).
\end{align*}
Hence, for all $\delta < R/4$, we have
\begin{align*}
\| \varphi_{1} \|^{2}_{1, 2; B_{\delta}^{+}}  &\leq C\|\tilde{\varphi}\|^{2}_{1, 2; B_{\delta}^{+}} + C\| \varphi_{1} - \tilde{\varphi} \|^{2}_{1, 2; B_{\delta}^{+}}\\
&\leq C(1 + \delta^{n})\beta^{-4} \left( \int_{B_{1}} \frac{(1 - |u|^{2})^{2}}{\ep^{2}} \dvol_{g} \right)  \left( \int_{B_{1}}e_{\ep}(u) \dvol_{g}\right) + C\delta^{n}\| \nabla u \|^{2}_{2; B_{1}}.
\end{align*}
Combining the above estimates for $\xi$ and for $\varphi_{i}$, $i = 1, \cdots, 5$, we arrive at
\begin{align}\label{phaseestimate}
\| u \times du \|_{2; B_{\delta}^{+}}^{2} \leq&\ C(1 + \delta^{n})\beta^{-4} \left( \int_{B_{1}} \frac{(1 - |u|^{2})^{2}}{\ep^{2}} \dvol_{g} \right)  \left( \int_{B_{1}}e_{\ep}(u) \dvol_{g}\right)\\\nonumber
&+ C\left( \beta^{2} + \delta^{n} \right)\|\nabla u\|^{2}_{2; B_{1}}.
\end{align}

\vskip 2mm
\noindent\textbf{Step 3: Estimates for $\left| \nabla |u|^{2} \right|^{2}$ and $(1 - |u|^{2})|\nabla u|^{2}$.}
\vskip 1mm

The estimates in this step are done exactly the same as in \cite{bbo}, and thus we'll merely state the conclusions and refer the reader to \cite{bbo} for details. Specifically, the following hold.
\begin{equation}\label{modulusestimate}
\int_{B_{r_{1}}}\left| \nabla |u|^{2} \right|^{2}\dvol_{g} \leq \beta^{2}\|\nabla u\|^{2}_{2; B_{1}} + C\beta^{-2}\int_{B_{1}}\frac{(1 - |u|^{2})^{2}}{\ep^{2}}\dvol_{g}.
\end{equation}
\begin{equation}\label{remainingestimate}
\int_{B_{r_{1}}} (1 - |u|^{2})|\nabla u|^{2}\dvol_{g} \leq \beta^{2}\|\nabla u\|^{2}_{2; B_{1}} + C\beta^{-2}\int_{B_{1}}\frac{(1 - |u|^{2})^{2}}{\ep^{2}}\dvol_{g}. 
\end{equation}

\vskip 2mm
\noindent\textbf{Step 4: Conclusion}
\vskip 1mm

Putting together \eqref{phaseestimate}, \eqref{modulusestimate} and \eqref{remainingestimate}, we conclude that for $\delta < R/4$,
\begin{align*}
\int_{B_{\delta}} e_{\ep}(u) \dvol_{g} \leq&\ C \left( \beta^{2} + \delta^{n} \right) \| \nabla u \|^{2}_{2; B_{1}} + C\beta^{-2}\int_{B_{1}} \frac{(1 - |u|^{2})^{2}}{\ep^{2}}\dvol_{g}\\
&+ C\beta^{-4} \left( \int_{B_{1}} \frac{(1 - |u|^{2})^{2}}{\ep^{2}} \dvol_{g} \right)  \left( \int_{B_{1}}e_{\ep}(u) \dvol_{g}\right).
\end{align*}
We can now finish the proof as in \cite{bbo} by distinguishing two cases. If 
\begin{equation*}
p_{\ep} \equiv \int_{B_{1}}\frac{(1 - |u|^{2})^{2}}{\ep^{2}}\dvol_{g} \leq (1/8)^{6}, 
\end{equation*}
then we may choose $\beta = p_{\ep}^{1/6}$ and get the desired estimate \eqref{energydecayestimate}. On the other hand, if $p_{\ep} \geq (1/8)^{6}$, then \eqref{energydecayestimate} holds obviously. Hence the proof of Proposition \ref{energydecay} is complete. (The inequality \eqref{energydecayestimate} was stated for half balls, but since $u$ is even across $T$, it doesn't matter whether we use half balls or whole balls.)
\end{proof}

\subsection{Proof of Theorem \ref{boundaryeta}}
The proof consists of applying Proposition \ref{energydecay} at a suitable scale, to be chosen with the help of the following result.
\begin{prop}[\cite{bbo}, Lemma III.1]
\label{goodradius}
Let $g \in \cM_{\lambda, \Lambda}$ and let $u: (B_{1}, g) \to \CC$ solve \eqref{GLequation}. Assume in addition that \eqref{smallenergyscaled} holds for some $\eta$. Then, for $\delta < \min\{1/4, \rho_{0}\}$ (with $\rho_{0}$ given by Proposition \ref{reversemonotonicity}) and $\ep < \delta^{2}$, there exists a radius $r_{0} \in (\ep^{1/2}, \delta)$ such that the following three inequalities hold.
\begin{equation}\label{potentialbound}
r_{0}^{2-n}\int_{B_{r_{0}}}\frac{(1 - |u|^{2})^{2}}{2\ep^{2}}\dvol \leq K\eta \left| \log\delta \right|,
\end{equation}
\begin{equation}\label{boundfrommonotonicity}
\int_{\delta r_{0}}^{r_{0}}\left[ r^{1-n} \int_{B_{r}}\frac{(1 - |u|^{2})^{2}}{2\ep^{2}}\dvol +r^{2-n} \int_{\partial B_{r}}|\nabla_{r}u|^{2} \sqrt{\det(g)}dH^{n-1} \right] dr \leq K\eta \left| \log\delta \right|,
\end{equation}
\begin{equation}\label{boundfromreversemono}
r_{0}^{2-n}\int_{B_{r_{0}}} e_{\ep}(u) \dvol \leq K(\delta r_{0})^{2-n}\int_{B_{\delta r_{0}}}e_{\ep}(u)\dvol +  K \eta \left| \log\delta \right|,
\end{equation}
where $K$ is a constant depending only on $n, \lambda$ and $\Lambda$.
\end{prop}
\begin{proof}[Sketch of proof]
We follow step-by-step the proof of Lemma III.1 in \cite{bbo} to find $r_{0} \in (\ep^{1/2}, \delta)$ such that \eqref{potentialbound} and \eqref{boundfrommonotonicity} hold, with the only change being that \eqref{lipschitzmonotonicity1} should be used in place of the monotonicity formula (II.1) used there. Next, we use Proposition \ref{reversemonotonicity} in place of of Lemma II.2 of \cite{bbo} to derive \eqref{boundfromreversemono} from \eqref{boundfrommonotonicity}. We refer the reader to Lemma III.1 of \cite{bbo} for details. 
\end{proof}

The conclusion of Theorem \ref{boundaryeta} now follows from Propositions \ref{energydecay} and \ref{goodradius} in the same way Theorem 2 follows from Theorem 3 and Lemma III.1 in \cite{bbo}. Nevertheless, we include this argument with the necessary modifications.

Reflecting $u$ evenly across $T$ and letting $\delta < \delta_{0}$ be a small constant to be determined later, with $\delta_{0}$ given by Proposition \ref{energydecay}, then Proposition \ref{goodradius} yields a radius $r_{0} \in (\ep^{1/2}, \delta_{0})$ for which \eqref{potentialbound}, \eqref{boundfrommonotonicity} and \eqref{boundfromreversemono} hold. Next, we consider the following rescaling:
\begin{equation}
\tilde{u}(x) = u(r_{0}x);\ \tilde{\ep} = \ep/r_{0};\ \tilde{g}_{ij}(x) = g_{ij}(r_{0}x).
\end{equation}
Then its easy to see that $\tilde{g}$ again belongs to $\cM^{+}_{\mu, \lambda, \Lambda}$ and that $\tilde{u}$ solves \eqref{halfGL} with $\tilde{\ep}$ in place of $\ep$. Therefore, since $\delta < \delta_{0}$, we may apply Proposition \ref{energydecay} to get \eqref{energydecayestimate} with $\tilde{u}$, $\tilde{\ep}$ and $\tilde{g}$ in place of $u$, $\ep$ and $g$, respectively. Scaling back, we obtain
\begin{align}
r_{0}^{2-n}\int_{B_{\delta r_{0}}}e_{\ep}(u) \dvol_{g} \leq &C \left[ \left( r_{0}^{2-n}\int_{B_{r_{0}}}\frac{(1 - |u|^{2})^{2}}{\ep^{2}} \dvol_{g} \right)^{1/3}\left( r_{0}^{2-n}\int_{B_{r_{0}}}e_{\ep}(u) \dvol_{g} \right) \right.\\
\nonumber &+ \left.\left( r_{0}^{2-n}\int_{B_{r_{0}}}\frac{(1 - |u|^{2})^{2}}{\ep^{2}}\dvol_{g} \right)^{2/3} + \delta^{n}r_{0}^{2-n}\int_{B_{r_{0}}}e_{\ep}(u)\dvol_{g} \right].
\end{align}
Combining the above inequality and \eqref{boundfromreversemono} from Proposition \ref{goodradius} and recalling \eqref{potentialbound}, we get the following inequality.
\begin{align}
r_{0}^{2-n}\int_{B_{r_{0}}}e_{\ep}(u) \dvol \leq &C \left[ \left( \eta \left| \log\delta \right| \right)^{1/3}\left( \delta^{2-n} r_{0}^{2-n}\int_{B_{r_{0}}}e_{\ep}(u) \dvol \right) \right.\\
\nonumber &+ \left. \delta^{2-n}\left( \eta \left| \log\delta \right| \right)^{2/3} + \delta^{2}r_{0}^{2-n}\int_{B_{r_{0}}}e_{\ep}(u)\dvol \right] + K\eta \left|\log\delta  \right|.
\end{align}
Moving all the terms involving $r_{0}^{2-n}\int_{B_{r_{0}}}e_{\ep}(u)\dvol$ to the left-hand side, we get
\begin{align}
&\left( 1 - \frac{C\left( \eta \left| \log\delta \right| \right)^{1/3}}{\delta^{n-2}} - C\delta^{2} \right) r_{0}^{2-n}\int_{B_{r_{0}}}e_{\ep}(u)\dvol \\
\nonumber \leq &C\delta^{2-n} \left( \eta \left| \log\delta \right| \right)^{2/3}+ K\eta \left| \log\delta \right|\\
\nonumber  \leq & C\delta^{2-n}\left( \eta \left| \log\delta \right| \right)^{2/3},
\end{align}
where in going from the second to the third line we absorbed the term $K\eta \left| \log\delta \right|$. To continue, we choose $\delta = \eta^{1/3n}$ if $\ep^{3n/2} < \eta$ (we required that $\delta^{2} > \ep$). Then the above inequalities give
\begin{equation}
\left( 1 - C\eta^{2/3n} \left| \log\eta \right| \right) r_{0}^{2-n}\int_{B^{+}_{r_{0}}}e_{\ep}(u)\dvol  \leq C\eta^{(n+2)/3n}\left| \log\eta \right|^{2/3}.
\end{equation}
Hence, if $\eta_{0}$ is chosen small enough, then we have
\begin{equation}\label{energyboundrscale}
r_{0}^{2-n}\int_{B_{r_{0}}}e_{\ep}(u)\dvol \leq C\eta^{(n + 2)/3n}\left| \log\eta \right|^{2/3},
\end{equation}
provided $\ep^{3n/2} < \eta < \eta_{0}$. However, by the monotonicity formula \eqref{lipschitzmonotonicity1} and the bound \eqref{smallenergyscaled}, inequality \eqref{energyboundrscale} also holds when $\ep^{3n/2} \geq \eta$. Therefore \eqref{energyboundrscale} holds as long as $\eta < \eta_{0}$. 

Next, we invoke Corollary \ref{monotonicitydecentralized} (appropriately scaled) and the remark following it to get that, for all $x \in B_{3r_{0}/4}$ and $\ep$ sufficiently small such that $\ep < \ep^{1/2}/8 < r_{0}/4$, there holds
\begin{equation}
\ep^{2-n}\int_{B_{\ep}(x)}e_{\ep}(u)\dvol_{g} \leq Cr_{0}^{2-n}\int_{B_{r_{0}}}e_{\ep}(u)\dvol_{g} \leq C\eta^{(n+2)/3n}\left| \log\eta \right|^{2/3}.
\end{equation}
In particular, we have
\begin{equation}\label{meanvaluesmall}
\ep^{-n}\int_{B_{\ep}(x)}(1 - |u|^{2})^{2}\dvol_{g} \leq C\eta^{(n + 2)/3n}\left| \log\eta \right|^{2/3}.
\end{equation}
Combining \eqref{meanvaluesmall} and Lemma III.3 of \cite{bbo}, we get \eqref{etaellipticityconclusion}, and hence Theorem \ref{boundaryeta} is proved with $r_{1} = 3r_{0}/4$.

\begin{rmk} Lemma III.3 of \cite{bbo} remains valid in our case because its proof only requires a gradient estimate of the form $|\nabla u| = O(\ep^{-1})$, which we have from Lemma 3.1.
\end{rmk}
\subsection{A corollary of $\eta$-ellipticity}
Combining Theorems \ref{interioreta} and \ref{boundaryeta}, we get the following corollary which will be useful in Section 6.
\begin{coro}\label{etacoro}
For all $\sigma \in (0, 1/4)$, there exists constants $\eta_{1}, \ep_{1}$, depending only on $n, \mu, \lambda, \Lambda$ and $\sigma$, such that if $g$ and $u$ are as in Theorem \ref{boundaryeta}, with $\ep < \ep_{1}$ and $u$ satisfying \eqref{smallenergyscaledbdy} for some $\eta < \eta_{1}$, then we have
\begin{equation}
|u(x)| \geq 1 - \sigma, \text{ for all }x \in B_{3/4}^{+}.
\end{equation}
\end{coro}
The proof of Corollary \ref{etacoro} requires the following two preliminary results, which roughly say that we can apply Theorems \ref{interioreta} and \ref{boundaryeta} to balls $B_{r}(x) \subset B_{1}$ not necessarily centered at the origin.

\begin{lemm}\label{etadecentralized}
For all $\sigma \in (0, 1/4)$, there exists constants $\eta_{2}, \ep_{2}$, depending only on $n, \mu, \lambda, \Lambda$ and $\sigma$, such that if $g \in \cM_{\mu, \lambda, \Lambda}$, $x_{0} \in B_{4/5}, \rho < 1/5$ and $u: B_{\rho}(x_{0}) \to \DD$ is a solution to the Ginzburg-Landau equation satisfying
\begin{equation}\label{rescaledeta}
\rho^{2-n}\int_{B_{\rho}(x_{0})} e_{\ep}(u)\dvol_{g} \leq \eta \left| \log\frac{\ep}{\rho} \right|, 
\end{equation}
where $\ep < \rho\ep_{2}$ and $\eta < \eta_{2}$, then there holds
\begin{equation}
|u(x)| \geq 1 - \sigma, \text{ for all } x \in B_{\lambda \rho \sqrt{\ep}/2}(x_{0}).
\end{equation}
\end{lemm}
\begin{lemm}\label{etadecentralizedbdy}
For all $\sigma \in (0, 1/4)$, there exists constants $\eta_{2}', \ep_{2}'$, depending only on $n, \mu, \lambda, \Lambda$ and $\sigma$, such that if $g \in \cM^{+}_{\mu, \lambda, \Lambda}$, $x_{0} \in T_{4/5}, \rho < 1/5$ and $u: B^{+}_{\rho}(x_{0}) \to \DD$ is a solution to the Ginzburg-Landau equation satisfying
\begin{equation}\label{rescaledetabdy}
\rho^{2-n}\int_{B^{+}_{\rho}(x_{0})} e_{\ep}(u)\dvol_{g} \leq \eta \left| \log\frac{\ep}{\rho} \right|, 
\end{equation}
where $\ep < \rho\ep_{2}'$ and $\eta < \eta_{2}'$, then there holds
\begin{equation}
|u(x)| \geq 1 - \sigma, \text{ for all } x \in B^{+}_{\lambda \rho \sqrt{\ep}/2}(x_{0}).
\end{equation}
\end{lemm}
Lemma \ref{etadecentralized} can be proved using suitable affine transformations as in the proof of Corollary \ref{monotonicitydecentralized}, and the proof is almost the same for Lemma \ref{etadecentralizedbdy}, except that we have to choose the affine transformations to preserve $T$, which can be done thanks to (1) and (2) in Definition \ref{boundaryclass}. We proceed to give the proof of Corollary \ref{etacoro}.
\begin{proof}[Proof of Corollary \ref{etacoro}]
We first reflect $u$ across $T$ using Lemma \ref{GLreflection}. Next, notice that for all $x_{0} \in B_{4/5} \cap T$ and for sufficiently small $\ep$, there holds
\begin{equation}
\left( \frac{1}{5} \right)^{2-n}\int_{B^{+}_{1/5}(x_{0})} e_{\ep}(u)\dvol_{g} \leq 2\cdot5^{n-2} \eta\left| \log\frac{\ep}{5} \right|
\end{equation}
Thus, requiring $\eta$ and $\ep$ to be small enough, we may apply Lemma \ref{etadecentralizedbdy} on $B_{1/5}^{+}(x_{0})$ for all $x_{0} \in B_{4/5} \cap T$ to get that $|u(x)| \geq 1 - \sigma$ whenever $x = (x', t) \in B_{4/5}^{+}$ satisfies $t \leq \lambda  \ep^{1/2}/10 \equiv \rho_{0}$. 

On the other hand, for $x_{0} = (x_{0}', t) \in B_{3/4}^{+}$ with $t \geq \rho_{0}$, we have, for sufficiently small $\ep$, 
\begin{align*}
\rho_{0}^{2-n}\int_{B_{\rho_{0}}(x_{0})} e_{\ep}(u)\dvol_{g} &\leq C \left( \frac{1}{4} \right)^{2-n}\int_{B_{1/4}(x_{0})}e_{\ep}(u)\dvol_{g} \\
&\leq C \cdot 4^{n-2} \eta \left| \log\ep \right| \\
&\leq \tilde{C} \cdot 2 \cdot 4^{n-2}\eta \left| \log \rho_{0}^{-1}\ep \right|,
\end{align*}
where the first line follows from Proposition \ref{monotonicitydecentralized}. Since $\ep/\rho_{0} = 10 \ep^{1/2}/\lambda$, we see from the above inequality that for sufficiently small $\eta$ and $\ep$, we may apply Lemma \ref{etadecentralized} to infer that $|u(x_{0})| \geq 1 - \sigma$.

Combining the results of the previous two paragraphs, we get $|u(x)| \geq 1 - \sigma$ whenever $x \in B_{3/4}^{+}$, and the proof of Corollary \ref{etacoro} is complete.
\end{proof}

\section{Convergence of Ginzburg-Landau solutions I: The regular part}

\subsection{Improvement of the $\eta$-ellipticity theorem}
The main result of this section is the following improvements of Theorems \ref{interioreta} and \ref{boundaryeta}, which require slightly better regularity of the metrics. The precise statements are given below.

\begin{prop}[Interior version, cf. \cite{bos}]
\label{interioretareg}
There exists constants $\eta_{3}, \ep_{3}$ and $C$, depending only on $n, \lambda$ and $\Lambda$ such that given $g \in \cM_{1, \lambda, \Lambda}$ and a solution $u: (B_{1}, g)\to \DD$ to \eqref{GLequation}, satisfying 
\begin{equation}
\int_{B_{1}}e_{\ep}(u) \dvol_{g} \leq \eta\left| \log\ep \right|
\end{equation}
with $\eta < \eta_{3}$ and $\ep < \ep_{3}$, we have that $|u| \geq 3/4$ on $B_{3/4}$, and that 
\begin{equation}
\sup\limits_{x \in B_{1/2}} e_{\ep}(u)(x) \leq C\int_{B_{1}}e_{\ep}(u)\dvol_{g}.
\end{equation}
\end{prop}

\begin{prop}[Boundary version]
\label{boundaryetareg}
There exists constants $\eta_{3}', \ep_{3}'$ and $C$, depending only on $n, \lambda$ and $\Lambda$ such that if $g$ is a metric in  $\cM^{+}_{1, \lambda, \Lambda}$ with respect to which $T$ is convex, and if $u: (B^{+}_{1}, g)\to \DD$ is a solution to \eqref{halfGL} satisfying 
\begin{equation}
\int_{B^{+}_{1}}e_{\ep}(u) \dvol_{g} \leq \eta\left| \log\ep \right|
\end{equation}
with $\eta < \eta_{3}'$ and $\ep < \ep_{3}'$, then we have $|u| \geq 3/4$ on $B^{+}_{3/4}$ and there holds
\begin{equation}\label{boundaryetaregconclusion}
\sup\limits_{x \in B^{+}_{1/3}} e_{\ep}(u)(x) \leq C\int_{B^{+}_{1}}e_{\ep}(u)\dvol_{g}.
\end{equation}
\end{prop}
We only prove the boundary version here since the proof of the interior version is essentially the same as in \cite{bos}. A key ingredient would be the following result, the interior version of which was first proven by Chen and Struwe (\cite{chenstruwe}).
\begin{lemm}
\label{gammareg}
There exist constants $\gamma_{0}$ and $C$, depending only on $n, \lambda$ and $\Lambda$, such that if $g$ is a metric in $\cM^{+}_{1, \lambda, \Lambda}$, $T$ is convex with respect to $g$ and $u:(B^{+}_{1}, g) \to \DD$ is a solution to \eqref{halfGL} satisfying 
\begin{equation}\label{gammasmall}
\int_{B^{+}_{1}}e_{\ep}(u)\dvol_{g} \leq \gamma_{0},
\end{equation}
then there holds
\begin{equation}\label{gammabound}
\sup\limits_{x \in B_{1/2}^{+}} e_{\ep}(u)(x) \leq C \int_{B^{+}_{1}} e_{\ep}(u)\dvol_{g}.
\end{equation}
\end{lemm}
\begin{proof}
The proof we give below is adapted from that of Theorem 2.2 \cite{schoenharmonic}.
We begin by noting the following B\^ochner-type inequality on $B_{1}^{+}$:
\begin{equation}\label{GLbochner}
\Delta e_{\ep}(u) \geq |\nabla^{2} u|^{2} + \Ric(\nabla u, \nabla u) + \ep^{-2}\left| \nabla |u|^{2} \right|^{2} - C_{n}e_{\ep}(u)^{2},
\end{equation}
where $C_{n}$ is a dimensional constant. This inequality makes sense almost everywhere since $g$ is $C^{1, 1}$ and since $u \in W^{3, 2}$ by elliptic regularity. Moreover, \eqref{GLbochner} can be easily derived from the Ginzburg-Landau equation following the computations in \cite{chenstruwe} or \cite{js}, the only difference being that a Ricci-term appears when one commutes the covariant derivative with the Laplace operator.

Next, we use Lemma \ref{GLreflection} to reflect $u$ evenly across $T$, let $r_{1} = 3/4$ and consider a point $x_{0} \in B^{+}_{r_{1}}$ where the following maximum is attained.
\begin{equation*}
\max\limits_{x \in B^{+}_{r_{1}}} \left( r_{1} - |x| \right)^{2}e_{\ep}(u)(x).
\end{equation*}
Following the proof of Theorem 2.2 of \cite{schoenharmonic}, we rescale by letting $e_{0} = e_{\ep}(u)(x_{0}), \rho_{0} = (r_{1} - |x_{0}|)/2$ and defining
\begin{equation*}
\tilde{\ep} = e_{0}^{1/2}\ep;\  \tilde{g}_{ij}(y) = g_{ij}(x_{0} + e_{0}^{-1/2}y) ;\ \tilde{u}(y) = u(x_{0} + e_{0}^{-1/2}y), \text{ for } y \in B_{e_{0}^{1/2}\rho_{0}}.
\end{equation*}
Then, as in \cite{schoenharmonic} there hold
\begin{equation}\label{rescaledGLbound}
e_{\tilde{\ep}}(\tilde{u})(0) = 1 \text{ and }e_{\tilde{\ep}}(\tilde{u})(y) \leq 4, \text{ for }y \in B_{e_{0}^{1/2}\rho_{0}}.
\end{equation}
To continue, we let $\psi(y) = x_{0} + e^{-1/2}_{0}y$ and define 
\begin{equation*}
B_{r, +}(y) \equiv B_{r}(y) \cap \psi^{-1}(B_{1}^{+}) \text{ and } \tilde{T}_{r}(y) \equiv B_{r}(y)\cap \psi^{-1}(T).
\end{equation*} 
Then we claim that for all $0 \leq \zeta \in W^{1, 2}_{0}(B_{e_{0}^{1/2}\rho_{0}})$, there holds
\begin{equation}\label{improvedetaclaim}
\int_{B_{e_{0}^{1/2}\rho_{0}, +}(0)} \langle \nabla \zeta, \nabla e_{\tilde{\ep}}(\tilde{u}) \rangle \dvol_{\tilde{g}} \leq C_{n, \lambda, \Lambda}\left( 1 + e_{0}^{-1} \right)\int_{B_{e_{0}^{1/2}\rho_{0}, +}(0)} \zeta e_{\tilde{\ep}}(\tilde{u}) \dvol_{\tilde{g}}.
\end{equation}
To see this, we first integrate by parts to get
\begin{equation}\label{bochnerbyparts}
\int_{B_{e_{0}^{1/2}\rho_{0}, +}(0)} \langle \nabla \zeta, \nabla e_{\tilde{\ep}}(\tilde{u}) \rangle \dvol_{\tilde{g}}  = -\int_{\tilde{T}_{e_{0}^{1/2}\rho_{0}}(0)} \zeta \partial_{n} e_{\tilde{\ep}}(\tilde{u}) d\sigma_{\tilde{g}} - \int_{B_{e_{0}^{1/2}\rho_{0}, +}(0)} \zeta \Delta e_{\tilde{\ep}}(\tilde{u})\dvol_{\tilde{g}}.
\end{equation}
For the boundary term on the right, we note that the following holds on $\tilde{T}_{e_{0}^{1/2}\rho_{0}}(0)$.
\begin{equation}
\partial_{n}e_{\tilde{\ep}}(\tilde{u}) = \langle \nabla_{\partial_{n}} \nabla u, \nabla u \rangle = \langle \nabla_{\nabla u}\nabla u, \partial_{n} \rangle = -\langle \nabla_{\nabla u}\partial_{n}, \nabla u \rangle \geq 0,
\end{equation}
where in the last inequality we used the convexity assumption, and the preceding equality follows from the Neumann boundary condition. On the other hand, for the second term on the right-hand side of \eqref{bochnerbyparts}, we use \eqref{GLbochner} (properly scaled), \eqref{rescaledGLbound} and the fact that $g \in \cM^{+}_{1, \lambda, \Lambda}$ to infer that
\begin{equation}
\Delta e_{\tilde{\ep}}(\tilde{u}) \geq - C_{n}\left( 1 + \frac{C_{n, \lambda, \Lambda}}{e_{0}} \right) e_{\tilde{\ep}}(\tilde{u}).
\end{equation}
Plugging the two inequalities above into \eqref{bochnerbyparts}, we get \eqref{improvedetaclaim} as claimed. To conclude, we write $x_{0} = (x_{0}', t_{0})$ and distinguish two cases according to the ratio of $t_{0}$ to $\rho_{0}$. 
\begin{enumerate}
\item[(1)] If $t_{0} \geq \rho_{0}/4$, then $B_{e_{0}^{1/2}\rho_{0}/4, +} = B_{e_{0}^{1/2}\rho_{0}/4}$ and we can finish the proof as in Theorem 2.2 of \cite{schoenharmonic}, with the help of \eqref{improvedetaclaim}. (See also the arguments for Case (2) below.) Note that our inequalities \eqref{energydensitybound} and \eqref{GLbochner} should replace inequalities (2.5) and (2.1) of \cite{schoenharmonic}, respectively.
\vskip 2mm
\item[(2)] If $t_{0} \leq \rho_{0}/4$, we in fact still follow \cite{schoenharmonic}, only that slightly more care is needed. For convenience, we introduce the notation $\hat{y} = \psi^{-1}\left((x_{0}', 0)\right) = (0, - e_{0}^{1/2}t_{0})$. We now show that $e_{0}^{1/2}\rho_{0} \leq 2$. Suppose not, then we have $e_{0}^{-1} \leq \rho_{0}^{2} \leq 1$. Hence by \eqref{improvedetaclaim} and the mean value inequality (see for instance \cite{gt}, Theorem 8.17, but with balls replaced by half-balls) applied to $B_{1, +}(\hat{y}) \subseteq B_{e_{0}^{1/2}\rho_{0}, +}(0)$, we have
\begin{align*}
1 = e_{\tilde{\ep}}(\tilde{u})(0) &\leq \sup\limits_{y \in B_{e_{0}^{1/2}t_{0}, +}(\hat{y})}(\tilde{u})(y) \leq C\int_{B_{1, +}(\hat{y})} e_{\tilde{\ep}}(\tilde{u}) \dvol_{\tilde{g}}\\
& = C \left( e_{0}^{-1/2} \right)^{2-n}\int_{B^{+}_{e_{0}^{-1/2}}(x'_{0})} e_{\ep}(u)\dvol_{g}\\
& \leq C \int_{B_{1}} e_{\ep}(u)\dvol_{g} \leq C\gamma_{0},
\end{align*}
where we used \eqref{energydensitybound} in the third inequality, and the last inequality follows from \eqref{gammasmall}. We then arrive at a contradiction if $\gamma_{0}$ is chosen small enough, and therefore $e_{0}^{1/2}\rho_{0} \leq 2$ as claimed. Now we use the mean value inequality again, but this time on $B_{e_{0}^{1/2}\rho_{0}/2, +}(\hat{y})$ to get
\begin{align*}
1 = e_{\tilde{\ep}}(\tilde{u})(0) &\leq \sup\limits_{y \in B_{e_{0}^{1/2}t_{0}, +}(\hat{y})} e_{\tilde{\ep}}(\tilde{u})(y) \leq C \left( e_{0}^{1/2}\rho_{0} \right)^{-n} \int_{B_{e_{0}^{1/2}\rho_{0}/2, +}(\hat{y})} e_{\tilde{\ep}}(\tilde{u}) \dvol_{\tilde{g}}\\
&= C\left( e_{0}^{1/2}\rho_{0} \right)^{-2} \rho_{0}^{2-n} \int_{B^{+}_{\rho_{0}/2}(x_{0}')} e_{\ep}(u)\dvol_{g}\leq Ce_{0}^{-1}\rho_{0}^{-2} \int_{B_{1}}e_{\ep}(u)\dvol_{g}.
\end{align*}
Thus we conclude that 
\begin{equation}
\left( r_{1} - |x_{0}| \right)^{2}e_{\ep}(u)(x_{0}) = 4\rho_{0}^{2}e_{0} \leq C\int_{B_{1}^{+}}e_{\ep}(u)\dvol_{g},
\end{equation}
which immediately implies \eqref{gammabound} by our choice of $x_{0}$.
\end{enumerate}
\end{proof}

A second ingredient we need for the proof of Theorem \ref{boundaryetareg} is a boundary version of Theorem A.2 of \cite{bos}.
\begin{lemm}\label{boundarybos}
There exists constants $\tilde{\ep}, \sigma_{0}, \alpha_{0}$ and $C$, depending only on $n, \lambda$ and $\Lambda$ such that if $g \in \cM_{1, \lambda, \Lambda}^{+}$, $T$ is convex with respect to $g$, and $u: (B_{1}^{+}, g) \to \DD$ is a solution to \eqref{halfGL} with $\ep < \tilde{\ep}$, satisfying
\begin{equation}
\label{alphaenergy}
\int_{B_{1}^{+}} e_{\ep}(u)\dvol_{g} \leq \ep^{-\alpha_{0}} \text{, and }
\end{equation}
\begin{equation}
\label{sigmaclose}
|u| \geq 1 - \sigma_{0} \text{ on }B_{1}^{+},
\end{equation}
then the gradient estimate \eqref{gammabound} holds.
\end{lemm}
\begin{proof}[Sketch of proof]
We begin by reflecting $u$ using Lemma \ref{GLreflection}. Then the arguments for inequality (A.28) in \cite{bos} carries over with only minor modifications to show that when $\sigma_{0}$ is small enough, we have, for $x_{0} \in B_{1/2}$ and $r_{0} \in (0, 1/8)$, 
\begin{align}\label{boscopy}
r_{0}^{2-n}\int_{B_{r_{0}}} e_{\ep}(u)\dvol_{g} &\leq C\left( r_{0}^{2} + r_{0}^{2-n}\ep^{\beta_{0}} \right) \int_{B_{1}}e_{\ep}(u)\dvol_{g}\\
\nonumber & \leq C\left( r_{0}^{2} + r_{0}^{2-n}\ep^{\beta_{0}} \right)\ep^{-\alpha_{0}},
\end{align}
where $C$ and $\beta_{0}$ are some constants depending at most on $n, \lambda$ and $\Lambda$. Choosing $r_{0} = \ep^{\beta_{0}/n}$ and requiring that $\alpha_{0} < 2\beta_{0}/n$, we see that for $\ep$ small enough, there holds
\begin{equation}\label{gammaassumption}
r_{0}^{2-n}\int_{B_{r_{0}}(x_{0})}e_{\ep}(u)\dvol_{g} \leq \gamma_{0}/4^{n-2}.
\end{equation}
As in the previous proof, we now let $x_{0} = (x_{0}', t_{0})$ and distinguish two cases.
\begin{enumerate}
\item[(1)] If $r_{0} \leq 4t_{0}$, then obviously $B_{r_{0}/4}(x_{0}) \subseteq B_{1}^{+}$. Moreover, from \eqref{gammaassumption} we see that
\begin{equation}
\left( \frac{r_{0}}{4} \right)^{2-n}\int_{B_{r_{0}/4}(x_{0})} e_{\ep}(u)\dvol_{g} \leq \gamma_{0}.
\end{equation}
An argument similar to the proof of Lemma \ref{etadecentralized} then allows us to apply Lemma \ref{gammareg}, or rather its interior version, on $B_{r_{0}/4}(x_{0})$ to infer that
\begin{equation}
r_{0}^{2}e_{\ep}(u)(x_{0}) \leq Cr_{0}^{2-n}\int_{B_{r_{0}/4}(x_{0})}e_{\ep}(u)\dvol_{g} \leq Cr_{0}^{2}\int_{B_{1}}e_{\ep}(u)\dvol_{g},
\end{equation}
and we are done. Note that we used \eqref{boscopy} in the second inequality.
\vskip 2mm
\item[(2)] If $r_{0} > 4t_{0}$, then we have $B^{+}_{r_{0}/2}(x_{0}') \subseteq B_{r_{0}}(x_{0})$ and \eqref{gammaassumption} yields
\begin{equation}
\left( \frac{r_{0}}{2} \right)^{2-n}\int_{B_{r_{0}/2}^{+}(x_{0}')}e_{\ep}(u)\dvol_{g} \leq \gamma_{0},
\end{equation}
and therefore we may use Lemma \ref{gammareg} on $B_{r_{0}/2}^{+}(x_{0}')$ to get
\begin{align}
r_{0}^{2} e_{\ep}(u)(x_{0}) &\leq r_{0}^{2}\sup\limits_{x \in B_{r_{0}/4}(x_{0}')} e_{\ep}(u)(x) \leq Cr_{0}^{2-n}\int_{B^{+}_{r_{0}/2}(x_{0}')}e_{\ep}(u)\dvol_{g} \\
\nonumber &\leq Cr_{0}^{2}\int_{B_{1}} e_{\ep}(u)\dvol_{g},
\end{align}
and again we're done upon dividing through by $r_{0}^{2}$.
\end{enumerate}
\end{proof}
\begin{proof}[Proof of Proposition \ref{boundaryetareg}]
By Corollary \ref{etacoro}, if we choose $\eta_{3}'$ and $\ep_{3}'$ small enough, then $u$ verifies \eqref{sigmaclose} on $B_{3/4}^{+}$ for some $\sigma_{0} < 1/4$, which verifies the first assertion. Moreover, decreasing $\ep_{3}'$ if necessary, we can also ensure that $\left| \log\ep \right| < \ep^{-\alpha_{0}}$ for all $\ep < \ep_{3}'$. Therefore we may apply Lemma \ref{boundarybos} on $B_{3/4}^{+}$ to conclude that 
\begin{equation}
\sup\limits_{x \in B_{3/8}^{+}} e_{\ep}(u)(x) \leq C\int_{B_{3/4}^{+}} e_{\ep}(u)\dvol_{g},
\end{equation}
which certainly implies the desired estimate, \eqref{boundaryetaregconclusion}. 
\end{proof}
\subsection{Higher-order estimates}
Starting from Propositions \ref{interioretareg} and \ref{boundaryetareg}, we may derive higher-order interior and boundary estimates by an inductive argument adapted from \cite{bbh}. 

\begin{thm}[Interior estimates]
\label{interiorhigher}
Suppose that $g$ is a smooth metric on $B_{1}$ and that $u : (B_{1}, g) \to \DD$ is a solution to \eqref{GLequation} satisfying 
\begin{equation}\label{energyeta}
\int_{B_{1}} e_{\ep}(u)\dvol_{g} < \eta\left| \log\ep \right|,
\end{equation}
for some $\eta < \eta_{3}$ and $\ep < \ep_{3}$, where $\eta_{3}, \ep_{3}$ are determined by Theorem \ref{interioretareg}. Then, letting $X \equiv \ep^{-2}(1 - |u|)$ and writing $u$ in polar coordinates as $u = \rho e^{i\varphi}$, we have
\begin{align}
|\nabla \varphi|_{l, 0; B_{3^{-3l-1}}} &\leq C_{l}\left| \log\ep \right|^{1/2}, \label{inductivephase}\\
|X|_{l, 0; B_{3^{-3l-2}}} &\leq C_{l} \left| \log\ep  \right|, \label{inductivemodulus}\\
|\nabla \rho|_{l, 0; B_{3^{-3l-3}}} & \leq C_{l}\ep^{a_{l}} \left| \log\ep \right|^{b_{l}}, \label{inductivegradmodulus}
\end{align}
for all $l \geq 0$, where $a_{l}$ and $b_{l}$ are constants depending only on $n$, $l$ and $g$.
\end{thm}
\begin{rmk}\label{phiSV}
Note that, under the assumptions of Proposition \ref{interiorhigher}, we have $|u| > 3/4$ on $B_{3/4}$, and hence $\varphi$ is single-valued on $B_{3/4}$. Also, without loss of generality we require that $\fint_{B_{1}}\varphi \in [0, 2\pi)$. Then it's clear that $|\varphi|_{0; B_{1/3}} \leq 2\pi + |\nabla \varphi|_{0; B_{1/3}}$.
\end{rmk}

\begin{proof}[Proof of Theorem \ref{interiorhigher}]
The proof will proceed by induction on $l$. When $l = 0$, the estimate \eqref{inductivephase} follows from \eqref{energyeta} and Proposition \ref{interioretareg}. To see \eqref{inductivemodulus}, we note that $\varphi$ and $\rho$ satisfy
\begin{equation}\label{phiPDE}
\Delta \varphi + 2 \rho^{-1} \langle \nabla \rho, \nabla \varphi \rangle = 0.
\end{equation}
\begin{equation}\label{rhoPDE}
-\Delta \rho + \rho |\nabla\varphi|^{2} = \ep^{-2}\left( 1 - \rho^{2} \right)\rho.
\end{equation}
A direct computation using \eqref{rhoPDE} then shows that
\begin{equation}\label{XPDE}
-\ep^{2}\Delta X + \rho(1 + \rho) X =  \rho |\nabla\varphi|^{2}.
\end{equation}
Therefore, from \eqref{inductivephase} and the fact that $1 \geq \rho > 3/4$ on $B_{3/4}$ we infer that
\begin{equation*}
-\ep^{2}\Delta X + X \leq C\left| \log\ep \right|, \text{ on }B_{1/3}
\end{equation*}
Therefore, with the help of Lemma \ref{bbhmax}, we easily get \eqref{inductivemodulus}. Next, we return to the equation \eqref{rhoPDE} and observe that, by \eqref{inductivephase} and \eqref{inductivemodulus}, we have
\begin{equation*}
|\Delta \rho|_{0; B_{3^{-2}}} \leq C\left| \log\ep \right| \text{ and } |1 - \rho|_{0; B_{3^{-2}}} \leq C\ep^{2}\left| \log\ep \right|.
\end{equation*}
Combining this with Lemma \ref{interiormulti} applied to $1 - \rho$, we obtain \eqref{inductivegradmodulus} with $\alpha_{0} = \beta_{0} = 1$.

For the inductive step, we assume the desired estimates hold for $0, 1, \cdots, l$ and prove them for $l + 1$. For later use, we fix radii $3^{-3l-4} < r_{3} < r_{2} < r_{1} < 3^{-3l-3}$. Also, $C$, $a$ and $b$ will denote positive constants that depend only on $n$, $l$ and $g$, and may change from line to line. To begin, we see from \eqref{phiPDE} and the induction hypotheses that
\begin{equation}\label{Deltaphicl}
|\Delta \varphi |_{l; B_{3^{-3l-3}}} \leq C\ep^{a} \left| \log\ep \right|^{b}.
\end{equation}
By \eqref{Deltaphicl} and the interior Schauder estimates, we have (recall Remark \ref{phiSV})
\begin{equation}\label{philplus1}
| \varphi |_{l+1, 1/2; B_{r_{1}}} \leq C\left( |\varphi |_{0; B_{1/3}} + |\Delta \varphi |_{l; B_{3^{-3l-3}}} \right) \leq C\left| \log \ep \right|^{1/2}.
\end{equation}
On the other hand, taking a multi-index $\alpha$ with length $l$ and apply $\partial^{\alpha}$ to \eqref{rhoPDE}, then we get
\begin{equation}\label{rhopartialPDE}
-\Delta \partial^{\alpha}\rho + \partial^{\alpha}\left( \rho |\nabla \varphi|^{2} \right) = \partial^{\alpha}\left( X(1 + \rho)\rho \right) + \left[ \partial^{\alpha}, \Delta \right]\rho,
\end{equation}
where $\left[ \partial^{\alpha}, \Delta \right]$ denotes the commutator $\partial^{\alpha}\Delta - \Delta \partial^{\alpha}$. Moreover, a straightforward computation shows that, pointwise, there holds
\begin{equation}\label{commutatorpointwise}
\left| [\partial^{\alpha}, \Delta] \rho \right| \leq C\sum\limits_{0 \leq j \leq |\alpha|}|\nabla^{j + 1} \rho |,
\end{equation}
where $C$ depends only on $n, |\alpha|$ and the $C^{|\alpha|+1}$-norm of $g$. Note that the right-hand side of the above inequality can be controlled with the help of \eqref{inductivegradmodulus}. Using this and the induction hypotheses in \eqref{rhopartialPDE}, we infer that 
\begin{equation}
|\Delta \partial^{\alpha}\rho|_{0; B_{3^{-3l-3}}} \leq C\left| \log\ep \right|.
\end{equation}
Since $|\partial^{\alpha}(1 - \rho)|_{0; B^{-3l-2}} \leq C\ep^{2}\left| \log\ep \right|$ by \eqref{inductivemodulus}, Lemma \ref{interiormulti} then yields
\begin{align}\label{rholplus1}
\nonumber &|\nabla \partial^{\alpha}\rho|_{0, 1/2; B_{r_{1}}} \leq C \ep^{a}\left| \log\ep \right|^{b}, \text{ for all }|\alpha| = l. \\
&\Longrightarrow |\nabla \rho|_{l, 1/2; B_{r_{1}}} \leq C\ep^{a}\left| \log\ep \right|^{b}.
\end{align}
Now, using \eqref{philplus1} and \eqref{rholplus1} in \eqref{phiPDE} yields
\begin{equation}
|\Delta \varphi |_{l, 1/2; B_{r_{1}}} \leq C\ep^{a}\left| \log \ep \right|^{b},
\end{equation}
from which we get that 
\begin{equation}\label{phiclplus1}
| \varphi |_{l+2, 1/2; B_{r_{2}}} \leq C \left( |\varphi |_{0; B_{1/3}} + |\Delta\varphi |_{l, 1/2; B_{r_{1}}} \right) \leq C \left| \log \ep \right|^{1/2}.
\end{equation}
In particular, this proves \eqref{inductivephase} with $l+1$ in place of $l$, since $r_{2} > 3^{-3l-4}$.

Next we prove \eqref{inductivemodulus} for $l + 1$. To that end, notice that by the equation \eqref{XPDE} and the induction hypotheses, we have
\begin{equation}\label{DeltaXcl}
|\Delta X |_{l; B_{3^{-3l-2}}} \leq C \ep^{-2} \left| \log\ep \right|.
\end{equation}
Hence the interior Schauder estimates yield
\begin{equation}\label{Xclplus1preliminary}
| X |_{l+1; B_{3^{-3l-3}}} \leq C \ep^{-2}\left| \log \ep \right|.
\end{equation}
Taking a multi-index $\alpha$ with $|\alpha| = l$ and applying $\partial^{\alpha}$ to both sides of \eqref{XPDE}, we see that
\begin{align}\label{DeltapartialalphaX}
| \Delta \partial^{\alpha}X |_{0; B_{3^{-3l-3}}} &\leq | \partial^{\alpha}\Delta X |_{0; B_{3^{-3l-3}}} + |[\partial^{\alpha}, \Delta]X|\\
\nonumber &\leq | \partial^{\alpha}\Delta X |_{0; B_{3^{-3l-3}}} + C | X |_{l+1; B_{3^{-3l-3}}}\\
\nonumber&\leq C \ep^{-2} \left| \log \ep \right|,
\end{align}
where the last inequality follows from \eqref{DeltaXcl} and \eqref{Xclplus1preliminary}, and the second inequality follows from \eqref{commutatorpointwise} with $X$ in place of $\rho$. Therefore, Lemma \ref{interiormulti} gives
\begin{align*}
|\nabla \partial^{\alpha} X |_{0; B_{r_{1}}}\leq C\left( |\partial^{\alpha}X |_{0; B_{3^{-3l-3}}} + |\partial^{\alpha}X |_{0; B_{3^{-3l-3}}}^{1/2} | \Delta\partial^{\alpha}X |_{0; B_{3^{-3l-3}}}^{1/2} \right)\leq C \ep^{-1}\left| \log\ep \right|.
\end{align*}
Since $\alpha$ is any multi-index with $|\alpha| = l$, we see that
\begin{equation}\label{Xclplus1almost}
|X |_{l + 1; B_{r_{1}}} \leq C \ep^{-1} \left| \log\ep \right|.
\end{equation}
Next we want to improve \eqref{Xclplus1almost} to get \eqref{inductivemodulus} for $l + 1$. Before that, let us observe that we can use \eqref{Xclplus1almost} to get a $C^{l+2}$-estimate on $X$, which will be useful for handling commutator terms in the subsequence argument. Specifically, looking back at \eqref{XPDE}, we see from \eqref{inductivegradmodulus}, \eqref{phiclplus1} and \eqref{Xclplus1almost} that
\begin{equation}\label{DeltaXclplus1}
|\Delta X |_{l + 1; B_{r_{2}}} \leq C \ep^{-3}\left| \log\ep \right|.
\end{equation}
Therefore, interior Schauder estimates give
\begin{equation}\label{Xclplus2}
|X |_{l + 2; B_{r_{3}}} \leq C \ep^{-3}\left| \log\ep \right|.
\end{equation}
Arguing as in \eqref{DeltapartialalphaX}, for $|\alpha| = l + 1$ we have
\begin{equation}\label{Deltapartialalphaplus1X}
|\Delta \partial^{\alpha}X|_{0; B_{r_{3}}} \leq C \ep^{-3}\left| \log\ep \right|.
\end{equation}
Then, as above, we infer from \eqref{Deltapartialalphaplus1X}, \eqref{Xclplus1almost} and Lemma \ref{interiormulti} that
\begin{equation}\label{Xclplus2almost}
|X|_{l + 2; B_{3^{-3l-4}}} \leq C \ep^{-2}\left| \log\ep \right|.
\end{equation}
We now proceed to improve \eqref{Xclplus1almost} to the desired estimate. To that end, we follow \cite{bbh} and rewrite \eqref{XPDE} as
\begin{equation}\label{XPDEalt}
-\ep^{2}\Delta X + 2X = \rho |\nabla \varphi |^{2} + 3\ep^{2}X^{2} - \ep^{4}X^{3} \equiv R.
\end{equation}
By \eqref{inductivemodulus}, \eqref{inductivegradmodulus}, \eqref{phiclplus1} and \eqref{Xclplus1almost}, we have $| R |_{l+1; B_{r_{2}}} \leq C \left| \log\ep \right|$.
Next, we take a multi-index $\alpha$ of length $l+1$ and apply it to both sides of the above equation, getting
\begin{equation}
-\ep^{2}\Delta \partial^{\alpha}X + 2\partial^{\alpha}X = \ep^{2} [\partial^{\alpha}, \Delta]X + \partial^{\alpha}R.
\end{equation}
By the $C^{l+1}$-bound on $R$ and the estimates \eqref{Xclplus2almost} and \eqref{commutatorpointwise}, we see that the right-hand side of the above equation is bounded pointwise on $B_{3^{-3l-4}}$ by $ C \left| \log\ep  \right|$. Therefore we may use Lemma \ref{bbhmax} to get
\begin{equation}\label{Xclplus1}
|\partial^{\alpha}X|_{0; B_{3^{-3l-5}}} \leq C \left( 1 + \ep^{-1} e^{-\frac{C}{\ep}} \right)\left| \log\ep \right|  \leq C \left| \log\ep \right|.
\end{equation}
Since $\alpha$ is any multi-index with length $l+1$, we conclude that 
\begin{equation}\label{Xlplusonefinal}
|X|_{l + 1; B_{3^{-3l-5}}} \leq C\left| \log\ep \right|;
\end{equation}
in other words, the estimate \eqref{inductivemodulus} holds for $l + 1$.

Finally, we prove \eqref{inductivegradmodulus} for $l + 1$. From \eqref{rhoPDE}, \eqref{phiclplus1} and \eqref{Xlplusonefinal}, we infer that $|\Delta \rho|_{l+1; B_{3^{-3l-5}}} \leq C \left| \log\ep \right|$, and hence 
\begin{equation}
|\rho|_{l + 2; B_{r_{4}}} \leq C\left| \log\ep \right|,
\end{equation}
where $3^{-3l-6} < r_{4} < 3^{-3l-5}$. Letting $\alpha$ be a multi-index of length $l + 1$ and applying $\partial^{\alpha}$ to both sides of \eqref{rhoPDE}, we see with the help of \eqref{commutatorpointwise} that 
\begin{equation}
|\Delta \partial^{\alpha}\rho|_{0; B_{r_{4}}} \leq C\left| \log\ep \right|
\end{equation}
Hence, recalling that $|\partial^{\alpha}\rho|_{0; B_{r_{4}}} \leq |\nabla \rho|_{l; B_{3^{-3l-3}}} \leq C\ep^{a}\left| \log\ep \right|^{b}$ and again using Lemma \ref{interiormulti}, we arrive at
\begin{equation}
|\nabla \partial^{\alpha} \rho |_{0; B_{3^{-3l-6}}} \leq C \ep^{a} \left| \log\ep \right|^{b},
\end{equation}
and the estimate \eqref{inductivegradmodulus} for $l+1$ follows immediately. By induction, we complete the proof of the Theorem.
\end{proof}
\begin{rmk} The structure of the above proof is borrowed from \cite{bbh}. However, in the setting of \cite{bbh} the authors were able to obtain estimates independent of $\ep$, whereas here we have to keep track of the $|\log\ep|$-factors.
\end{rmk}
\begin{thm}[Boundary estimates]\label{boundaryhigher}
Let $g = (g_{ij})$ be a smooth metric on $B_{1}^{+}$ satisfying $g_{nn} \equiv 1$ and $g_{in} \equiv 0$ for all $i = 1, \cdots , n-1$ and suppose that $u:(B_{1}^{+}, g) \to \DD$ is a solution to \eqref{halfGL} such that
\begin{equation}
\int_{B_{1}^{+}} e_{\ep}(u)\dvol_{g} < \eta \left| \log\ep \right|,
\end{equation}
for some $\eta < \eta_{3}'$ and $\ep < \ep_{3}'$, where $\eta_{3}', \ep_{3}'$ are as in Proposition \ref{boundaryetareg}. Then, letting $X$, $\rho$ and $\theta$ be as in Theorem \ref{interiorhigher}, we have
\begin{align}
|\nabla \varphi|_{1, 1/2; B_{3^{-4}}} &\leq C \left| \log\ep \right|^{1/2}, \label{boundaryphase}\\
|X|_{1, 0; B_{3^{-5}}} &\leq C \left| \log\ep  \right|, \label{boundarymodulus}\\
|\nabla \rho|_{1, 1/2; B_{3^{-6}}} & \leq C\ep^{a} \left| \log\ep \right|^{b}, \label{boundarygradmodulus}
\end{align}
where $C, a$ and $b$ are positive constants depending only on $n$ and the metric $g$.
\end{thm}
\begin{proof}
Below we will use $C, a, b$ to denote positive constants depending only on $n$ and $g$. We first note that, by our assumptions and Proposition \ref{boundaryetareg}, there holds $\rho \geq 3/4$ on $B_{3/4}^{+}$ and that
\begin{equation}\label{c1phi}
|\nabla \varphi|_{0; B_{1/3}^{+}} \leq C\left| \log\ep \right|^{1/2}.
\end{equation}
Then, as in the proof of Theorem \ref{interiorhigher}, we know that $X$ satisfies
\begin{equation*}
-\ep^{2}\Delta X + X \leq C\left| \log\ep \right| \text{ on }B_{1/3}^{+}.
\end{equation*}
Moreover, since $\partial_{\nu}X = 0$ on $T$, we may invoke Lemma \ref{bbhmaxboundary} to infer that 
\begin{equation*}
|X|_{0; B_{1/9}^{+}} \leq C\left| \log\ep \right|.
\end{equation*}
By the above two estimates, the equation \eqref{rhoPDE}, and the fact that $\partial_{\nu}\rho = 0$ on $T$, we can use Lemma \ref{boundarymulti} in place of \eqref{interiormulti} and argue as in Theorem \ref{interiorhigher} to get
\begin{equation*}
|\nabla \rho|_{0; B_{1/27}^{+}} \leq C\ep \left| \log\ep \right|.
\end{equation*}

To go from here to the desired estimates, we would follow the induction step in the previous proof with $l = 0$, except that we replace Lemma \ref{interiormulti} by Lemma \ref{boundarymulti} or Remark \ref{dirichletmulti}, and replace Lemma \ref{bbhmax} by Lemma \ref{bbhmaxboundary} or \ref{bbhmaxdirichlet}. These boundary versions would be applicable because by \eqref{halfGL}, $\varphi$, $\rho$, $X$, and $\partial_{i}X$ for $i = 1, \cdots, n-1$ all satisfy the homogeneous Neumann condition on $T$, while $\partial_{n}X$ satisfies the homogeneous Dirichlet condition on $T$ since it equals $\partial_{\nu}X$ by our assumptions on $g$. The details will be omitted to avoid repetition.
\end{proof}
\subsection{Convergence away from singular set}
Up to this point, we have mostly been working with solutions to the Ginzburg-Landau equation on a ball. We now return to the manifold setting. Recall that $M$ is isometrically embedded in a closed manifold $\widetilde{M}$ of the same dimension, which is isometrically embedded in $\RR^{N}$. Also, in this section and in Section 7, we will regard $\mu_{k}$, $\mu$ and any other measure we introduce as measures on $\RR^{N}$ supported on $\overline{M}$. For instance, $\mu(B^{N}_{r}(p))$ would be the same as $\mu(B^{N}_{r}(p) \cap \overline{M})$.

We begin with the following consequence of the monotonicity formula, Proposition \ref{lipschitzmonotonicity}. 
\begin{prop}\label{globalmono} There exist universal positive constants $r_{2}, \chi, C$, such that whenever $u: M \to \CC$ is a solution to \eqref{GLNeumann}, the following hold.
\begin{enumerate}
\item[(1)] For all $p \in M$, the function $\rho \mapsto e^{\chi\rho}\rho^{2-n}\mu_{k}(\tilde{B}_{\rho}(p))$ is increasing whenever $\rho \in (0, \min\{r_{2}, d(p, \partial M)\})$.
\item[(2)] For all $p \in \partial M$, the function $\rho \mapsto e^{\chi\rho}\rho^{2-n}\mu_{k}(\tilde{B}^{+}_{\rho}(p))$ is increasing for $\rho \in (0, r_{2})$.
\end{enumerate}
\end{prop}
\begin{proof}
For (1), we simply apply Proposition \ref{lipschitzmonotonicity} on small enough geodesic balls (cf. Remark \ref{fermirmk}(2)). For (2), we take Fermi coordinates at $p \in \partial M$, reflect the solution using Lemma \ref{GLreflection} (cf. Remark \ref{fermirmk}(1)), and then apply Proposition \ref{lipschitzmonotonicity}.
\end{proof}
As a consequence of the above proposition and \eqref{GLenergyupperbound}, we get the following uniform upper bound.
\begin{prop}\label{strongdensityupperbound}
Let $K_{0}$ be as in \eqref{GLenergyupperbound}. There exists universal constants $r_{3}, C > 0$ such that for all $p \in \overline{M}$, $\rho < r_{3}$ and for all $k$, we have $\rho^{2-n}\mu_{k}(B^{N}_{\rho}(p)) \leq CK_{0}$.
\end{prop}
\begin{proof}
Note that by Remark \ref{fermirmk}(2) there exists a universal $s_{1}$ such that for all $s < s_{1}$ we have
\begin{equation}\label{geoeuclidean}
B^{N}_{s/2}(p) \cap \widetilde{M} \subset \tilde{B}_{s}(p) \subset B^{N}_{2s}(p) \cap \widetilde{M}, \text{ if $p \in \widetilde{M}$; }
\end{equation}
\begin{equation}\label{fermieuclidean}
B^{N}_{s/2}(p) \cap \overline{M} \subset \tilde{B}^{+}_{s}(p) \subset B^{N}_{2s}(p) \cap \overline{M},\text{ if if $p \in \partial M$. }
\end{equation}
Note that by \eqref{geoeuclidean} we have 
\begin{equation}\label{distancerel}
d_{0}(p, \partial M)/2 \leq d(p, \partial M) \leq 2 d_{0}(p, \partial M),
\end{equation}
whenever $d < s_{1}/4$. We will take $r_{3} = \min\{s_{1}, r_{2}\}/64$, where $r_{2}$ is as in Proposition \ref{globalmono}. To continue, we distinguish four cases.

\noindent\textbf{Case 1:} If $p \in \partial M$, then since $\rho < r_{3}$, we have
\begin{align*}
\rho^{2-n}\mu_{k}(B^{N}_{\rho}(p)) &\leq 2^{n-2}(2\rho)^{2-n}\mu_{k}(\tilde{B}^{+}_{2\rho}(p)) \leq Cr_{2}^{2-n}\mu_{k}(\tilde{B}^{+}_{r_{2}}(p))\\& \leq Cr_{2}^{2-n}\mu_{k}(\overline{M}) \leq CK_{0}.
\end{align*}
In the second inequality we invoked Proposition \ref{globalmono}(2) since $2\rho < r_{2}$, and in the last inequality we absorbed $r_{2}$ into the universal constant $C$.
\vskip1mm
\noindent\textbf{Case 2:} If $p \in M$ and $d \equiv d_{g}(p, \partial M) \geq 8r_{3}$, then we have
\begin{align*}
\rho^{2-n}\mu_{k}(B^{N}_{\rho}(p)) &\leq \rho^{2-n}\mu_{k}(\tilde{B}_{2\rho}(p)) \leq Cr_{3}^{2-n}\mu_{k}(\tilde{B}_{2r_{3}}(p))\\
&\leq Cr_{3}^{2-n}\mu_{k}(\overline{M}) \leq CK_{0}.
\end{align*}
In the second inequality we used Proposition \ref{globalmono}(1) since $2\rho < 2r_{3} < \min\{d, r_{2}\}$.
\vskip 1mm
\noindent\textbf{Case 3:} If $p \in M$ and $r_{3} > \rho \geq d/8$, then we may use \eqref{distancerel} to infer that $d_{0} \leq 2d < 16r_{3}$ and take $q \in B^{N}_{d_{0}}(p) \cap \partial M$ to estimate
\begin{align*}
\rho^{2-n}\mu_{k}(B_{\rho}^{N}(p)) & \leq \rho^{2-n}\mu_{k}(B^{N}_{\rho + d_{0}}(q)) \leq \rho^{2-n}\mu_{k}(\tilde{B}^{+}_{2\rho + 2d_{0}}(q))\\
&\leq C(\rho + d_{0})^{n-2}\rho^{2-n}r_{2}^{2-n}\mu_{k}(\tilde{B}^{+}_{r_{2}}(q)) \leq Cr_{2}^{2-n}\mu_{k}(\overline{M})\\
&\leq CK_{0}.
\end{align*}
Note that we used Proposition \ref{globalmono}(2) in the third inequality, since $2\rho + 2d_{0} \leq 34r_{3} < r_{2}$. Also, in the fourth inequality we used $d_{0} \leq 2d \leq 16\rho$ to bound $(\rho + d_{0})/\rho$.
\vskip 1mm
\noindent\textbf{Case 4:} If $p \in M$ and $\rho < d/8 < r_{3}$, then since $d < 8r_{3} < s_{1}/4$, we may use \eqref{distancerel} to deduce that $2\rho \leq d/4 \leq d_{0}/2 \leq d$. We now estimate
\begin{align*}
\rho^{2-n}\mu_{k}(B_{\rho}^{N}(p)) & \leq \rho^{2-n}\mu_{k}(\tilde{B}_{2\rho}(p)) \leq Cd_{0}^{2-n}\mu_{k}(\tilde{B}_{d_{0}/2}(p)) \leq Cd_{0}^{2-n}\mu_{k}(B^{N}_{d_{0}}(p)).
\end{align*}
Note that Proposition \ref{globalmono}(1) is used in the second inequality. The last term can then be estimated as in Case 3.
\end{proof}
\begin{rmk}\label{localizeduppersemicont} The above argument can be localized to show that for all $k \in \NN$, $p \in \partial M$, $r < \min\{r_{2}, s_{1}/4\}$ and $q \in B^{N}_{r/4}(p)\cap \overline{M}$ we have 
\begin{equation}
\rho^{2-n}\mu_{k}(B^{N}_{\rho}(q)) \leq Cr^{2-n}\mu_{k}(B^{N}_{r}(p)),
\end{equation}
where we require that $\rho <  \min\{r_{3}, r/200\}$.
\end{rmk}
Going back to Proposition \ref{globalmono} and letting $k$ tend to infinity, one can easily show that the limit measure $\mu$ has the same monotonicity properties as does each $\mu_{k}$. Moreover, since both the geodesic distance (at interior points) and the Fermi distance (at boundary points) eventually agree with the Euclidean distance on smaller and smaller balls, we infer that the density $\Theta(\mu, \cdot)$ defined below exists everywhere on $\overline{M}$ and we could use either geodesic/Fermi balls or Euclidean balls in the definition.
\begin{equation}\label{densityrel}
\Theta(\mu, p) \equiv \lim\limits_{r \to 0}r^{2-n}\mu(B^{N}_{r}(p))
 = \left\{ 
 \begin{array}{cl}
 \lim\limits_{r \to 0}r^{2-n}\mu(\tilde{B}_{r}(p)) & \text{ if }p \in M,\\
 \lim\limits_{r \to 0}r^{2-n}\mu(\tilde{B}^{+}_{r}(p)) & \text{ if }p \in \partial M.
 \end{array}
 \right.
\end{equation}
The singular set $\Sigma$ is then defined as the set of points where $\mu$ has positive density, i.e.
\begin{equation*}
\Sigma \equiv \{ p \in \overline{M} |\ \Theta(\mu, p) > 0 \}.
\end{equation*}
\begin{prop}\label{sigmaclosed}
The set $\Sigma$ is closed.
\end{prop}
\begin{proof}
We will show that $\overline{M}\setminus \Sigma$ is relatively open. To that end, suppose that $\Theta(\mu, p) = 0$ at some $p \in \overline{M}$. If $p \in \partial M$, then we may fix a small enough radius $r < \min\{r_{2}, s_{1}\}$ such that for all sufficiently large $k$ there holds
\begin{equation}
r^{2-n}\mu_{k}(B_{r}^{N}(p)) < \min\{\eta_{3}, \eta_{3}'\}/(2^{n-2}C),
\end{equation}
where $\eta_{3}, \eta_{3}'$ are from Propositions \ref{interioretareg}, \ref{boundaryetareg}, respectively, and $C$ is the constant from Remark \ref{localizeduppersemicont}. By Remark \ref{localizeduppersemicont} and \eqref{geoeuclidean}, for each $q \in B_{r/4}^{N}(p) \cap M$ we may choose $\rho < \min\{r_{3}, r/200, d_{0}(q, \partial M)\}$ and apply Proposition \ref{interioretareg} to the geodesic ball $\tilde{B}_{\rho/2}(q)$ to get
\begin{equation*}
\sup\limits_{\tilde{B}_{\rho/4}(q)}e_{\ep_{k}}(u_{k}) \leq C\rho^{-n}\int_{\tilde{B}_{\rho/2}(q)}e_{\ep_{k}}(u_{k})\dvol \leq C\rho^{-2}\eta_{3}'\left| \log\ep_{k} \right|.
\end{equation*}
Hence for all $s < \rho/2$, we have
\begin{equation*}
s^{2-n}\mu_{k}(\tilde{B}_{s}(q)) \leq Cs^{2}\rho^{-2}.
\end{equation*}
Passing to the limit as $k \to \infty$, we infer that $\Theta(\mu, q) = 0$. Similarly, for each $q \in B_{r/4}^{N}(p) \cap \partial M$ we may apply Proposition \ref{boundaryetareg} to a small enough Fermi ball around $q$ to conclude that $\Theta(\mu, q) = 0$. Therefore $\Theta(\mu, \cdot)$ vanishes in a relative neighborhood of $p$. If $p \in M$, the argument is similar (and in fact simpler), and we also get that $\Theta(\mu, \cdot)$ vanishes near $p$. 
\end{proof}
\begin{rmk}\label{finitemeasure} The previous proof actually shows that if $\Theta(\mu, p) < \min\{\eta_{3}, \eta_{3}'\}/2^{n-2}C \equiv \underline{\eta}$, then $\Theta(\mu, p) = 0$ and $p \notin \Sigma$. Therefore we infer that
\begin{equation}\label{lowerdensityboundpointwise}
\underline{\eta} \leq \Theta(\mu, p)\text{ for all }p \in \Sigma.
\end{equation}
Combining this density lower bound and the assumption \eqref{GLenergyupperbound} with a standard covering argument, we get that $H^{n-2}(\Sigma) < \infty$.
\end{rmk}
In Section 7 we will need a slightly more quantitative version of the lower density bound in the previous remark.
\begin{prop}\label{densitybounds}
Let $\underline{\eta}$ be as in Remark \ref{finitemeasure}. There exists constants $r_{4}$ and $C$ such that
\begin{equation}\label{densityinequalities}
5^{-n}\underline{\eta} \leq r^{2-n}\mu(B_{r}^{N}(p)) \text{ for all $p \in \Sigma$ and $r < r_{4}$.}
\end{equation}
\end{prop}
\begin{proof}
We argue by contradiction. Negating the statement we obtain sequences $r_{i} \to 0$ and $p_{i} \in \Sigma$ such that 
\begin{equation}\label{lowerboundnegation}
\mu(B_{r_{i}}^{N}(p_{i})) < 5^{-n}\underline{\eta}r_{i}^{n-2}.
\end{equation} 
Next, note that if
\begin{equation*}
\limsup\limits_{i \to \infty}r_{i}^{-1}d(p_{i}, \partial M) > 1/5, 
\end{equation*}
then along a subsequence, which we do not relabel, we may use monotonicity and \eqref{lowerdensityboundpointwise} to estimate
\begin{align*}
5^{-n}\underline{\eta} &\geq r_{i}^{2-n}\mu(B_{r_{i}}^{N}(p_{i})) \geq 5^{2-n}(r_{i}/5)^{2-n}\mu(\tilde{B}_{r_{i}/5}(p_{i}))\\
&\geq 5^{2-n}e^{-\chi(r_{i}/5)}\underline\eta,
\end{align*}
which is a contradiction for sufficiently large $i$. Hence in what follows we may assume that
\begin{equation*}
\limsup\limits_{i \to \infty}r_{i}^{-1}d(p_{i}, \partial M) \leq 1/5.
\end{equation*}
Then eventually, by \eqref{distancerel}, we have $d_{0}(p_{i}, \partial M) \leq 2d(p_{i}, \partial M) < r_{i}/2$, which allows us to choose $q_{i} \in B_{r_{i}/2}^{N}(p_{i}) \cap \partial M$. Note that we then have
\begin{equation}
\tilde{B}^{+}_{r_{i}/4}(q_{i}) \subset B^{N}_{r_{i}/2}(q_{i}) \subset B^{N}_{r_{i}}(p_{i}).
\end{equation}
Consequently we apply Proposition \ref{globalmono}(2) to estimate
\begin{align*}
5^{-n}\underline\eta & \geq r_{i}^{2-n}\mu(B^{N}_{r_{i}}(p_{i})) \geq 4^{2-n}(r_{i}/4)^{2-n}\mu(\tilde{B}^{+}_{r_{i}/4}(q_{i}))\\
&\geq 4^{2-n}e^{-\chi(r_{i}/4)}\underline\eta,
\end{align*}
again a contradiction for $i$ large. 
\end{proof}

The remaining properties of $\Sigma$ will be established in the next section. We now define the harmonic $1$-form $\psi$ in the conclusion of Theorem \ref{GLgeneralconvergence} and prove its asserted properties.
\begin{thm}\label{regpartconvergence}
Let $\rho_{k} = |u_{k}|$, $X_{k} = \ep_{k}^{-2}(1 - \rho_{k})$ and let $\psi_{k} = \left| \log\ep_{k} \right|^{-1/2}u_{k}\times du_{k}$. Then, passing to a subsequence if necessary, the following hold.
\begin{enumerate}
\item[(1)] $\nabla \rho_{k} \to 0$ and $\ep_{k}X_{k}\to 0$ in $C_{\loc}^{m}(M\setminus\Sigma)\cap C_{\loc}^{1}(\overline{M}\setminus \Sigma)$ for all $m$.
\item[(2)] $\psi_{k}$ converges in $C^{m}_{\loc}(M\setminus \Sigma) \cap C^{1}_{\loc}(\overline{M}\setminus\Sigma)$ for all $m$. 
\item[(3)] Denoting the limit $1$-form in part (2) by $\psi$, then 
\begin{equation*}
\lim_{k \to \infty}\mu_{k}\llcorner K = (|\psi|^{2}/2)\dvol\ \llcorner K
\end{equation*} 
for all compact set $K \subset \overline{M}\setminus\Sigma$. In particular, $\int_{M}|\psi|^{2}\dvol \leq 2K_{0} < \infty.$

\item[(4)] $\psi$ is smooth and harmonic over all of $\overline{M}$, and satisfies $\fn \psi = 0$ on $\partial M$.
\end{enumerate}
\end{thm}
\begin{proof}

The proof will be based on the estimates obtained in Section 6.2. To begin, take $p \in M \setminus \Sigma$. Then, by the definition of $\Sigma$, we can fix a small enough radius $r = r_{p}$ such that for $k$ sufficiently large, there holds 
\begin{equation*}
r^{2-n}\int_{B_{r}(p)}e_{\ep_{k}}(u_{k}) < \eta_{3}\left| \log\frac{\ep_{k}}{r} \right|.
\end{equation*}
But then, writing $u_{k}$ locally in polar form as $\rho_{k}e^{i\varphi_{k}}$ with the requirement that $\fint_{B_{r}(p)} \varphi_{k}\dvol \in [0, 2\pi)$ and recalling that $X_{k} = \ep_{k}^{-2}(1 - \rho_{k})$, we may apply Theorem \ref{interiorhigher} (suitably scaled) and pass to subsequences to arrange the following:
\begin{enumerate}
\item[(i)] $|\log\ep_{k}|^{-1/2}d\varphi_{k}$ converges in $C^{l}(\tilde{B}_{3^{-3l-4}r}(p))$ for each $l$.
\item[(ii)] $\left|\nabla\rho_{k}\right| $ and $\ep_{k}X_{k}$ converge in $C^{l}(\tilde{B}_{3^{-3l-3}r}(p))$ to zero for all $l$.
\end{enumerate}
 Next we notice that since $u_{k} = \rho_{k}e^{i\varphi_{k}}$, we have 
 \begin{align*}
du_{k} &= e^{i\varphi_{k}}(d\rho_{k} + i\rho_{k}d\varphi_{k});\\
 \left| \log\ep_{k} \right|^{1/2}\psi_{k} &= u_{k} \times du_{k} = \rho_{k}^{2}d\varphi_{k} = d\varphi_{k} - \ep_{k}^{2}X_{k}d\varphi_{k}.
\end{align*}
Hence by (i) and (ii) above, $\psi_{k}$ converges in $C^{l}(\tilde{B}_{3^{-3l-4}r}(p))$ to some limit $\psi$. Moreover, noting that $d^{\ast}(u_{k}\times du_{k}) = 0$ by \eqref{GLNeumann}, and that
\begin{align*}
d\psi_{k} &= 2|\log\ep_{k}|^{-1/2}\rho_{k}d\rho_{k}\wedge d\varphi_{k} \to 0 \text{ in }C^{0}(\tilde{B}_{3^{-3}r(p)}),
\end{align*}
we conclude that $d\psi = 0$ and $d^{\ast}\psi = 0$ pointwise on $\tilde{B}_{3^{-7}r}(p)$. Therefore $\psi$ is harmonic and smooth in the same ball. 

Next, consider $p \in \partial M \setminus \Sigma$. Then we can choose $r$ such that, eventually, 
\begin{equation*}
r^{2-n}\int_{\tilde{B}^{+}_{r}(p)}e_{\ep_{k}}(u_{k}) < \eta_{3}'\left| \log\frac{\ep_{k}}{r} \right|.
\end{equation*}
Again writing $u_{k} = \rho_{k}e^{i\varphi_{k}}$ on $\tilde{B}^{+}_{r}(p)$, but using Theorem \ref{boundaryhigher} in place of Theorem \ref{interiorhigher}, we may argue as above and conclude that $|\nabla \rho_{k}|$ and $\ep_{k}X_{k}$ converge in $C^{1}$ to zero in $\tilde{B}^{+}_{3^{-6}r}(p)$, and that $\psi_{k}$ converges in $C^{1}(\tilde{B}^{+}_{3^{-6}r}(p))$ to a limit $1$-form $\psi$ satisfying that $d\psi = 0$ and $d^{\ast}\psi = 0$. Moreover, by \eqref{GLNeumann}, we deduce that $\fn \psi  = 0$ on $\tilde{B}^{+}_{3^{-6}r}(p) \cap \partial M$. Therefore $\psi$ is harmonic in $\tilde{B}^{+}_{3^{-6}r}(p)$ and smooth up to $\tilde{B}^{+}_{3^{-6}r}(p)\cap \partial M$.

From these local results, we infer after possibly passing to further subsequences that
\begin{enumerate}
\item[(i)'] $\psi_{k}$ converges in $C^{m}_{\loc}(M\setminus \Sigma) \cap C_{\loc}^{1}(\overline{M}\setminus \Sigma)$ for all $m$ to a $1$-form $\psi \in C^{\infty}(\overline{M}\setminus \Sigma)$ which is harmonic in $M\setminus \Sigma$ and satisfies $\fn\psi = 0$ on $\partial M \setminus \Sigma$.
\item[(ii)'] $\nabla \rho_{k}$ and $\ep_{k}X_{k}$ converge in $C^{m}_{\loc}(M\setminus\Sigma) \cap C_{\loc}^{1}(\overline{M}\setminus\Sigma)$ to zero for all $m$.
\end{enumerate}
Note that properties (i)' and (ii)' establish conclusion (1) and (2) of Theorem \ref{regpartconvergence} and almost gives assertion (4), except that $\psi$ is not yet shown to be smooth across $\Sigma$. To prove that, we need assertion (3), which we now turn to. 

Recalling \eqref{gradientdecomposition}, we observe that $e_{\ep_{k}}(u_{k})$ can be decomposed as follows:
\begin{align}\label{densitydecomposition}
e_{\ep_{k}}(u_{k}) &= \frac{|u_{k}\times du_{k}|^{2}}{2} + \frac{\left| \nabla |u_{k}|^{2} \right|^{2} }{8} + \frac{(1 - |u_{k}|^{2}) |\nabla u_{k}|^{2} }{2} + \frac{(1 - |u_{k}|^{2})^{2}}{4\ep_{k}^{2}}\\\nonumber
&= \left| \log\ep_{k} \right|\frac{|\psi_{k}|^{2}}{2} + \frac{\rho_{k}^{2} \left| \nabla \rho_{k} \right|^{2} }{2} + \ep_{k}^{2}X_{k}\frac{|\nabla u_{k}|^{2}}{2} + \ep_{k}^{2}\frac{X_{k}^{2}}{4}.
\end{align}
Dividing through by $|\log\ep_{k}|$ and using (i)' and (ii)' above, we see that $\lim_{k\to \infty}\mu_{k} \llcorner K = (|\psi|^{2}/2)\dvol\ \llcorner K$ in the sense of measures for each compact set $K \subset \overline{M}\setminus \Sigma$.
Consequently,
\begin{equation}\label{psil2finite}
\int_{K}\frac{|\psi|^{2}}{2}\dvol \leq \limsup\limits_{k \to \infty}\mu_{k}(\overline{M}),\text{ for all }K \subset\subset \overline{M}\setminus\Sigma,
\end{equation}
which implies the second assertion of conclusion (3) since $\Sigma$ has zero measure with respect to $\dvol$ by Remark \ref{finitemeasure}.

Finally we return to (4). With the help of (i)' above, we will show that $\psi$ agrees with the harmonic part in its Hodge decomposition. Specifically, by Theorem 7.7.7 of \cite{morreybook}, we can write
\begin{equation}\label{psihodge}
\psi = d\alpha + d^{\ast}\beta + h,
\end{equation}
where $\alpha \in W^{1, 2}(M)$, $\beta \in W^{1, 2}_{2, \fn}(M)$ and $h \in \cH_{1, \fn}$. To proceed, we note the following.
\begin{equation}\label{psidclosed}
\int_{M}\langle \psi, d^{\ast}\xi \rangle\dvol =0 \text{ for all }\xi \in W^{1, 2}_{2, \fn}(M).
\end{equation}
\begin{equation}\label{alphaharmonic}
\int_{M}\langle d\alpha, dv \rangle\dvol = 0\text{ for all }v \in W^{1, 2}(M).
\end{equation}
We assume these for now and give their proofs at the end. Substituting $\xi = \beta$ into \eqref{psidclosed}, we immediately see that $d^{\ast}\beta = 0$. On the other hand, putting $v = \alpha$ in \eqref{alphaharmonic} shows that $d\alpha = 0$. Recalling \eqref{psihodge}, we see that $\psi$ must be equal to $h$, which establishes assertion (4) and finishes the proof of Theorem \ref{regpartconvergence}.
\end{proof}

\begin{proof}[Proof of \eqref{psidclosed}]
Take a sequence $\{W_{k}\}$ of open subsets of $\widetilde{M}$ such that $W_{1} \supset W_{2} \supset \cdots \supset \Sigma = \cap_{k = 1}^{\infty}W_{k}$. Since $\Sigma$ has finite $H^{n-2}$-measure and hence vanishing $2$-capacity (cf. \cite{evansgariepy}, Section 4.7), we may obtain as in Theorem 3, Section 4.7.2 of \cite{evansgariepy} a sequence of cut-off functions $\{\zeta_{k}\}$ such that $\supp \zeta_{k} \subset\subset W_{k}$, $\zeta_{k}\equiv 1$ on $W_{k  + 1}$, and that 
\begin{equation}\label{cap2vanish}
\int_{W_{k}}|\nabla\zeta_{k}|^{2}\dvol < 1/k.
\end{equation}
Recalling that $d\psi = 0$ away from $\Sigma$, we see that for each $k$ and for all smooth $2$-form $\xi$ satisfying $\fn \xi = 0$ on $\partial M$, we may integrate by parts and obtain
\begin{align*}
0 &= \int_{M}\langle d\psi, (1 - \zeta_{k})\xi \rangle \dvol  = \int_{M}\langle \psi, d^{\ast}\left( (1 - \zeta_{k})\xi \right) \rangle\dvol \\
&= \int_{M}\langle \psi, (1 - \zeta_{k})d^{\ast}\xi \rangle\dvol - \int_{M}\langle \psi, \xi \llcorner d\zeta_{k} \rangle\dvol\\
&= I - II.
\end{align*}
By construction of $\{W_{k}\}$ and $\{\zeta_{k}\}$ and the finiteness of $\|\psi\|_{2; M}$, we deduce that (I) tends to $\int_{M}\langle \psi, d^{\ast}\xi \rangle\dvol$ while (II) tends to zero as $k$ tends to infinity. Hence we obtain \eqref{psidclosed} for all smooth $\xi$ with $\fn\xi=  0$. Since $\psi$ has finite $L^{2}$-norm, we may extend the identity to all $\xi$ lying in $W^{1, 2}_{2,\fn}$ by approximation.
\end{proof}
\begin{proof}[Proof of \eqref{alphaharmonic}]
As mentioned in the proof of Theorem \ref{regpartconvergence}, \eqref{psidclosed} implies that $d^{\ast}\beta = 0$ and thus $d\alpha = \psi - h$. Next, letting $v \in C^{1}_{c}(\overline{M}\setminus\Sigma)$ and integrating by parts using the boundary condition on $\psi$ and $h$ and the fact that they are $d^{\ast}$-closed away from $\Sigma$, we get
\begin{equation*}
\int_{M}\langle d\alpha, dv \rangle\dvol = \int_{M\setminus\Sigma} \langle \psi, dv \rangle\dvol - \int_{M}\langle h, dv \rangle\dvol = 0.
\end{equation*}
Now, by a capacity argument similar to the one in the previous proof, we may extend the above to any $v \in C^{1}(\overline{M})$, and then to any $v \in W^{1, 2}(M)$ by approximation. 
\end{proof}
\section{Convergence of Ginzburg-Landau solutions II: The singular part}

The goal of this section is to derive further properties of $\Sigma$ and show that it supports a rectifiable $(n-2)$-varifold which is stationary with free boundary. We remark here that our strategy for proving rectifiability and stationarity is similar to that in \cite{bospara}, Part II (see also \cite{bosring}, Appendix B). 

We begin by noting that, by (2) of Theorem \ref{regpartconvergence}, if we define
\begin{equation}\label{singpart}
\mu_{s} \equiv \mu - \frac{|\psi|^{2}}{2}\dvol\ \llcorner \overline{M},
\end{equation}
then $\mu_{s}$ is a nonnegative Radon measure with $\supp(\mu_{s}) \subset \Sigma$. Moreover, since $\psi$ is smooth on $\overline{M}$, we see that everywhere on $\Sigma$, the density $\Theta(\mu_{s}, p)$ exists and coincides with $\Theta(\mu, p)$. Hence, by Propositions \ref{strongdensityupperbound} and \ref{densitybounds}, decreasing $\underline{\eta}, r_{3}$ and $r_{4}$ if necessary, we have
\begin{equation}\label{singulardensityinequalities}
5^{-n}\underline{\eta}r^{n-2} \leq \mu_{s}(B_{r}^{N}(p)) \leq CK_{0} r^{n-2},
\end{equation}
for all $p \in \Sigma$ and $r < \min\{r_{3}, r_{4}\}$. In particular, we infer that $\mu_{s}$ is absolutely continuous with respect to $H^{n-2}\llcorner \Sigma$, and therefore we may also write
\begin{equation*}\mu_{s} = \Theta(\mu_{s}, \cdot)H^{n-2}\llcorner \Sigma.,
\end{equation*}
where $\Theta(\mu_{s}, p) \in [\underline{\eta}, CE_{0}]$ for all $p$ in $\Sigma$. 

It follows from the density upper and lower bounds and Preiss' theorem (\cite{preiss}) applied to $\mu_{s}$ that $\Sigma\ ( = \supp(\mu_{s}))$ is a countably $(n-2)$-rectifiable set in $\RR^{N}$. Thus, to complete the proof of Theorem \ref{GLgeneralconvergence}, it remains to verify conclusion (3), i.e. to show that the rectifiable varifold $V = (\Sigma, \Theta(\mu_{s}, \cdot))$ is stationary with free boundary.

The starting point is the following identity.
\begin{equation}\label{ukstationary}
\int_{M}e_{\ep_{k}}(u_{k})\Div X - \langle \nabla_{\nabla u_{k}}X, \nabla u_{k} \rangle\dvol = 0,
\end{equation}
where $X$ is any admissible vector field. To prove \eqref{ukstationary}, we simply substitute $\zeta = \langle X, \nabla u_{k} \rangle$ into \eqref{GLvariation} (with $u_{k}$ in place of $u$) and integrate by parts, noting that the boundary term vanishes since $\langle X,\nu \rangle = 0$ on $\partial M$.

\begin{rmk}\label{stationarynotation} Although $X$ vector field, we would sometimes, by abuse of notation, identify it with its dual $1$-form with respect to $g$. For instance, we may write $\langle \nabla_{\nabla u_{k}}X, \nabla u_{k} \rangle$ as $\langle \nabla X, du_{k}\otimes du_{k} \rangle$ etc. Finally, by $du_{k}\otimes du_{k}$ we would mean, in terms of coordinates, the following tensor.
\begin{equation*}
(du_{k}\otimes du_{k})(\partial_{r}, \partial_{s})  = \partial_{r}u_{k}\cdot \partial_{s}u_{k} = \text{Re}\left( \partial_{r}u \ \overline{\partial_{s}u} \right).
\end{equation*}

\end{rmk}
We may pass to the limit in \eqref{ukstationary} as $k$ tends to infinity to obtain useful information on the limit measure. For convenience we will do it locally. Specifically, for $p \in \overline{M}$ we fix a small ball $B_{p} \equiv B^{N}_{r}(p)$ and introduce an orthonormal frame $\{e_{i}\}_{i = 1}^{n}$ for $T\widetilde{M}$, with the additional requirement that $e_{n} = \nu$ in $B_{p} \cap \overline{M}$ if $p$ lies on $\partial M$, where $\nu = \nabla \dist(\cdot, \partial M)$. Then we have the following result.
\begin{lemm}
There exists Radon measures $\{\alpha_{ij}\}$ supported on $\Sigma$ such that for all admissible vector fields $X$ compactly supported in $B_{p}$, we have
\begin{equation}\label{limitstationary}
\int_{M}\frac{|\psi|^{2}}{2}\Div X - \langle \nabla X, \psi \otimes \psi \rangle\dvol + \int_{\Sigma}\Div X d\mu_{s} - \int_{\Sigma}\langle \nabla_{e_{i}}X, e_{j} \rangle d\alpha_{ij} = 0.
\end{equation}
\end{lemm}
\begin{proof}
By \eqref{singpart} we see that for all $X$ admissible, 
\begin{equation}\label{easystationarylimit}
\lim\limits_{k \to \infty}\int_{M} e_{\ep_{k}}(u_{k})\Div X\dvol = \int_{M}\frac{|\psi|}{2}\Div X \dvol + \int_{\Sigma}\Div X d\mu_{s}.
\end{equation}
As for the second term on the left of \eqref{ukstationary}, recalling that away from $\Sigma$ we may locally write $u_{k} = \rho_{k}e^{i\varphi_{k}}$ and find that
\begin{align*}
du_{k}\otimes du_{k} &= d\rho_{k}\otimes d\rho_{k} + \rho_{k}^{2}d\varphi_{k}\otimes d\varphi_{k}\\
&= d\rho_{k}\otimes d\rho_{k} + \left| \log\ep_{k} \right| \psi_{k} \otimes \psi_{k} + \rho_{k}^{2}(1 - \rho_{k}^{2})d\varphi_{k}\otimes d\varphi_{k}\\
&= d\rho_{k}\otimes d\rho_{k} + \left| \log\ep_{k} \right| \psi_{k}\otimes \psi_{k} +\ep_{k}^{2} \rho_{k}^{2}(1 + \rho_{k})X_{k}d\varphi_{k}\otimes d\varphi_{k}.
\end{align*}
By Theorem \ref{regpartconvergence}, we see that $|\log\ep_{k}|^{-1}du_{k}\otimes du_{k}$ converges locally uniformly on $\overline{M}\setminus \Sigma$ to $\psi \otimes \psi$. Therefore there exists Radon measures $\{\alpha_{ij}\}$ supported on $\Sigma$ such that the following holds in the sense of measures on $B_{p}$.
\begin{equation}\label{stressenergylimit}
\mu_{ij}\equiv \lim\limits_{k \to\infty}\frac{\nabla_{i}u \cdot \nabla_{j}u}{\left|\log\ep_{k}\right|} \dvol\ \llcorner \overline{M}= \psi_{i}\psi_{j} \dvol\ \llcorner \overline{M} + \alpha_{ij}.
\end{equation}
It now follows from \eqref{stressenergylimit} that for all admissible $X$ compactly supported in $B_{p}$,
\begin{equation}\label{hardstationarylimit}
\lim\limits_{k \to \infty}\int_{M} \langle \nabla X, du_{k}\otimes du_{k} \rangle\dvol = \int_{M}\langle X, \psi\otimes \psi \rangle\dvol + \int_{\Sigma}\langle \nabla_{e_{i}}X, e_{j} \rangle d\alpha_{ij}.
\end{equation}
Combining \eqref{easystationarylimit} and \eqref{hardstationarylimit}, we are done.
\end{proof}
We next show that the first integral in \eqref{limitstationary} actually vanishes for all admissible $X$.
\begin{lemm}\label{regpartstationary}
For all admissible vector field $X$, we have
\begin{equation}\label{psistationary}
\int_{M}\frac{|\psi|^{2}}{2}\Div X - \langle \nabla X, \psi \otimes \psi \rangle \dvol = 0.
\end{equation}
\end{lemm}
\begin{proof}
Take an admissible vector field $X$ and let $\{\varphi_{t}: \widetilde{M} \to \widetilde{M}\}$ be the flow generated by $X$ for $t$ in a neighborhood of $0$. Note that $\varphi_{t}$ preserves $\partial M$ and $\overline{M}$ since $\langle X, \nu \rangle = 0$ on $\partial M$. 

Next, we apply the homotopy formula for differential forms (\cite{gms}, p. 135), to see that each pullback $\varphi_{t}^{\ast}\psi$ is cohomologous to $\psi$, i.e. for each $t$ there exists a $C^{1}$-function $\alpha_{t}$ such that $\varphi_{t}^{\ast}\psi = \psi + d\alpha_{t}$. Then, we compute
\begin{align*}
\int_{M}|\varphi_{t}^{\ast}\psi|^{2} &= \int_{M} |\psi|^{2} + |d\alpha_{t}|^{2}\dvol + 2\int_{M}\langle \psi, d\alpha_{t} \rangle\dvol \geq \int_{M} |\psi|^{2}\dvol.
\end{align*}
Note that the integral of $\langle \psi, d\alpha_{t} \rangle$ vanishes because we can integrate by parts and use the fact that $\fn\psi = 0$ on $\partial M$ and $d^{\ast}\psi = 0$ in $M$. Since the above inequality holds for all $t$, we get that 
\begin{align}\label{liederivative}
0 &= \frac{1}{2}\frac{d}{dt}\int_{M}|\varphi_{t}^{\ast}|\dvol = \int_{M}\langle \psi, \cL_{X}\psi \rangle\dvol \\\nonumber
&= \int_{M}\langle \psi, d(\psi(X)) \rangle\dvol \text{ ($\cL_{X} = d\iota_{X} + \iota_{X}d$ by Cartan's formula). }
\end{align}
A closer look at the integrand in the second line yields
\begin{align*}
\langle \psi, d(\psi(X)) \rangle &= \nabla_{\psi^{\#}} \langle \psi^{\#}, X \rangle\\
&= \langle \nabla_{\psi^{\#}}\psi^{\#}, X \rangle + \langle \psi^{\#}, \nabla_{\psi^{\#}}X \rangle\\
&= \langle \nabla_{X}\psi^{\#}, \psi^{\#} \rangle + \langle X, \psi\otimes \psi \rangle\\
&= \langle X, \nabla(|\psi^{2}|/2) \rangle + \langle X, \psi\otimes \psi  \rangle,
\end{align*}
where for the third equality we used the fact that $d\psi = 0$. Putting this into \eqref{liederivative} and integrating the first term by parts (no boundary term appears since $\langle  X, \nu\rangle = 0$ on $\partial M$), we get the desired identity, \eqref{psistationary}.
\end{proof}
Combining the previous two Lemmas, we get that for all $p \in \Sigma$ and admissible $X$ compactly supported in $B_{p}$, there holds
\begin{equation}\label{cleanstationary}
\int_{\Sigma} \Div Xd\mu_{s} - \int_{\Sigma}\langle \nabla_{e_{i}}X, e_{j} \rangle d\alpha_{ij} = 0.
\end{equation}
Moreover, as in the case for $\mu_{s}$, it's quite straightforward from the definition of $\alpha_{ij}$ and Proposition \ref{densitybounds} that $\alpha_{ij}$ coincides with $\mu_{ij}\ \llcorner \Sigma$ and that it is absolutely continuous with respect to $H^{n-2}\llcorner \Sigma$. Therefore, since $\Theta(\mu_{s},\cdot)$ is always positive on $\Sigma$, we may write 
\begin{equation*}
\alpha_{ij} = A_{ij}(\cdot)\Theta(\mu_{s}, \cdot)H^{n-2}\llcorner\Sigma\ (\ = A_{ij}(\cdot)\mu_{s}\ ).
\end{equation*}
Consequently, \eqref{cleanstationary} becomes 
\begin{equation}\label{almoststationary}
\int_{\Sigma} \langle \nabla_{e_{i}(p)}X(p), e_{j}(p) \rangle (\delta_{ij} - A_{ij}(p))\Theta(\mu_{s}, p)dH^{n-2}(p) = 0.
\end{equation}
Note that, by definition, the matrix $(A_{ij}(p))$ is symmetric for $H^{n-2}$-almost every $p$. The proof of Theorem \ref{GLgeneralconvergence}(3) would be complete if we can show that $I - A(p)$ projects onto $T_{p}\Sigma$ orthogonally for $H^{n-2}$-a.e. $p \in \Sigma$. To see that, we need the two results below.
\begin{lemm}\label{Aproperties}
For $H^{n-2}$-almost every $p \in \Sigma$, the trace of $I - A(p)$ is at least $n-2$, and all its eigenvalues lie in $[-1, 1]$.
\end{lemm}
\begin{proof}
We notice by definition that, in the sense of measures, 
\begin{align*}
\sum\limits_{i = 1}^{n}\alpha_{ii} &= \sum\limits_{i = 1}^{n}\mu_{ii} - |\psi|^{2}\dvol\ \llcorner \overline{M}\\
&= \lim\limits_{k \to \infty}\frac{|\nabla u_{k}|^{2}}{|\log\ep_{k}|} \dvol\ \llcorner \overline{M} - |\psi|^{2}\dvol\ \llcorner \overline{M} \leq 2\mu_{s}.
\end{align*}
Thus by the differentiation theorem for Radon measures (\cite{leonGMT}, Theorem 4.7), we have $\tr A(p) \leq 2$ and hence $\tr(I - A(p)) \geq n-2$ for $H^{n-2}$-a.e. $p \in \Sigma$. For the second assertion, we take an arbitrary $\xi \in \RR^{n}$ and observe that, as measures,
\begin{align*}
0 \leq \xi^{i}\xi^{j}\alpha_{ij} = \lim\limits_{k \to \infty} \frac{\langle \nabla u_{k}, \xi \rangle^{2}}{|\log\ep_{k}|}\dvol \ \llcorner\overline{M} - \left( \psi(\xi) \right)^{2}\dvol\ \llcorner \overline{M} \leq 2\mu_{s} |\xi|^{2},
\end{align*}
and the result follows again from the differentiation theorem.
\end{proof}
\begin{prop}\label{blownupstationary} 
For $H^{n-2}$-almost every $p_{0} \in \Sigma \cap \overline{M}$ and for all ambient vector field $X \in C^{1}_{0}(B^{N}_{1}; \RR^{N})$, there holds
\begin{equation}\label{interiorblowup}
\sum\limits_{1 \leq i,j \leq n}(\delta_{ij} - A_{ij}(p_{0}))\int_{T_{p_{0}}\Sigma \cap B^{N}_{1}} \langle \nabla_{e_{i}(p_{0})}X(y), e_{j}(p_{0}) \rangle dH^{n-2}(y) = 0,
\end{equation}
where $T_{p_{0}}\Sigma$ denotes the approximate tangent plane of $\Sigma$ at $p_{0}$.
\end{prop}
\begin{proof}
For convenience, for any vector field $X \in C^{1}_{0}(\RR^{N}; \RR^{N})$, we let $F(X)$ denote the left-hand side of \eqref{almoststationary}, which still make sense even for this larger class of vector fields. For later use, we note that $F(X)$ depends only on $X|_{\overline{M}}$. We will show that the conclusion holds whenever $p_{0}$ is chosen so that $T_{p_{0}}\Sigma$ exists at $p_{0}$, and that $p_{0}$ is a Lebesgue point for each $\left(\delta_{ij} - A_{ij}(\cdot)\right) \Theta(\mu_{s}, \cdot)$ with respect to $H^{n-2}\llcorner \Sigma$. Note that both conditions hold for $H^{n-2}$-almost every $p_{0} \in \Sigma$.

In the case where $p_{0} \in M$, we take any ambient vector field $X \in C^{1}_{0}(B^{N}_{1}; \RR^{N})$ and define $X_{\rho}(p) = X\left( \rho^{-1}(p - p_{0}) \right)$. Writing $X_{\rho}(p) = X_{\rho}^{\tan}(p) + X_{\rho}^{\text{nor}}(p)$, where $X^{\tan}_{\rho}$ and $X^{\text{nor}}_{\rho}$ denote the projections of $X_{\rho}(p)$ onto $T_{p}M$ and $T_{p}^{\perp}M$, respectively, we observe that
\begin{align}\label{rescalestationary}
\int_{B^{N}_{1}} &\langle \nabla_{e_{i}(p_{0} + \rho y)}X(y),  e_{j}(p_{0} + \rho y) \rangle \left(\delta_{ij} - A_{ij}(p_{0} + \rho y)\right)\Theta(\mu_{s}, p_{0} + \rho y) dH^{n-2}(y)\\\nonumber
&= \rho^{3-n}F(X_{\rho}) = \rho^{3-n}(F(X_{\rho}^{\tan}) + F(X_{\rho}^{\text{nor}}))= \rho^{3-n}(0 + F(X_{\rho}^{\text{nor}}))\\\nonumber
& =  -\rho^{3-n}\int_{B^{N}_{\rho}(p_{0})\cap \Sigma}\langle X_{\rho}, H_{ij} \rangle(\delta_{ij} - A_{ij})\Theta(\mu_{s}, \cdot)dH^{n-2},
\end{align}
where $H$ denotes the 2nd fundamental form of $\widetilde{M}$ in $\RR^{N}$. Notice that the last line tends to zero as $\rho \to 0$. 
On the other hand, since $\{e_{i}\}$ and $X$ are $C^{1}$, by our choice of $p_{0}$ we infer that the first line of \eqref{rescalestationary} tends to the following expression as $\rho$ tends to zero:
\begin{equation}\label{scalelimitstationary}
\Theta(\mu_{s}, p_{0}) \left( \delta_{ij} - A_{ij}(p_{0}) \right)  \int_{B_{1}^{N}\cap T_{p_{0}}\Sigma} \langle \nabla_{e_{i}(p_{0})}X(y), e_{j}(p_{0}) \rangle dH^{n-2}(y).
\end{equation}
Thus, sending $\rho$ to zero in \eqref{rescalestationary}, we get
\begin{equation}
\Theta(\mu_{s}, p_{0}) \left( \delta_{ij} - A_{ij}(p_{0}) \right)  \int_{B_{1}^{N}\cap T_{p_{0}}\Sigma} \langle \nabla_{e_{i}(p_{0})}X(y), e_{j}(p_{0}) \rangle dH^{n-2}(y)  = 0.
\end{equation}
This establishes \eqref{interiorblowup} since the density is always positive on $\Sigma$.

For the case $p_{0} \in \Sigma \cap \partial M$, the proof is almost identical, except that now we further decompose $X_{\rho}^{\tan}$ as $\left( X_{\rho}^{\tan} - \varphi_{\rho}\nu \right) + \varphi_{\rho}\nu$. The choice of $\{\varphi_{\rho}\}$ will be made later. For now we only require that they are supported in $B^{N}_{\rho}(p_{0})$ and satisfy $\varphi_{\rho} = \langle X_{\rho}, \nu \rangle$ on $B^{N}_{\rho}(p_{0}) \cap \partial M$. Then, in place of \eqref{rescalestationary}, we get
\begin{align}\label{bdyrescalestationary}
&\int_{B^{N}_{1}} \langle \nabla_{e_{i}(p_{0} + \rho y)}X(y),  e_{j}(p_{0} + \rho y) \rangle \left(\delta_{ij} - A_{ij}(p_{0} + \rho y)\right)\Theta(\mu_{s}, p_{0} + \rho y) dH^{n-2}(y)\\ \nonumber&= -\rho^{3-n}\int_{B^{N}_{\rho}(p_{0})\cap \Sigma}\langle X_{\rho}, H_{ij} \rangle(\delta_{ij} - A_{ij})\Theta(\mu_{s}, \cdot)dH^{n-2} + \rho^{3-n} F(\varphi_{\rho}\nu).
\end{align}
As in the previous case, when $\rho \to 0$, the first line tends to \eqref{scalelimitstationary}, whereas the first term on the second line tends to zero. For the remaining term, we choose
\begin{equation*}
\varphi_{\rho}(y) = \left\langle X_{\rho}\left(\pi(y)\right), \nu\left(\pi(y)\right) \right\rangle \xi\left(\rho^{-1}\dist(y, \partial M)\right)
\end{equation*}
for $y \in \overline{M} \cap B_{\rho}^{N}(p_{0})$, and extend it arbitrarily to have compact support on $B_{\rho}^{N}(p_{0})$. The cutoff function $\xi$ is compactly supported in $[0, 1)$ with $\xi(t) \equiv 1$ for $t < 1/2$. Also, $\pi$ denotes the nearest-point projection onto $\partial M$. Then, when we compute $\rho^{3-n}F(\varphi_{\rho}\nu)$ and let $\rho \to 0$, we are left only with the terms where either $X_{\rho}(\pi(\cdot))$ or $\xi(\dist(\cdot, \partial M)/\rho)$ is differentiated, while all the remaining terms vanish in the limit for scaling reasons. More precisely, noting that $\nabla_{e_{i}}\dist(\cdot, \partial M) = \langle e_{i}, \nu \rangle = \delta_{in}$ and that $\nabla_{\nu}\pi(\cdot) = 0$, we compute
\begin{align}\label{longcomputation}
&\rho^{3-n}F(\varphi_{\rho}\nu)\\
\nonumber=& \rho^{3-n}\sum\limits_{i \leq n-1}\int_{B_{\rho}^{N}(p_{0}) \cap \Sigma}( \nabla_{e_{i}}\varphi_{\rho})(- A_{in})\Theta(\mu_{s}, \cdot) dH^{n-2}\\
\nonumber&+ \rho^{3-n}\int_{B_{\rho}^{N}(p_{0})\cap \Sigma}( \nabla_{\nu}\varphi_{\rho} )(1 - A_{nn})\Theta(\mu_{s}, \cdot)dH^{n-2} + O(\rho)\\
\nonumber=&\sum\limits_{i \leq n-1}\int_{B^{N}_{1} \cap \rho^{-1}(\Sigma - p_{0})}  \left\langle \nabla_{(\nabla_{e_{i}}\pi)(p_{0} + \rho y)} X\left(\frac{\pi(p_{0} + \rho y) - p_{0}}{\rho}\right), \nu(\pi(p_{0} + \rho y))  \right\rangle\\
\nonumber&\times \xi\left(\rho^{-1}\dist(p_{0} + \rho y, \partial M)\right)\left( - A_{in}(p_{0} + \rho y)\right) \Theta(\mu_{s}, p_{0} + \rho y) dH^{n-2}(y)\\
\nonumber&+  \int_{B_{1}^{N} \cap \rho^{-1}(\Sigma - p_{0})} \left\langle X\left(\frac{\pi(p_{0} + \rho y) - p_{0}}{\rho}\right), \nu(\pi(p_{0} + \rho y)) \right\rangle \\
\nonumber&\times \xi'\left(\rho^{-1}\dist(p_{0} + \rho y, \partial M)\right)\left(1 - A_{nn}(p_{0} + \rho y)\right)\Theta(\mu_{s}, p_{0} + \rho y) dH^{n-2}(y)\\
\nonumber&+ O(\rho).
\end{align}
At this point, we notice that, by the definition of the approximate tangent plane (cf. \cite{leonGMT}, Section 11) and the fact that $\Sigma \subset \overline{M}$, we infer that $T_{p_{0}}\Sigma$, which is a whole $(n-2)$-plane in $\RR^{N}$, cannot be transverse to $T_{p_{0}}\partial M$ in $T_{p_{0}}\widetilde{M}$ and hence must be contained in it. Therefore, to continue, we may assume without loss of generality that $e_{i}(p_{0}) = \paop{y^{i}}, e_{n}(p_{0}) = \nu(p_{0}) = \paop{y^{n}}$, and that 
\begin{equation*}
T_{p_{0}}\Sigma = \RR^{n-2}\times\{0\};\ T_{p_{0}}\partial M = \RR^{n-1}\times\{0\};\ T_{p_{0}} \widetilde{M} = \RR^{n}\times \{0\}.
\end{equation*}
Moreover, we have
\begin{align*}
&\rho^{-1}(\Sigma - p_{0}) \cap B_{1}^{N} \to T_{p_{0}}\Sigma \cap B_{1}^{N} = (\RR^{n-2}\times \{0\}) \cap B_{1}^{N},\\
&\rho^{-1}(\partial M - p_{0}) \cap B_{1}^{N} \to T_{p_{0}}\partial M \cap B_{1}^{N} = (\RR^{n-1}\times \{0\}) \cap B_{1}^{N},
\end{align*}
where both convergences are in the Hausdorff distance. (The first convergence holds by the lower bound in \eqref{singulardensityinequalities}, whereas the second one follows because $\partial M$ is smooth.) Consequently, the coordinates $y^{n}, y^{n+1}, \cdots, y^{N}$ and the distance
\begin{equation*}
\rho^{-1}\dist(p_{0} + \rho y, \partial M)\ ( = \dist(y, \rho^{-1}(\partial M - p_{0})))
\end{equation*}
all converge to zero uniformly on $B_{1}^{N} \cap \rho^{-1}(\Sigma - p_{0})$ as $\rho \to 0$. Hence, we have that
\begin{equation*}
\xi(\rho^{-1}\dist(p_{0} + \rho y, \partial M)) = 1 \text{ and } \xi'(\rho^{-1}\dist(p_{0} + \rho y, \partial M)) = 0
\end{equation*}
for $y \in \rho^{-1}(\Sigma - p_{0})$ and $\rho$ small enough. 

With the help of these observations, we may return to \eqref{longcomputation} and let $\rho \to 0$ to get
\begin{equation}\label{Flimit}
\lim\limits_{\rho \to 0}\rho^{3-n} F(\varphi_{\rho}\nu) = \sum\limits_{i \leq n-1}\Theta(\mu_{s}, p_{0})( - A_{in}(p_{0}))\int_{T_{p_{0}}\Sigma \cap B^{N}_{1}} \langle \nabla_{e_{i}(p_{0})}X(y), \nu(p_{0}) \rangle dH^{n-2}(y),
\end{equation}
Recalling \eqref{bdyrescalestationary} and the remarks following it, we conclude that the right-hand side above is equal to \eqref{scalelimitstationary}. Making some cancellations, we arrive at
\begin{align}\label{firstchoiceX}
\sum\limits_{i, j \leq n-1} &(\delta_{ij} - A_{ij}(p_{0})) \int_{B^{N}_{1} \cap T_{p_{0}}\Sigma} \langle \nabla_{e_{i}(p_{0})}X(y), e_{j}(p_{0}) \rangle dH^{n-2}(y)\\\nonumber
=& \sum\limits_{j \leq n-1} A_{nj}(p_{0}) \int_{B^{N}_{1} \cap T_{p_{0}}\Sigma} \langle \nabla_{\nu(p_{0})}X(y), e_{j}(p_{0}) \rangle dH^{n-2}(y)\\\nonumber
&- (1 - A_{nn}(p_{0}))\int_{B_{1}^{N} \cap T_{p_{0}}\Sigma} \langle \nabla_{\nu(p_{0})}X(y), \nu(p_{0}) \rangle dH^{n-2}(y).
\end{align}
Note that this relation holds as long as $X$ has compact support when restricted to $B_{1}^{N} \cap T_{p_{0}}\Sigma$. To continue, we choose $X(y) = y^{n}\xi(y^{1}, \cdots, y^{n-1})\paop{y^{k}}$, where $1 \leq k \leq n-1$ and $\xi \in C^{1}_{0}(B_{1}^{n-1})$. Then we have
\begin{align*}
0 = A_{nk}(p_{0}) \int_{B_{1}^{N} \cap T_{p_{0}}\Sigma} \xi dH^{n-2} \text{ for all }\xi \in C^{1}_{0}(B^{n-1}_{1}),
\end{align*}
which forces $A_{nk}(p_{0})$ and, by symmetry, $A_{kn}(p_{0})$ to vanish for $1 \leq k \leq n-1$. 
Going back to \eqref{Flimit}, we infer that $\lim_{\rho \to 0}\rho^{3-n}F(\varphi_{\rho}\nu) = 0$. Recalling \eqref{bdyrescalestationary}, we get \eqref{interiorblowup} again.
\end{proof}

\begin{proof}[Conclusion of the proof of Theorem \ref{GLgeneralconvergence}(3)] 
Assuming that $e_{i}(p_{0}) = \paop{y^{i}}$ for $1 \leq i \leq n$ and that
\begin{equation*}
T_{p_{0}}\Sigma = \RR^{n-2} \times \{0\};\ T_{p_{0}}\widetilde{M} = \RR^{n} \times \{0\},
\end{equation*}
we choose any $j = n-1, n$ and $k = 1, \cdots, n$ and $\xi\in C^{1}_{0}(B^{n-2}_{1})$, and substitute $X(y) = y^{j}\xi(y^{1}, \cdots, y^{n-2})\paop{y^{k}}$ into \eqref{interiorblowup} to arrive at
\begin{equation*}
(\delta_{jk} - A_{jk}(p_{0}))\int_{T_{p_{0}}\Sigma \cap B_{1}^{N}} \xi dH^{n-2} = 0,
\end{equation*}
from which we get
\begin{equation}\label{nullity}
\delta_{jk} - A_{jk}(p_{0}) = 0, \text{ for }j = n-1, n \text{ and }k = 1, \cdots, n.
\end{equation}
In particular, $I - A(p_{0})$ has at least two zero eigenvalues. Recalling Lemma \ref{Aproperties}, we conclude that $I - A(p_{0})$ is the orthogonal projection onto an $(n-2)$-plane, which by \eqref{nullity} must be $\RR^{n-2}\times \{0\} = T_{p_{0}}\Sigma$. Consequently, \eqref{almoststationary} becomes the desired stationarity condition, \eqref{GLlimitstationary}.
\end{proof}
\section{An example}

In this section, we take $M$ to be the solid ellipsoid $\{x^{2} + y^{2} + (z/l)^{2} \leq 1\} \subset \RR^{3}$ with the Euclidean metric $\delta$, where $l$ is a constant to be determined. Note that $\partial M$ is certainly convex in the sense defined in the introduction. Moreover, $\cH_{1, \fn}(M) = \{0\}$, i.e. $M$ admits not non-zero harmonic $1$-forms with normal component vanishing on the boundary. We will construct a sequence of solutions to \eqref{GLNeumann} verifying all the assumptions of Theorem \ref{GLgeneralconvergence}, such that the limit varifold $V$ is non-zero and supported on $\partial M$. 

To begin, we slightly modify the potential term in the Ginzburg-Landau energy $E_{\ep}$ so that it's defined on $W^{1, 2}(M; \CC)$. For example we can let
\begin{equation}\label{nwepotential}
W(t) = \left\{
\begin{array}{cl}
\frac{(t^{2} - 1)^{2}}{4}, &\ t \leq 1,\\
(t - 1)^{2}, &\ t > 1.
\end{array}
\right.,
\end{equation}
and define
\begin{equation*}
\tilde{E}_{\ep}(u) = \int_{M}\frac{|\nabla u|^{2}}{2} + \frac{W(|u|)}{\ep^{2}}\dvol,\ u \in W^{1, 2}(M; \CC).
\end{equation*}
It's not hard to see that if $u \in W^{1, 2}(M; \CC)$ is a critical point of $\tilde{E}_{\ep}$, then in fact $|u| \leq 1$, $\tilde{E}_{\ep}(u) = E_{\ep}(u)$, and $u$ is a classical solution to \eqref{GLNeumann}. (Indeed, if $\{|u| > 1\} \neq \emptyset$, one derives a contradiction by applying the strong maximum principle and the Hopf lemma to $|u|^{2}$.) Therefore we will construct our solutions as critical points of $\tilde{E}_{\ep}$. Moreover, we restrict our attention to functions that are rotationally symmetric with respect to the $z$-axis. Specifically, in terms of cylindrical coordinates $(r, \theta, z)$ on $\RR^{3}$, we define 
\begin{equation*}
\cA = \{u \in W^{1, 2}(M; \CC)|\ \partial_{\theta}u = 0\}.
\end{equation*}
An important observation is that any critical point of $\tilde{E}_{\ep}|_{\cA}$ is also a critical point for $\tilde{E}_{\ep}$, and hence a solution to \eqref{GLNeumann}. This is because if $u \in \cA$, then for any $\zeta \in C^{1}(\overline{M}; \CC)$ we may use cylindrical coordinates to find that
\begin{align*}
\delta \tilde{E}_{\ep}(u)(\zeta)  = \delta\tilde{E}_{\ep}(u)(\overline{\zeta}),
\end{align*}
where $\overline{\zeta}(r, z) = \fint_{0}^{2\pi}\zeta(r, \theta, z)d\theta$. Next, as in Section 2 of \cite{stern1}, we apply standard min-max methods to $\tilde{E}_{\ep}|_{\cA}$ to produce a sequence $\{u_{k}\}$ of solutions to \eqref{GLNeumann} with $\ep = \ep_{k} \to 0$ whose energies $c_{k} \equiv \tilde{E}_{\ep_{k}}(u_{k}) = E_{\ep_{k}}(u_{k})$ are characterized by the following min-max value.
\begin{equation}\label{minmaxvalues}
0 < \alpha \leq c_{k} = \inf_{h \in \Gamma}\sup_{y \in \overline{B^{2}}}E_{\ep_{k}}(h(y)),
\end{equation}
where $B^{2}$ is the unit disk in $\RR^{2}$, $\Gamma = \{h \in C^{0}(\overline{B^{2}}; \cA)|\ h(y) = \text{const. } y \text{ for all }y \in S^{1}\}$ and $\alpha$ is a positive constant independent of $k$. We next estimate $c_{k}$ relative to $|\log\ep_{k}|$.
\begin{claim*}
\begin{equation}\label{minmaxenergybounds}
0 < \liminf\limits_{k \to \infty}\left| \log\ep_{k}\right|^{-1}c_{k} \leq \limsup\limits_{k \to \infty}\left| \log\ep_{k} \right|^{-1}c_{k} \leq K < \infty,
\end{equation}
where $K$ is independent of $l > 0$.
\end{claim*}
\begin{proof}
For the upper bound in \eqref{minmaxenergybounds}, we follow the arguments in Section 4 of \cite{stern1}. Specifically, we take the family $\{v_{y, \ep}\}_{y \in \overline{B^{2}},\ep > 0}$ constructed there and define 
\begin{equation}
h_{k}(y) = v_{y, \ep_{k}}\circ f : M \to \RR^{2},
\end{equation}
where $f: M \to \RR^{2}$ is defined by $f(r, \theta, z) = (r, z)$. Since $f(M)$ is a bounded subset of $\RR^{2}$, we see as indicated in \cite{stern1} that $h_{k} \in \Gamma$. To estimate $\tilde{E}_{\ep_{k}}(h_{k}(y))$ (which coincides with $E_{\ep_{k}}(h_{k}(y))$ because $v_{y, \ep_{k}}$ maps into the unit disk), we observe that $[f]_{0, 1; M} \leq 1$ and that 
\begin{equation}\label{jacobianlevel}
|Jf(p)| = 1 \text{, for all }p \in M;\ 
H^{1}(f^{-1}(x)) \leq 2\pi, \text{ for all }x \in \RR^{2}.
\end{equation}
The Jacobian lower bound follows from direction computation (recall that $|Jf| = \sqrt{\det(Df Df^{t})}$), while the measure upper bound holds because the level sets of $f$ are horizontal circles in $M$. Letting $w = -y/(1 - |y|)$, we apply the co-area formula as in \cite{stern1} to arrive at
\begin{align}\label{sternminmaxestimate}
E_{\ep_{k}}(h_{k}(y)) \leq & C\int_{f(M)\setminus B^{2}_{\ep_{k}}(w)} |x - w|^{-2} H^{1}(f^{-1}(x)) dH^{2}(x)+ C\ep_{k}^{-2}\int_{B^{2}_{\ep_{k}}(w)}H^{1}(f^{-1}(x))dH^{2}(x)\\\nonumber
\leq & 2\pi C \left( \int_{f(M)\setminus B^{2}_{\ep_{k}}(w)}|x - w|^{-2}dH^{2}(x) + 1 \right),
\end{align}
where $C$ is independent of $l$ and $y$. To estimate the first integral in parentheses, note that since $f(M) \subset B^{2}_{l}$, when $|w| \geq 2l$ we have $|x - w| \geq l$ for all $x \in f(M)$, and hence, for some $C$ independent of $l$ and $y$, we have
\begin{equation*}
\int_{f(M)\setminus B^{2}_{\ep_{k}}(w)}|x - w|^{-2} dH^{2}(x) \leq Cl^{-2}H^{2}(f(M)) \leq C.
\end{equation*}
On the other hand, if $|w| \leq 2l$, then $f(M) \subset B^{2}_{3l}(w)$ and hence
\begin{align*}
\int_{f(M)\setminus B^{2}_{\ep_{k}}(w)}|x - w|^{-2} dH^{2}(x) &\leq  \int_{B^{2}_{3l}(w) \setminus B^{2}_{\ep_{k}}(w)}|x - w|^{-2} dH^{2}(x) = 2\pi\left( \log(3l) + \left| \log\ep_{k} \right| \right).
\end{align*}
Putting the previous two estimates back into \eqref{sternminmaxestimate} and recalling the definition of $c_{k}$, we obtain that, for some $C$ independent of $l$ and $y$, there holds
\begin{equation*}
\left| \log\ep_{k} \right|^{-1}c_{k} \leq \left| \log\ep_{k} \right|^{-1}\sup\limits_{y \in \overline{B^{2}}}E_{\ep_{k}}(h_{k}(y)) \leq C\frac{1 + \log l + \left| \log\ep_{k} \right|}{\left| \log\ep_{k} \right|},
\end{equation*}
from which the upper bound in \eqref{minmaxenergybounds} follows immediately.

For the lower bound, we argue by contradiction and suppose that $c_{\ep} = o(\left| \log\ep \right|)$. Then the estimates in Section 6 of the present paper would be applicable everywhere on $\overline{M}$. In particular, we see from the proof of Theorem \ref{regpartconvergence} that $\nabla \rho_{k},\ \ep_{k}X_{k}$ and $d(u_{k}\times du_{k}) = 2\rho_{k}d\rho_{k}\wedge \varphi_{k}$ all converge uniformly to zero on $\overline{M}$. Recalling that $d^{\ast}(u_{k} \times du_{k}) = 0$ by \eqref{GLNeumann}, we have
\begin{equation*}
\|u_{k} \times du_{k}\|_{2; B^{3}}^{2} \leq C\cD_{\delta}(u_{k}\times du_{k}, u_{k}\times du_{k}) \to 0 \text{ as }k \to \infty,
\end{equation*}
where the inequality holds because $\cH_{1, \fn}(M) = \{0\}$ (cf. \eqref{hodgecoercive}). In view of \eqref{densitydecomposition}, the above discussion implies that $e_{\ep_{k}}(u_{k}) \to 0$ in $L^{1}(B^{3})$ and hence $E_{\ep_{k}}(u_{k}) \to 0$ as $k \to \infty$, contradicting \eqref{minmaxvalues}.
\end{proof}

Having obtained the bounds \eqref{minmaxenergybounds}, we may apply Theorem \ref{GLgeneralconvergence} to $\{u_{k}\}$. Note that in this case the $1$-form $\psi$ in the conclusion vanishes because $\cH_{1, \fn}(M) = \{0\}$, and hence $\mu = \|V\|$. Consequently, \eqref{minmaxenergybounds} yields
\begin{equation}\label{varifoldmass}
0 <  \|V\|(\overline{M}) \leq K. 
\end{equation}
In particular, $V \neq 0$. To conclude, we prove the following. 
\begin{claim*}
Let $S = \partial M \cap \{z = 0\}$. Then for $l$ sufficiently large, we have $V = \theta_{0}v(S)$, where $\theta_{0}$ is a positive constant, and $v(S)$ denotes the varifold associated to $S$.
\end{claim*}
\begin{proof}
Letting $L = \{(0, 0, z)|\ |z| \leq l\}$, we first show that
\begin{equation*}
\supp V ( = \Sigma) \subset L \cup S.
\end{equation*}
Indeed, suppose this is not the case, then for all $p \in \supp V \setminus (S \cup L)$, since $\Sigma$ inherits the rotational symmetry of the sequence $\{u_{k}\}$, it must contain the circle $C$ obtained by rotating $p$ around $L$. This in turn forces $T_{p}\Sigma = T_{p}C = \Span \{\partial_{\theta}|_{p}\}$ for all $p \in \supp V \setminus (S \cup L)$. Based on this observation we will derive a contradiction with the stationarity of $V$. We treat the cases $p \in M$ and $p \in \partial M$ separately.

If $p \in M$, take ball $B^{3}_{2r}(p)$ contained in $M$ and let $\zeta$ be a cut-off function with $\supp\zeta \subset B^{3}_{2r}(p)$ and $\zeta \equiv 1$ on $B^{3}_{r}(p)$. Then, following the proof of Lemma B.2 of \cite{coldingdelellis} and recalling $T_{q}\Sigma = \Span\{\partial_{\theta}|_{q}\}$ for all $q \in B^{3}_{2r}(p)$, we compute
\begin{align*}
\delta V (\zeta \partial_{r}) &= \int_{B^{3}_{2r}(p)} \langle D_{\partial_{\theta}}\left(\zeta \partial_{r}\right), \partial_{\theta}\rangle d\|V\|=\int_{B^{3}_{2r}(p)} \zeta D^{2}r(\partial_{\theta}, \partial_{\theta}) + \partial_{\theta}\zeta \langle \partial_{r}, \partial_{\theta}\rangle d\|V\| \\
& \geq \int_{B^{3}_{r}(p)} D^{2}r(\partial_{\theta}, \partial_{\theta}) d\|V\| > 0,
\end{align*}
since the integrand is everywhere positive, and $\|V\|(B^{3}_{r}(p)) > 0$ because $p \in \supp V$. However this is a contradiction to the stationarity of $V$. 

On the other hand, if $p \in \partial M$, then we take some ball $B^{3}_{2r}(p)$ that does not intersect $L \cup S$. Without loss of generality, we may assume that $B^{3}_{2r}(p)$ is contained in $\{z > 0\}$. Next we let $\zeta$ be as in the previous case and define a vector field $X$ on $B^{3}_{2r}(p)$ by letting $X(q) = \pi_{q}(\partial_{z}|_{q})$, where $\pi_{q}$ denotes projection onto the tangent plane to the level set of $x^{2} + y^{2} + (z/l)^{2}$ passing through $q$. Then $\zeta X$ is admissible in the sense defined in the Introduction, and we compute
\begin{align*}
\delta V(\zeta X) &= \int_{B^{3}_{2r}(p)} \zeta \langle D_{\partial_{\theta}}X, \partial_{\theta} \rangle + \partial_{\theta}\zeta \langle X, \partial_{\theta} \rangle d\|V\|= - \int_{B^{3}_{2r}(p)} \zeta \langle X, D_{\partial_{\theta}}\partial_{\theta} \rangle d\|V\| < 0,
\end{align*}
where in the second equality we used the fact that $X \perp \partial_{\theta}$ and the final inequality follows because the integrand is strictly positive and $\|V\|(B^{3}_{r}(p)) > 0$. Since $\zeta X$ is admissible, we again have a contradiction to the fact that $V$ is stationary with free boundary.

Since $\supp V \subset L \cup S$, the constancy theorem (see \cite{leonGMT}, Section 41) applied to $V\llcorner L$ (which is stationary in $M$ in the usual sense) and $V \llcorner S$ (which is stationary when viewed as varifold in $\partial M$) then implies that there exist constants $\theta_{0}, \theta_{1} \geq 0$ such that
\begin{equation*}
V = \theta_{0} v(S) + \theta_{1} v(L).
\end{equation*}
Since $V \neq 0$, it remains to show that $\theta_{1} = 0$ when $l$ is large enough. To that end, observe that by the upper bound in \eqref{varifoldmass} we can estimate
\begin{align*}
\theta_{1} = \frac{1}{2l} \|V\|(L) \leq \frac{1}{2l}\|V\|(\overline{M}) \leq \frac{K}{2l}.
\end{align*}
We take $l$ so large that $K/2l < \eta_{3}/2$, where $\eta_{3}$ is the constant in Proposition \ref{interioretareg} with $n = 3, \lambda  = 1$ and $\Lambda = 1$, then we see that
\begin{equation*}
\mu(B^{3}_{1/2}) ( = \|V\|(B^{3}_{1/2})) = \theta_{1} < \eta_{3}/2.
\end{equation*}
Hence for $k$ large enough we have $\int_{B^{3}_{1/4}} e_{\ep_{k}}(u_{k}) dx < \eta_{3} \left| \log(4\ep_{k}) \right|$.
Proposition \ref{interioretareg} then implies that $\sup_{B^{3}_{1/8}}e_{\ep_{k}}(u_{k}) \leq C\int_{B^{3}_{1/4}}e_{\ep_{k}}(u_{k})dx$.
Consequently, for all $r < 1/8$, we use the above estimate and \eqref{varifoldmass} to get that for sufficiently large $k$, 
\begin{align*}
r^{-1}\mu_{k}(B^{3}_{r}) \leq \frac{Cr^{2}}{\left| \log\ep_{k} \right|}\int_{B^{3}_{1/2}}e_{\ep_{k}}(u_{k})dx\ \Longrightarrow \  r^{-1}\mu(B^{3}_{r}) \leq Cr^{2}K.
\end{align*}
Therefore $\theta_{1} = \Theta(\mu, 0) = 0$ and $V = \theta_{0}v(S)$ with $\theta_{0} > 0$, as claimed.
\end{proof}

\appendix
\section{Auxiliary estimates for solutions}

We first recall a result essentially due to Bethuel, Brezis and Helein and establish some simple extensions of it. 
\begin{lemm}[cf. \cite{bbh}, Lemma 2]
\label{bbhmax}
Let $g \in \cM_{\lambda, \Lambda}$ and suppose $v$ satisfies
\begin{equation*}
\left\{
\begin{array}{cl}
-\ep^{2}\Delta_{g}v + 2v \leq F& \text{ in }B_{R}\\
v \leq A& \text{ on }\partial B_{R},
\end{array}
\right.
\end{equation*}
where $F$ and $A$ are positive constants. Then for all $x \in B_{R}$ we have
\begin{equation}\label{bbhestimate}
v(x) \leq \frac{F}{2} + A e^{\sqrt{\lambda}\frac{|x|^{2} - R^{2}}{2\ep R}}, 
\end{equation}
provided $\ep$ is small enough depending on $n, \lambda, \Lambda$ and $R$.
\end{lemm}
\begin{proof}
Letting $w(x) = F/2 + Ae^{\sqrt{\lambda}\frac{|x|^{2} - R^{2}}{2\ep R}}$, a straightforward computation shows that
\begin{equation}
-\ep^{2}\Delta_{g} w + 2w \geq F,
\end{equation}
provided $\ep$ is sufficiently small. On the other hand, clearly we have $w  = F/2 + A \geq A$ on $\partial B_{R}$. Thus $w$ is a supersolution and the standard maximum principle yields the result.
\end{proof}

We would also need a boundary version of the above result.
\begin{lemm}
\label{bbhmaxboundary}
Let $g \in \cM^{+}_{\lambda, \Lambda}$ and suppose $v$ satisfies 
\begin{equation*}
\left\{
\begin{array}{cl}
-\ep^{2}\Delta_{g}v + 2v \leq F& \text{ in }B^{+}_{R}\\
v \leq A& \text{ on }\partial B_{R} \cap \{x^{n} > 0\}\\
\pa{v}{x^{n}} = 0 &\text{ on }T_{R}.
\end{array}
\right.
\end{equation*}
Then \eqref{bbhestimate} holds on $B_{R}^{+}$ as long as $\ep$ is sufficiently small depending on $n, \lambda, \Lambda$ and $R$.
\end{lemm}
\begin{proof}
Again letting $w(x)$ be the right-hand side of \eqref{bbhestimate}, then as in the previous lemma we have $-\ep^{2}\Delta_{g}( v - w ) + 2 ( v - w ) \leq 0$ on $B_{R}^{+}$ for small enough $\ep$. Also, $v - w \leq 0$ in $\partial B_{R} \cap \{x^{n} > 0\}$. Moreover, it is easy to see that 
\begin{equation*}
\pa{(v - w)}{x^{n}} = 0 \text{ on }T_{R}.
\end{equation*}
Thus, the Hopf Lemma implies that either $v - w$ is constant, or attains its maximum on $\partial B_{R} \cap \{x^{n} > 0\}$. In either case, we get the desired estimate.
\end{proof}
We can also replace the Neumann boundary condition in Lemma \ref{bbhmaxboundary} by the Dirichlet boundary condition. The proof is essentially the same and will be omitted. (One need only check that $v \leq w$ on $T_{R}$ under the assumption below.)
\begin{lemm}\label{bbhmaxdirichlet}
Under the same assumptions as in Lemma \ref{bbhmaxboundary}, but replacing the condition $\pa{v}{x^{n}} = 0$ on $T_{R}$ by the condition $v = 0 \text{ on }T_{R}$. Then \eqref{bbhestimate} holds for all $x \in B_{R}^{+}$ provided $\ep$ is small enough depending on $n, \lambda, \Lambda$ and $R$.
\end{lemm}

Next we recall the multiplicative version of the interior Schauder estimates.
\begin{lemm}[cf. \cite{bbh}, Lemma A.1]
\label{interiormulti}
Suppose that $u$ satisfies $\Delta_{g}u = f$ on $B_{1}$, then we have
\begin{align}
|D u|_{0; B_{s}} &\leq C\left( |u|_{0; B_{1}} + |u|_{0; B_{1}}^{1/2} |f|_{0; B_{1}}^{1/2} \right) \label{multigradsup} \\
[D u]_{1/2; B_{s}} & \leq C\left( |u|_{0; B_{1}} + |u|_{0; B_{1}}^{1/4} |f|_{0; B_{1}}^{3/4} \right), \label{multigradholder}
\end{align}
where $C$ depends only on $n, s$ and the $C^{0, 1/2}$-norm of the metric $g$.
\end{lemm}
\begin{proof}
The first estimate can be proved as in Lemma A.1 of \cite{bbh}, so we will focus on the second estimate. Take two points $x, y \in B_{1}$ and let $d_{x} = \dist(x, \partial B_{1})$, $d_{y} = \dist(y, \partial B_{1})$ and $d_{x, y} = \min\{ d_{x}, d_{y} \}$. We will show that for all $k$,
\begin{equation}\label{multigradholdercomponent}
\frac{|\partial_{k} u(x) - \partial_{k} u(y)|}{|x - y|^{1/2}} \leq C\left( d_{x, y}^{-3/2}|u|_{0; B_{1}} + |u|_{0; B_{1}}^{1/4} |f|_{0; B_{1}}^{3/4} \right).
\end{equation}
Letting $\lambda_{0} = \frac{|u|_{0; B_{1}}^{1/2}}{|f|_{0; B_{1}}^{1/2}}$, we distinguish three cases depending on the relative sizes of $d_{x}, d_{y}$ and $\lambda_{0}$. Without loss of generality, we assume that $d_{x} \leq d_{y}$, so that $d_{x, y} = d_{x}$. Also, two basic estimates we will frequently use in the subsequent argument are the following.
\begin{equation}\label{supscaled}
|\partial_{k} u|_{0; B_{\lambda/2}(z)} \leq C\left( \lambda^{-1}|u|_{0; B_{1}} + \lambda |f|_{0; B_{1}} \right) \text{ and }
\end{equation}
\begin{equation}\label{holderscaled}
[\partial_{k} u]_{1/2; B_{\lambda/2}(z)} \leq C\left( \lambda^{-3/2}|u|_{0; B_{1}} + \lambda^{1/2} |f|_{0; B_{1}} \right),
\end{equation}
where $z \in B_{1}$ and $\lambda \leq d_{z}$. Both estimates follow from a simple scaling argument by considering $y \mapsto z + \lambda y$ for $y \in B_{1}$ (note that the H\"older constant of $g$ improves when scaled this way). For \eqref{supscaled}, we refer the reader to the proof of Lemma A.1 in \cite{bbh} for the details, while \eqref{holderscaled} is proven similarly. Now we proceed with the proof of \eqref{multigradholdercomponent}.

\vskip 2mm
\noindent\textbf{Case 1:} If $\lambda_{0} \leq d_{x} \leq d_{y}$, then we apply \eqref{holderscaled} with $z = x$ and $\lambda = \lambda_{0}$ to get
\begin{equation}
[\partial_{k} u]_{1/2; B_{\lambda_{0}/2}(x)} \leq C |u|_{0; B_{1}}^{1/4} |f|_{0; B_{1}}^{3/4}.
\end{equation}
If $|x - y| \leq \lambda_{0}/2$, then of course we are done. On the other hand, if $|x - y| > \lambda_{0}/2$, then
\begin{equation}\label{holderfar}
\frac{|\partial_{k} u(x) - \partial_{k} u(y)|}{|x - y|^{1/2}} \leq C\lambda_{0}^{-1/2}\left( |\partial_{k} u(x)| + |\partial_{k} u(y)|\right). 
\end{equation}
By \eqref{supscaled} with $z = x, y$ and $\lambda = \lambda_{0}$, we see that both $|\partial_{k} u(x)|$ and $|\partial_{k} u(y)|$ are bounded by $C|u|_{0; B_{1}}^{1/2} |f|_{0; B_{1}}^{1/2}$. Combining this with the above inequality and recalling the definition of $\lambda_{0}$, we arrive at
\begin{equation}
\frac{|\partial_{k} u(x) - \partial_{k} u(y)|}{|x - y|^{1/2}} \leq C|u|_{0; B_{1}}^{1/4} |f|_{0; B_{1}}^{3/4}.
\end{equation}

\vskip 2mm
\noindent\textbf{Case 2:} If $d_{x} < \lambda_{0} \leq d_{y}$, then \eqref{holderscaled} with $z = y$ and $\lambda = \lambda_{0}$ yields 
\begin{equation}
[\partial_{k} u]_{1/2; B_{\lambda_{0}/2}(y)} \leq C |u|_{0; B_{1}}^{1/4} |f|_{0; B_{1}}^{3/4}.
\end{equation}
As in the previous case, if $|x - y| \leq \lambda_{0}/2$ then we are done, and if $|x - y| > \lambda_{0}/2$, then we have \eqref{holderfar}. Taking $z = y$ and $\lambda = \lambda_{0}$ in \eqref{supscaled}, we see that $|\partial_{k} u(y)| \leq C|u|_{0; B_{1}}^{1/2}  |f|_{0; B_{1}}^{1/2}$. On the other hand, choosing $z = x$ and $\lambda = d_{x}$ in \eqref{supscaled}, we get
\begin{equation}
|\partial_{k} u(x)| \leq C\left( d_{x}^{-1}|u|_{0; B_{1}} + d_{x}|f|_{0; B_{1}} \right) \leq C\left( d_{x}^{-1}|u|_{0; B_{1}} + |u|_{0; B_{1}}^{1/2}  |f|_{0; B_{1}}^{1/2} \right).
\end{equation}
Putting the above estimates for $|\partial_{k} u(x)|$ and $|\partial_{k} u(y)|$ into \eqref{holderfar}, we arrive at
\begin{align*}
\frac{|\partial_{k} u(x) - \partial_{k} u(y)|}{|x - y|^{1/2}} &\leq C\left(  \lambda_{0}^{-1/2} d_{x}^{-1}|u|_{0; B_{1}} + \lambda_{0}^{-1/2}|u|_{0; B_{1}}^{1/2}  |f|_{0; B_{1}}^{1/2}  \right)\\
&\leq C\left( d_{x}^{-3/2}|u|_{0; B_{1}} + |u|_{0; B_{1}}^{1/4} |f|_{0; B_{1}}^{3/4} \right),
\end{align*}
where in the second inequality we used the definition of $\lambda_{0}$ and the fact that $\lambda_{0}^{-1} \leq d_{x}^{-1}$ in the present case.

\vskip 2mm
\noindent\textbf{Case 3:} If $d_{x} \leq d_{y} < \lambda_{0}$, we use \eqref{holderscaled} with $z = y$ and $\lambda = d_{y}$ to get
\begin{equation}
[\partial_{k} u]_{1/2; B_{d_{y}/2}(y)} \leq C\left( d_{y}^{-3/2}|u|_{0; B_{1}} + |u|_{0; B_{1}}^{1/4} |f|_{0; B_{1}}^{3/4} \right).
\end{equation}
If $|x - y| \leq d_{y}/2$, we are done. On the other hand if $|x - y| > d_{y}/2$, there holds
\begin{equation}\label{holderfar2}
\frac{|\partial_{k} u(x) - \partial_{k} u(y)|}{|x - y|^{1/2}} \leq Cd_{y}^{-1/2}\left( |\partial_{k} u(x)| + |\partial_{k} u(y)| \right).
\end{equation}
Estimating $|\partial_{k} u(x)|$ using \eqref{supscaled} with $z = x$ and $\lambda = d_{x}$, we get
\begin{equation*}
|\partial_{k} u(x)| \leq C\left( d_{x}^{-1}|u|_{0; B_{1}} + |u|_{0; B_{1}}^{1/2} |f|_{0; B_{1}}^{1/2} \right),
\end{equation*}
and thus, since $d_{y} \geq d_{x}$, the first term on the right-hand side of \eqref{holderfar2} satisfies
\begin{align*}
d_{y}^{-1/2}|\partial_{k} u(x)| \leq  C\left( d_{x}^{-3/2}|u|_{0; B_{1}} + d_{x}^{-1/2}|u|_{0; B_{1}}^{1/2}|f|_{0; B_{1}}^{1/2} \right).
\end{align*}
Now notice that since $d_{x} < \lambda_{0}$, we in fact have $|f|_{0; B_{1}}^{1/2} < d_{x}^{-1}|u|_{0; B_{1}}^{1/2}$. Plugging this into the previous inequality, we arrive at
\begin{equation}
d_{y}^{-1/2} |\partial_{k} u(x)| < C d_{x}^{-3/2}|u|_{0; B_{1}}.
\end{equation}
The term $d_{y}^{-1/2}|\partial_{k} u(y)|$ on the right-hand side of \eqref{holderfar2} can be estimated similarly, and we conclude that
\begin{equation}
\frac{|\partial_{k} u(x) - \partial_{k} u(y)|}{|x - y|^{1/2}} \leq Cd_{x}^{-3/2} |u|_{0; B_{1}}.
\end{equation}
\end{proof}

A simple reflection argument yields the following boundary version of \eqref{interiormulti}.
\begin{lemm}
\label{boundarymulti}
Let $g$ be a $C^{0, 1/2}$-metric on $B_{1}$ satisfying conditions (1) and (2) in Definition \ref{boundaryclass}, and suppose $u$ satisfies $\Delta_{g} u = f$ on $B_{1}^{+}$ along with the Neumann boundary condition $\pa{u}{x^{n}} = 0$ on $T$, then the estimates \eqref{multigradsup} and \eqref{multigradholder} hold with $B_{s}^{+}$ and $B_{1}^{+}$ in place of $B_{s}$ and $B_{1}$, respectively. 
\end{lemm}
\begin{proof}[Sketch of proof]
Extending $u$ and $f$ to $B_{1}$ by even reflection in $x^{n}$ and following the proof of Lemma \ref{GLreflection}, we see that the extended $u$ solves $\Delta_{g}u = f$ on $B_{1}$, and we can conclude using Lemma \ref{interiormulti}.
\end{proof}
\begin{rmk}\label{dirichletmulti}
Lemma \ref{boundarymulti} holds true with the boundary condition replaced by $u = 0$ on $T$. To prove this version, one would extend $u$ and $f$ as odd functions in $x^{n}$. It's not hard to check that $u$ solves $\Delta_{g}u = f$ on $B_{1}$, and we are again reduced to the interior case.
\end{rmk}

\section{Estimates for the Green matrix of a class of elliptic systems}


The purpose of this appendix is to derive estimates on the fundamental solution for certain elliptic systems in divergence form. We shall see that under mild assumptions on the coefficients, the fundamental solution has the same growth near the diagonal as in the constant-coefficient case.

Given positive integers $n$ and $N$, we are interested in elliptic operators of the following type: 
\begin{equation}\label{divergenceoperator}
\left( Lu \right)_{\alpha} \equiv \partial_{i}\left( A^{ij}_{\alpha \beta}\partial_{j}u_{\beta} + B^{i}_{\alpha \beta}u_{\beta} \right) + C^{i}_{\alpha \beta}\partial_{i}u_{\beta} + D_{\alpha \beta}u_{\beta},
\end{equation}
where $u = (u_{1}, \cdots, u_{N}): B_{1}\subset \RR^{n} \to \RR^{N}$ and the repeated indices are meant to be summed. The Latin indices are understood to run from $1$ to $n$, while the Greek indices go from $1$ to $N$. 

The assumptions on the coefficients of $L$ are as follows. For some fixed constants $\Lambdabar, \lambdabar, c_{0} > 0$, and $1 \geq \bar{\mu} > 0$, we assume that
\begin{enumerate}
\item[(H1)] $A^{ij}_{\alpha\beta} \in C^{0, \bar{\mu}}(B_{1})$ and $B_{\alpha\beta}^{i}, C_{\alpha, \beta}^{i}, D_{\alpha,\beta} \in L^{\infty}(B_{1})$ for all $i, j, \alpha, \beta$.
\item[(H2)] $[A^{ij}_{\alpha\beta}]_{0, \bar{\mu}; B_{1}} + \|B_{\alpha\beta}^{i}\|_{\infty; B_{1}} + \|C_{\alpha\beta}^{i}\|_{\infty; B_{1}} + \|D_{\alpha\beta}\|_{\infty; B_{1}} \leq \Lambdabar$, for all $i, j, \alpha, \beta$.
\item[(H3)] $\lambdabar |\xi|^{2} \leq A^{ij}_{\alpha\beta}(x)\xi_{i\alpha} \xi_{j\beta} \leq \lambdabar^{-1} |\xi|^{2}$, for all $x \in B_{1}$ and $\xi = (\xi_{i\alpha})$ in $\RR^{nN}$.
\item[(H4)] For all $u \in W^{1, 2}_{0}(B_{1}; \RR^{N})$, there holds
\begin{equation}\label{coercive}
\int_{B_{1}}\left( A^{ij}_{\alpha\beta}\partial_{j}u_{\beta} + B^{i}_{\alpha\beta}u_{\beta}\right)\partial_{i}u_{\alpha} - u_{\alpha}\left( C^{i}_{\alpha\beta}\partial_{i}u_{\beta} + D_{\alpha\beta} \right) \geq c_{0} \| u\|^{2}_{1, 2; B_{1}}.
\end{equation}
\end{enumerate}

Under these assumptions, following the construction outlined in Section 2 of \cite{fuchs} (see also \cite{lsw}), we find for each $\RR^{N}$-valued finite Radon measure $\nu$ a unique element $u_{\nu}$ of $\cap_{r < n/(n-1)} W_{0}^{1, r}(B_{1})$ satisfying
\begin{equation}
\left( Lu_{\nu} \right)_{\alpha} = \nu_{\alpha} \text{ for }\alpha = 1, \cdots, N. 
\end{equation}
For each $y \in B_{1}$ and $\gamma \in \{1, \cdots ,N\}$, we define $G_{\gamma}(\cdot, y) = (G_{\gamma\beta}(\cdot, y))_{1 \leq \beta \leq N}$ to be the $u_{\nu}$ corresponding to the following choice of $\nu$:
\begin{equation*}
\nu = e_{\gamma}\delta_{y},
\end{equation*}
where $e_{\gamma}$ denotes the $\gamma$-standard basis vector on $\RR^{N}$, and $\delta_{y}$ is the Dirac measure supported at $y$. In what follows, the matrix-valued function $G = (G_{\gamma\beta})_{1 \leq \gamma, \beta \leq N}: B_{1} \times B_{1} \to \RR^{N^{2}}$ will be referred to as the Green matrix of $L$. For later purposes, we introduce the adjoint of $L$, denoted $L'$ and defined by 
\begin{equation*}
( L^{'}v )_{\alpha} = \partial_{i}\left( A^{ij}_{\alpha \beta}\partial_{j}v_{\beta} - C^{i}_{\alpha \beta}v_{\beta} \right) - B^{i}_{\alpha \beta}\partial_{i}v_{\beta} + D_{\alpha \beta}v_{\beta}.
\end{equation*}
We immediately see from the definition that $L'$ satisfies (H1) to (H4) since $L$ does, and therefore we can define the Green matrix of $L'$, denoted $G'$. 

The basic properties of the Green matrix is summarized below. 
\begin{prop}[cf. \cite{fuchs}, Section 3]
\label{greensbasic}
\begin{enumerate}
\item[(1)] For all $1 \leq p < n/(n-1)$, we have $\|G(\cdot, y)\|_{1, p; B_{1}} \leq C$, where $C = C(n, N, c_{0}, \lambdabar, \Lambdabar, \bar{\mu}, p)$.

\item[(2)] Let $y \in B_{1/2}$ and $0 < r \leq 1/4$. Then for all $1 \leq q < \infty$, we have 
\begin{equation*}
\|G(\cdot, y)\|_{1, q; B_{1}\setminus B_{r}(y)} \leq C = C(n, N, c_{0}, \lambdabar, \Lambdabar, \bar{\mu}, q, r).
\end{equation*}

\item[(3)] For all $x, y \in B_{1}$ with $x \neq y$ and for all $1 \leq \alpha, \beta \leq N$, we have $G_{\alpha\beta}(x, y) = G'_{\beta\alpha}(y, x)$.

\item[(4)] Let $\Delta$ denote the diagonal $\{(x, x)| x \in B_{1}\}$. Then $G \in C^{0, \alpha}_{\text{loc}}(B_{1} \times B_{1} \setminus \Delta; \RR^{N^{2}})$ for all $0 < \alpha < 1$.

\item[(5)] Let $\cM(B_{1}; \RR^{N})$ denote the space of $\RR^{N}$-valued finite Radon measure on $B_{1}$. Then for each $\nu \in \cM_{B_{1}; \RR^{N}}$, we have that $G$ is $L^{1}$ on $B_{1} \times B_{1}$ with respect to both $\nu \times H^{n}$ and $H^{n} \times \nu$. Moreover, $u_{\nu}$ can be represented as follows.
\begin{equation*}
u_{\nu, \alpha}(x) = \int_{B_{1}} G_{\beta\alpha}(x, y)d\nu_{\beta}(y) = \int_{B_{1}} G'_{\alpha\beta}(y, x)d\nu_{\beta}(y),\text{ for $H^{n}$-a.e. $x \in B_{1}$.}
\end{equation*}
\end{enumerate}
\end{prop}
\begin{rmk} In \cite{fuchs} the construction of $G$ and the proof of its properties were only given when $L$ has no lower order terms. However, the argument goes through as long as $L$ is coercive in the sense of (H4) and satisfies the $W^{1, p}$-estimates listed in Section 1 of \cite{fuchs}, which hold under the assumptions (H1) to (H3) by \cite{morreybook}, Theorem 6.4.8.
\end{rmk}
We are now ready to state the main result of this appendix.
\begin{thm}[see also \cite{fuchs}, Theorem 7]
\label{greensgrowth}
There exists constants $C_{1}, C_{2} > 0$ and $ 1/8 \geq R_{1} > 0$, depending only on $n, N, c_{0}, \lambdabar, \Lambdabar$ and $\bar{\mu}$, such that
\begin{equation}
\label{zerothboundcentral}
|G(x, y)| \leq C_{1}|x - y|^{2-n} + C_{2}R^{1 + \bar{\mu} - n}, 
\end{equation}
whenever $y \in B_{1}$, $R \leq \min\{R_{1}, \dist(y, \partial B_{1})/4  \}$ and $x \in B_{2R}(y)\setminus \{y\}$.
\end{thm}
\begin{coro}\label{growthnearzero}
Let $C_{1}$ and $C_{2}$ be as in Theorem \ref{greensgrowth}. There exists a radius $R_{2} = R_{2}(n, N, c_{0}, \lambdabar, \Lambdabar, \bar{\mu}) < 1/8$, such that 
\begin{equation}\label{zerothbound}
|G(x, y)| \leq C_{1}|x - y|^{2-n} + C_{2} R_{2}^{1 + \bar{\mu} - n},
\end{equation}
whenever $x , y \in B_{R_{2}}$ and $x \neq y$. 
\end{coro}

For the proof of Theorem \ref{greensgrowth}, we need to introduce two other elliptic operators. For $u: B_{1} \to \RR^{N}$ and $x_{0} \in B_{1/2}$, we define
\begin{align*}
\left( L^{(0)}u \right)_{\alpha} &= \partial_{i}\left( A^{ij}_{\alpha\beta}(x_{0})\partial_{j}u_{\beta} \right),\\
\left( L^{(1)}u \right)_{\alpha} &= \partial_{i}\left( A^{ij}_{\alpha\beta}\partial_{j}u_{\beta} \right) + C^{i}_{\alpha\beta}\partial_{i}u_{\beta} + \hat{D}_{\alpha\beta}u_{\beta},
\end{align*}
where $\hat{D}_{\alpha\beta} = D_{\alpha\beta} - K\delta_{\alpha\beta}$ and $K$ is a large enough constant depending only on $n, N, c_{0}, \lambdabar, \Lambdabar$ and $\bar{\mu}$ that makes $L^{(1)}$ satisfy (H4). The Green matrix of $L^{(1)}$ is denoted $G^{(1)}$ and it has all the properties listed in Proposition \ref{greensbasic}.
On the other hand, it is well-known that the constant-coefficient operator $L^{(0)}$ has a fundamental solution $E: \RR^{n} \to \RR^{N^{2}}$,whose properties we summarize in the two Propositions below. Both results can be found in Chapter 6 of \cite{morreybook}.
\begin{prop}\label{constantcoefficient}
\begin{enumerate}
\item[(1)] $\left( L^{(0)}E \right)_{\alpha} = e_{\alpha}\delta_{0}$, where $e_{\alpha}$ is the $\alpha$-th coordinate vector in $\RR^{N}$ and $\delta_{0}$ is the Dirac measure supported at $0 \in \RR^{n}$.

\item[(2)] $E$ is an even function and is positively homogeneous of degree $2-n$ on $\RR^{n}\setminus \{0\}$.

\item[(3)] $E$ is smooth away from the origin and its derivatives satisfy
\begin{equation}\label{constantcoefficientgrowth}
|\partial^{\nu}E(x)| \leq C |x|^{2-n-\nu},\ C = C(n, N, \lambdabar, \Lambdabar, |\nu|)
\end{equation}
where $\nu = (\nu_{1}, \cdots, \nu_{n})$ is a multi-index, $\partial^{\nu} = \partial_{1}^{\nu_{1}}\cdots \partial_{n}^{\nu_{n}}$ and $|\nu| = \nu_{1} + \cdots + \nu_{n}$.
\end{enumerate}
\end{prop}

\begin{prop}\label{greensrepresentation} Fixing an exponent $1 < p < \infty$, we have
\begin{enumerate}
\item[(1)] Given $(f_{\alpha}) \in L^{p}(B_{1}; \RR^{N})$, define $u_{\alpha}(x) = \int_{B_{1}}E_{\beta\alpha}(x - y)f_{\beta}(y)dy$ for $x \in B_{1}$. Then $u \in W^{2, p}(B_{1}; \RR^{N})$ and is a weak solution to $L^{(0)}u = f$. Moreover,
\begin{equation}\label{w2ppotential}
\|u\|_{2, p; B_{1}} \leq C\|f\|_{p; B_{1}},\ C = C(n, N, \lambdabar, \Lambdabar, p).
\end{equation}

\item[(2)] Given $(F^{i}_{\alpha}) \in L^{p}(B_{1}; \RR^{nN})$, define $v_{\alpha}(x) = \int_{B_{1}}\partial_{i}E_{\beta\alpha}(x - y)F^{i}_{\beta}(y)dy$ for $x \in B_{1}$. Then $u \in W^{1, p}(B_{1}; \RR^{N})$ and is a weak solution to $L^{(0)}u = \partial_{i}F^{i}$. Moreover
\begin{equation}\label{w1ppotential}
\|u\|_{1, p; B_{1}} \leq C\|F\|_{p; B_{1}},\ C = C(n, N, \lambdabar, \Lambdabar, p).
\end{equation}
\end{enumerate}
\end{prop}
\begin{rmk} The homogeneity of $E$ implies that both statements in Proposition \ref{greensrepresentation} remain true with $B_{1}$ replaced by any $B_{r}(x_{0})$. However the norms $\|f\|_{p; B_{1}}$ and $\|F\|_{p; B_{1}}$ should be replaced by $r^{2 - n/p}\|f\|_{p; B_{r}(x_{0})}$ and $r^{1 - n/p}\|F\|_{p; B_{r}(x_{0})}$, respectively. Also, $\|u\|_{k, p; B_{1}}$ in \eqref{w2ppotential} and \eqref{w1ppotential} should be replaced by the scaled version
\begin{equation}
\|u\|^{\ast}_{k, p; B_{r}(x_{0})} \equiv \sum\limits_{i = 0}^{k}r^{k-\frac{n}{p}} \| D^{k}u \|_{p; B_{r}(x_{0})}.
\end{equation}
\end{rmk}

The proof of Theorem \ref{greensgrowth} consists of two steps. We first prove Theorem \ref{greensgrowth} with $G^{(1)}$ in place of $G$, based on estimates for the fundamental solution of $L^{(0)}$. Then we extend the result to $G$. We begin with the first step, where we in fact get bounds for the first-order derivatives of $G^{(1)}$ as well.

\begin{thm}\label{growthstepone}There exists $R_{0} = R_{0}(n, N, c_{0}, \lambdabar, \Lambdabar, \bar{\mu}) < 1/8$ such that Theorem \ref{greensgrowth} holds with $G^{(1)}$ in place of $G$ and $R_{0}$ in place of $R_{1}$. In addition, we have
\begin{equation}\label{firstderivative}
\left| \pa{G^{(1)}}{x}(x, y) \right| \leq C_{1} |x - y|^{1-n} + |x - y|^{-1}R^{1 + \bar{\mu} - n},
\end{equation}
whenever $y \in B_{1}$, $R \leq \min\{ R_{0}, \dist(y ,\partial B_{1})/4\}$ and $x \in B_{R}(y)\setminus \{y\}$.
\end{thm}
The proof of Theorem \ref{growthstepone} we are about to give is a modification of the one given in \cite{fuchs}. The basic idea is to obtain the desired estimates on $G^{(1)}$ by writing it as $E$ plus some perturbation terms which grow no faster than $E$ near the diagonal. To that end, we define the perturbation operator, denoted $T_{r}$, by letting
\begin{align*}
\left(T_{r}u \right)_{\gamma}(x) =& \int_{B_{r}(x_{0})} \partial_{i}E_{\alpha\gamma}(x - y)\left( A^{ij}_{\alpha\beta}(x_{0}) - A^{ij}_{\alpha\beta}(y) \right)\partial_{j}u_{\beta}(y)dy\\
& - \int_{B_{r}(x_{0})} E_{\alpha\gamma}(x - y) \left( C^{i}_{\alpha\beta}(y)\partial_{i}u_{\beta}(y) + \hat{D}_{\alpha\beta}(y) u_{\beta}(y)\right)dy\\
\equiv & \left(T^{I}_{r}u\right)_{\gamma}(x) - \left(T^{II}_{r}u\right)_{\gamma}(x),
\end{align*}
for $B_{r}(x_{0}) \subset B_{1}$. From Proposition \ref{greensrepresentation}, we infer that $v \equiv T_{r}u$ is a weak solution to $L^{(0)}v = \left( L^{(0)} - L^{(1)} \right)u$ on $B_{r}(x_{0})$, and that $T_{r}$ defines a bounded operator from $W^{1, p}(B_{r}(x_{0}); \RR^{N})$ to itself with respect to the norm $\|\cdot\|^{\ast}_{1, p; B_{r}(x_{0})}$. Furthermore, we shall see below that $T_{r}$ is a contraction mapping provided $r$ is chosen small enough.

\begin{prop}[cf. \cite{fuchs}, (1.10) and the remark preceding it]
\label{perturbationsmall}
There exists $r_{0} = r_{0}(n, N, \lambdabar, \Lambdabar, \bar{\mu})$ such that whenever $r \leq r_{0}$, we have
\begin{equation}
\|T_{r}\|^{\ast}_{p} < \frac{1}{2},
\end{equation}
where $\| T_{r} \|^{\ast}_{p}$ denotes the operator norm of $T_{r}:W^{1, p}(B_{1}; \RR^{N}) \to W^{1, p}(B_{1}; \RR^{N})$ with respect to $\|\cdot\|_{1, p; B_{r}(x_{0})}^{\ast}$.
\end{prop}
\begin{proof}[Proof of Proposition \ref{perturbationsmall}]
Suppose $r \leq r_{0}$, with $r_{0}$ to be determined later. We treat $T^{I}_{r}$ and $T^{II}_{r}$ separately. The former was already handled in \cite{fuchs}, but we include the simple argument for completeness. Specifically, by assumption (H2) and Proposition \ref{greensrepresentation}, for $u \in W^{1, p}(B_{1}; \RR^{N})$ we have
\begin{align*}
\left\| T^{I}_{r}u \right\|^{\ast}_{1, p; B_{r}(x_{0})} &\leq Cr^{1-p/n}\left\| \left( A(x_{0}) - A \right) D u\right\|_{p; B_{r}(x_{0})}
\leq C\Lambdabar r^{\bar{\mu}}\| u \|^{\ast}_{1, p; B_{r}(x_{0})}.
\end{align*}
For $T^{II}_{r}u$, we again apply Proposition \ref{greensrepresentation} to get
\begin{align*}
\left\| T^{II}_{r}u \right\|^{\ast}_{2, p; B_{r}(x_{0})} & \leq Cr^{2 - p/n}\| C^{i}_{\alpha\beta}\partial_{i} u_{\beta}  + \hat{D}_{\alpha\beta} u_{\beta} \|_{p; B_{r}(x_{0})} \leq  C\Lambdabar r \| u \|^{\ast}_{1, p; B_{r}(x_{0})}.
\end{align*}
By the Sobolev inequality and the H\"older inequality, we then get
\begin{align*}
\|T_{r}^{II}u\|^{\ast}_{1, p; B_{r}(x_{0})} &\leq Cr\|T_{r}^{II}u\|^{\ast}_{1, p^{\ast}; B_{r}(x_{0})} \leq Cr\|T_{r}^{II}u\|^{\ast}_{2, p; B_{r}(x_{0})}.
\end{align*}
Therefore we arrive at $\|T^{II}_{r}u\|^{\ast}_{1, p; B_{r}(x_{0})} \leq C\Lambdabar r^{2}\|u\|^{\ast}_{1, p; B_{r}(x_{0})}$. Combining the estimates for $T^{I}_{r}u$ and $T^{II}_{r}u$, we conclude that
\begin{equation}
\|T_{r}u\|^{\ast}_{1, p; B_{r}(x_{0})} \leq C\Lambdabar r^{\bar{\mu}}\|u\|^{\ast}_{1, p; B_{r}(x_{0})},
\end{equation}
for all $u \in W^{1, p}(B_{r}(x_{0}); \RR^{N})$. Choosing $r_{0}$ small enough completes the proof. 
\end{proof}

We proceed to the proof Theorem \ref{growthstepone}, which will be given in two parts: First we establish the estimate \eqref{zerothboundcentral} for $G^{(1)}$, then we prove \eqref{firstderivative}.
\begin{proof}[Proof of Theorem \ref{growthstepone}, Part 1: \eqref{zerothboundcentral} holds for $G^{(1)}$]
Throughout the proof, we will denote by $C$ any constant that depends only on $n, N, c_{0}, \lambdabar, \Lambdabar$ and $\bar{\mu}$, and specify other types of dependence when necessary.

As in \cite{fuchs}, we fix $p = n/(n - \bar{\mu})$ and $q = n/(1 - \bar{\mu})$. Note that $p < n/(n-1)$ whereas $q > n$. Next, let $y \in B_{1}$ and $r \leq \min\{r_{0}/2, \dist(y, \partial B_{1})/4\}$, with $r_{0}$ given by Proposition \ref{perturbationsmall}. For simplicity, we assume that $y = 0$. Furthermore, for a fixed index $\gamma$ we write $G_{\gamma}^{(1)}(\cdot)$ for $\left( G^{(1)}_{\gamma\alpha}(\cdot, 0) \right)_{1 \leq \alpha \leq N}$. Similarly, we write $E_{\gamma}(\cdot)$ for $\left( E_{\gamma\alpha}(\cdot) \right)_{1 \leq \alpha \leq N}$. 

We then observe that both $G_{\gamma}^{(1)} - T_{2r}G_{\gamma}^{(1)}$ and $E_{\gamma}$ are solutions on $B_{2r}$ to $\left( L^{(0)}u \right) = e_{\gamma}\delta_{0}$, and hence we may write
\begin{equation}\label{g1decomposition}
G_{\gamma}^{(1)} = E_{\gamma} + T_{2r}G_{\gamma}^{(1)} + w,
\end{equation}
with $w$ satisfying $L^{(0)}w = 0$ in $B_{2r}$.
Since $r \leq r_{0}/2$, we have by Proposition \ref{perturbationsmall} that $\|T_{2r}\|_{p} < 1/2$ and thus $(I - T_{2r})$ has a bounded inverse on $W^{1, p}(B_{2r}; \RR^{N})$ given by $\sum_{l = 0}^{\infty} T_{2r}^{l}$. Applying this to \eqref{g1decomposition}, we see that
\begin{equation}\label{gseries}
G_{\gamma}^{(1)} = E_{\gamma} + \sum\limits_{l = 1}^{\infty}T_{2r}^{l} E_{\gamma} + \sum\limits_{l = 0}^{\infty} T_{2r}^{l}w,
\end{equation}
as elements of $W^{1, p}(B_{2r}; \RR^{N})$. Among the terms on the right-hand side, the first term already has the desired growth by \eqref{constantcoefficientgrowth}, and thus the proof boils down to showing that the remaining two terms grows no faster than $|x|^{2-n}$ near the origin. 

\begin{lemm}[cf. \cite{fuchs}, Lemma 4.1] 
\label{estimatesonw}
We have $w \in W^{1, q}(B_{2r}; \RR^{N})$ and $\|w\|^{\ast}_{1, q; B_{2r}} \leq Cr^{1 - n/p}$.
\end{lemm}
\begin{proof}[Proof of Lemma \ref{estimatesonw}] We will estimate $\|w\|^{\ast}_{1, q; B_{r}}$ and $\|w\|^{\ast}_{1, q; B_{2r}\setminus B_{r}}$ separately.
Since $L^{(0)}w = 0$ on $B_{2r}$, interior $W^{1, p}$-estimates give
\begin{equation}\label{wsmallball}
\|w\|^{\ast}_{1, q; B_{r}} \leq C\| w \|^{\ast}_{1, p; B_{2r}}.
\end{equation}
To estimate $\|w\|^{\ast}_{1, p; B_{2r}}$, we write 
\begin{equation*}
\|w\|^{\ast}_{1, p; B_{2r}} \leq \|G_{\gamma}^{(1)}\|^{\ast}_{1, p; B_{2r}} + \|T_{2r}G_{\gamma}^{(1)}\|^{\ast}_{1, p; B_{2r}} + \|E_{\gamma}\|^{\ast}_{1, p; B_{2r}}.
\end{equation*} 
The terms $\|E_{\gamma}\|^{\ast}_{1, p; B_{2r}}$ and $\|G^{(1)}_{\gamma}\|^{\ast}_{1, p; B_{2r}}$ are handled exactly as in \cite{fuchs}. In short, the former can be bounded with the help of \eqref{constantcoefficientgrowth}, while the latter is estimated using the Sobolev inequality and Proposition \ref{greensbasic}(1). The result is that
\begin{equation}\label{smallballscaled}
\| E_{\gamma} \|^{\ast}_{1, p; B_{2r}} + \| G^{(1)}_{\gamma} \|^{\ast}_{1, p; B_{2r}} \leq Cr^{1 - n/p}.
\end{equation}
It follows from \eqref{smallballscaled} and Proposition \ref{perturbationsmall} that $\|T_{2r}G^{(1)}_{\gamma}\|^{\ast}_{1, p; B_{2r}} \leq Cr^{1 - n/p}$ as well.  Putting these back into \eqref{wsmallball}, we obtain
\begin{equation}\label{w1qsmallball}
\|w\|^{\ast}_{1, q; B_{r}} \leq Cr^{1-n/p}.
\end{equation}
It remains to show that $\|w\|_{1, q; B_{2r} \setminus B_{r}} \leq Cr^{1-n/p}$. Again, we start with
\begin{equation}\label{w1qdecomposition}
\|w\|^{\ast}_{1, q; B_{2r}\setminus B_{r}} \leq \|G_{\gamma}^{(1)}\|^{\ast}_{1, q; B_{2r}\setminus B_{r}} + \|T_{2r}G_{\gamma}^{(1)}\|^{\ast}_{1, q; B_{2r}\setminus B_{r}} + \|E_{\gamma}\|^{\ast}_{1, q; B_{2r}\setminus B_{r}}.
\end{equation}
The last two terms are estimated as in \cite{fuchs}, Lemma 4.1, and the conclusion is that
\begin{equation}\label{annulusscaled}
\|G^{(1)}\|^{\ast}_{1, q; B_{2r} \setminus B_{r}} + \|E\|^{\ast}_{1, q; B_{2r}\setminus B_{r}} \leq Cr^{1-n/p}.
\end{equation}
As for the term $\|T_{2r}G_{\gamma}^{(1)}\|^{\ast}_{1, q; B_{2r}\setminus B_{r}}$, by its definition we can write
\begin{equation*}
\|T_{2r}G_{\gamma}^{(1)}\|_{1, q; B_{2r}\setminus B_{r}}^{\ast} \leq \|T^{I}_{2r}G_{\gamma}^{(1)}\|_{1, q; B_{2r}\setminus B_{r}}^{\ast} + \|T^{II}_{2r}G_{\gamma}^{(1)}\|_{1, q; B_{2r}\setminus B_{r}}^{\ast}.
\end{equation*}
Following the arguments on \cite{fuchs}, p.520, we infer that
\begin{equation}\label{TGI1qannulus}
\|T^{I}_{2r}G^{(1)}_{\gamma} \|^{\ast}_{1, q; B_{2r} \setminus B_{r}} \leq Cr^{1 - n/p}.
\end{equation}
For the second term, we break the definition of $T^{II}_{2r}u$ into two integrals, one over $B_{r/2}$ and the other over $B_{2r}\setminus B_{r/2}$. That is, we write
\begin{align*}
\left(T^{II}_{2r}G_{\gamma}^{(1)}\right)_{\nu} (x) = & \int_{B_{r/2}} E_{\alpha\nu}(x - y)\left( C^{i}_{\alpha\beta}(y) \partial_{i}G^{(1)}_{\gamma\beta}(y) + \hat{D}_{\alpha\beta}(y)G^{(1)}_{\gamma\beta}(y) \right)dy\\ & + \int_{B_{2r} \setminus B_{r/2}} E_{\alpha\nu}(x - y)\left( C^{i}_{\alpha\beta}(y) \partial_{i}G^{(1)}_{\gamma\beta}(y) + \hat{D}_{\alpha\beta}(y)G^{(1)}_{\gamma\beta}(y) \right)dy\\
\equiv &\ \psi_{1}(x) + \psi_{2}(x).
\end{align*}
To estimate $\psi_{1}$, note that when $x \in B_{2r}\setminus B_{r}$ and $y \in B_{r/2}$, $E(x - y)$ is smooth in both variables and we have by \eqref{constantcoefficientgrowth} that $|E(x - y)| \leq Cr^{2-n}$ and $|D_{x} E (x - y)| \leq Cr^{1-n}$. Thus \begin{equation*}
\left| D\psi_{1} (x)\right| \leq Cr^{1-n} \Lambdabar \|G^{(1)}\|_{1, 1; B_{r/2}} \leq Cr^{1 - n/p}\| G^{(1)}\|_{1, p; B_{r/2}} \leq Cr^{1 - n/p},
\end{equation*}
where in the last inequality we again used Proposition \ref{greensbasic}(1). Integrating the above pointwise estimate, we obtain
\begin{equation*}
\left\| D \psi_{1} \right\|_{q;B_{2r}\setminus B_{r}} \leq Cr^{1-n/p}r^{n/q} = Cr^{2-n} \leq Cr^{1 - n/p}.
\end{equation*}
A similar computation yields $\| \psi_{1} \|_{q; B_{2r}\setminus B_{r}} \leq Cr^{2 - n/p}$, and we get 
\begin{equation*}
\| \psi_{1} \|^{\ast}_{1, q; B_{2r} \setminus B_{r}} \leq Cr^{1 - n/q}r^{1-n/p} \leq Cr^{1-n/p},
\end{equation*}
where the last inequality follows from the fact that $q > n$. As for $\psi_{2}$, observe that by Proposition \ref{greensrepresentation}$(1)$ and the remark after it, we have
\begin{equation*}
\| \psi_{2} \|^{\ast}_{1, q; B_{2r}} \leq C\| G^{(1)}_{\gamma} \|^{\ast}_{1, q; B_{2r}\setminus B_{r}} \leq C\|G_{\gamma}^{(1)}\|^{\ast}_{1, p; B_{3r}\setminus B_{r/2}} \leq Cr^{1 - n/p},
\end{equation*}
where in the second inequality we used interior elliptic estimates and the fact that $L^{(1)}G_{\gamma}^{(1)} = 0$ away from the origin, while in the last inequality we used a version of \eqref{annulusscaled}. Combining the two estimates above, we get
\begin{equation}\label{TGII1qannulus}
\| T^{II}_{2r}G^{(1)}_{\gamma} \|^{\ast}_{1, q; B_{2r}\setminus B_{r}} \leq Cr^{1-n/p}.
\end{equation}
Substituting \eqref{annulusscaled}, \eqref{TGI1qannulus} and \eqref{TGII1qannulus} back into \eqref{w1qdecomposition}, we get 
\begin{equation}\label{w1qannulus}
\|w\|^{\ast}_{1, q; B_{2r} \setminus B_{r}} \leq Cr^{1-n/p}.
\end{equation}
The proof of Lemma \ref{estimatesonw} is now complete in view of \eqref{w1qsmallball} and \eqref{w1qannulus}.
\end{proof}
It follows from Lemma \ref{estimatesonw} and Proposition \ref{perturbationsmall} that $\sum\limits_{l = 0}^{\infty}T^{l}_{2r}w \in W^{1, q}(B_{2r}; \RR^{N})$ and 
\begin{equation}\label{wseries1q}
\left\| \sum\limits_{l = 0}^{\infty}T^{l}_{2r}w \right\|^{\ast}_{1, q; B_{2r}} \leq \sum\limits 2^{-l} \|w\|^{\ast}_{1, q; B_{2r}} \leq Cr^{1-n/p}.
\end{equation}
Since $q > n$, by the Sobolev embedding we have
\begin{equation}\label{wseriessup}
\left| \sum\limits_{l = 0}^{\infty}T^{l}_{2r}w \right|_{0; B_{2r}} \leq C\left\| \sum\limits_{l = 0}^{\infty}T^{l}_{2r}w \right\|^{\ast}_{1, q; B_{2r}} \leq Cr^{1 - n/p} = Cr^{1 + \bar{\mu} - n}.
\end{equation}
Note that this accounts for the second term on the right-hand side of \eqref{zerothboundcentral}.

Next we turn to estimating the term $\sum\limits_{l = 1}^{\infty}T_{2r}^{l} E_{\gamma}$. Borrowing the notation in \cite{fuchs}, we write $u_{l}$ for $T_{2r}^{l}E_{\gamma}$, for $l = 0, 1, 2, \cdots$.
\begin{lemm}[cf. \cite{fuchs}, Lemma 4.3]
\label{estimatesonE}
The following estimates hold for all $l = 0, 1, 2, \cdots$.
\begin{enumerate}
\item[(i)] $|u_{l}(x)| \leq C(Cr^{\bar{\mu}})^{l} |x|^{2-n}$, for all $x \in B_{2r} \setminus \{0\}$.
\item[(ii)] $|D u_{l}(x)| \leq C(Cr^{\bar{\mu}})^{l} |x|^{1-n}$, for all $x \in B_{2r} \setminus \{0\}$.
\item[(iii)] $ |D u_{l}(x) - D u_{l}(y)| \leq C(Cr^{\bar{\mu}})^{l}|x - y|^{\bar{\mu}} \max(|x|^{1 - n - \bar{\mu}}, |y|^{1 - n - \bar{\mu}})$, for all $x, y \in B_{2r}\setminus \{0\}$.
\end{enumerate}
\end{lemm}
\begin{proof}[Proof of Lemma \ref{estimatesonE}] The proof proceeds by induction on $l$. To simplify notations, we will often drop the Greek and Latin indices when writing convolution integrals.

For $l = 0$, $u_{l}$ is just $E_{\gamma}$, and we have the desired estimates from \eqref{constantcoefficientgrowth}. Now, assume that the estimates hold for $l-1$. Then by definition we can write $u_{l}$ as 
\begin{align*}
\left(u_{l}\right)_{\nu}(x) =& \int_{B_{2r}} \partial_{i}E_{\alpha\nu}(x - y)\left( A^{ij}_{\alpha\beta}(0) - A^{ij}_{\alpha\beta}(y) \right)\partial_{j}\left(u_{l-1} \right)_{\beta}(y)dy\\
& - \int_{B_{2r}} E_{\alpha\nu}(x - y) \left( C^{i}_{\alpha\beta}(y)\partial_{i}\left( u_{l-1} \right)_{\beta}(y) + \hat{D}_{\alpha\beta}(y)\left( u_{l-1}\right)_{\beta}(y) \right)dy\\
\equiv & \left(u_{l, 1}\right)_{\nu}(x) - \left( u_{l, 2} \right)_{\nu}(x).
\end{align*}
The potential estimates in \cite{fuchs} carries over to our setting to show that $u_{l, 1}$ satisfies (ii) and (iii) under the induction hypothesis. To prove (i) for $u_{l, 1}$, observe that since
\begin{equation*}
u_{l, 1}(x) = \int_{B_{2r}} D E(x - y) \left( A(0) - A(y) \right)D u_{l-1}(y) dy,
\end{equation*}
taking absolute values and using \eqref{constantcoefficientgrowth} along with estimate (ii) for $u_{l-1}$, we obtain
\begin{align}\label{ulpreliminary}
| u_{l, 1}(x) | &\leq C \int_{B_{2r}} |x - y|^{1-n} [A]_{\bar{\mu}, B_{1}}r^{\bar{\mu}}|D u_{l-1}(y)| dy\\
\nonumber&\leq C(Cr^{\bar{\mu}})^{l-1}\Lambdabar r^{\bar{\mu}} \int_{B_{2r}} |x - y|^{1-n} |y|^{1-n}dy.
\end{align}
The evaluate the last integral, we need the following fact (cf. \cite{aubin}, Proposition 4.12): \textit{For $0 < \tau, \sigma < n$, there exists a constant $C = C(n, \sigma, \tau)$ such that for $x \in B_{1}$, we have
\begin{equation}\label{poleconvolution}
\int_{B_{1}} |x - y|^{\sigma - n}|y|^{\tau - n}dy \leq 
\left\{
\begin{array}{cc}
C|x|^{\sigma + \tau - n}&,\ \sigma + \tau < n.\\
C(1 + \log |x|)&,\ \sigma + \tau = n.\\
C &,\ \sigma + \tau > n.
\end{array}
\right.
\end{equation}}
It follows from \eqref{poleconvolution} and \eqref{ulpreliminary} that
\begin{equation}
|u_{l, 1}(x)| \leq C\Lambdabar r^{\bar{\mu}}(Cr^{\bar{\mu}})^{l-1} C |x|^{2-n} \leq C (Cr^{\bar{\mu}})^{l} |x|^{2-n}.
\end{equation}
Therefore, (i) holds for $u_{l, 1}$. Next, we note that a similar reasoning applied to the following integral and the defining integral of $u_{l, 2}$ yields (i) and (ii) for $u_{l, 2}$.
\begin{equation}
D u_{l, 2}(x) = \int_{B_{2r}} D E(x - y)\left( C(y)D u_{l-1}(y) + \hat{D}(y)u_{l-1}(y) \right)dy.
\end{equation}

Finally, to get (iii) for $u_{l, 2}$, it suffices to consider $x$ and $z$ satisfying 
\begin{equation}\label{smalldistance}
\rho \equiv |x - z| \leq \frac{1}{8}\min\{|x|, |z|\},
\end{equation}
for if the reverse inequality holds, then (iii) follows easily from (ii). Moreover, we will assume that $|x| \leq |z|$ and write $\xi = (x + z)/2$. Note that \eqref{smalldistance} then implies
\begin{equation}\label{xicomparable}
\frac{15}{16}|x| \leq |\xi| \leq \frac{17}{16}|x|.
\end{equation}
We now break $D u_{l, 2}(x) - D u_{l, 2}(z)$ into four integrals and estimate them separately. Specifically, 
\begin{align*}
D u_{l, 2}(x)  - D u_{l, 2}(z) =& \int_{B_{\rho}(\xi)} D E(x - y) F(y) dy - \int_{B_{\rho}(\xi)} D E(z - y) F(y)dy\\
 &+ \int_{B_{|\xi|/2}} \left( D E(x - y) - D E(z - y) \right)F(y)dy\\
 & + \int_{B_{2r}\setminus (B_{\rho}(\xi) \cup B_{|\xi|/2})}  \left( D E(x - y) - D E(z - y) \right)F(y)dy\\
 \equiv&\ I - II + III + IV,
\end{align*}
where we let $F(y) = C(y)D u_{l-1}(y) + D(y)u_{l-1}(y)$ to save space. Note that by the induction hypotheses, $|F(y)| \leq C(Cr^{\bar{\mu}})^{l-1}\Lambdabar (1 + |y|)|y|^{1-n} \leq C (Cr^{\bar{\mu}})^{l-1}|y|^{1-n}$.

Integrals (I) and (II) are estimated in identical ways, so we only consider (I) here. Taking absolute values, we see that
\begin{align}\label{estimateforI}
\left| (I) \right| &\leq C(Cr^{\bar{\mu}})^{l-1} \int_{B_{\rho}(\xi)}|x - y|^{1-n} |y|^{1-n}dy \\
\nonumber&\leq C(Cr^{\bar{\mu}})^{l-1}\int_{B_{3\rho}(x)} |y - x|^{1-n} |x|^{1-n}dy\\
\nonumber&\leq C(Cr^{\bar{\mu}})^{l-1}|x|^{1-n} \rho \leq C(Cr^{\bar{\mu}})^{l-1} r^{\bar{\mu}}|x|^{1-n-\bar{\mu}}\rho^{\bar{\mu}}\\
\nonumber&\leq C(Cr^{\bar{\mu}})^{l} |x|^{1-n - \bar{\mu}}\rho^{\bar{\mu}}.
\end{align}
In the second inequality above, we used \eqref{xicomparable} to infer that 
\begin{equation*}
|y| \geq |\xi| - \rho \geq |x| - 2\rho \geq \frac{3}{4}|x|,\text{ whenever $y \in B_{\rho}(\xi)$.}
\end{equation*}
Also, in the second to last inequality, we wrote $|x|^{1 - n} = |x|^{1 - n - \bar{\mu}}|x|^{\bar{\mu}}$ and estimated the second factor by $r^{\bar{\mu}}$.

Next we estimate (III). By the mean-value theorem, there exists $x' \in B_{\rho}(x)$ such that 
\begin{equation}\label{meanvalue}
\left| D E(x - y) - D E(z - y) \right|  \leq \rho \left| D^{2}E(x' - y) \right| \leq C\rho |x' - y|^{-n}.
\end{equation}
Now observe that for all $y \in B_{|\xi|/2}$ and $x' \in B_{\rho}(x)$, from \eqref{xicomparable} we have 
\begin{equation}\label{distanceforIII}
|x' - y| \geq |x| - \rho - |\xi|/2  \geq \frac{1}{4}|x|.
\end{equation}
Therefore we obtain $\left| D E(x - y) - D E(z - y) \right| \leq C\rho |x|^{-n}$, and hence
\begin{align*}
|\text{(III)}| & \leq C\rho \int_{B_{|\xi|/2}} |x|^{-n}|F(y)|dy \leq C\rho \int_{B_{|\xi|/2}} |x|^{-n} C(Cr^{\bar{\mu}})^{l-1} |y|^{1-n}dy\\
&\leq  C(Cr^{\bar{\mu}})^{l-1}\rho |x|^{-n}|\xi| \leq C(Cr^{\bar{\mu}})^{l-1}\rho^{\bar{\mu}} r^{\bar{\mu}}|x|^{-n-\bar{\mu}}|x|\\
&\leq C(Cr^{\bar{\mu}})^{l}\rho^{\bar{\mu}}|x|^{1 - n - \bar{\mu}},
\end{align*}
where in the second to last inequality we used \eqref{xicomparable} and that $|x|^{-n} \leq |x|^{-n-\bar{\mu}}r^{\bar{\mu}}$.

Finally we estimate (IV). Again we use the mean-value theorem to find $x' \in B_{\rho/2}(\xi)$ such that \eqref{meanvalue} holds. To estimate the right-hand side in \eqref{meanvalue} in a useful way, we notice that for $y \in B_{2r} \setminus (B_{\rho}(\xi) \cup B_{|\xi|/2})$ and $x' \in B_{\rho/2}(\xi)$, we have $|x' - \xi| \leq \rho/2 \leq |y - \xi|/2$ and hence 
\begin{equation*}
|y - x'| \geq |y - \xi| - |x' - \xi| \geq \frac{1}{2}|y - \xi|.
\end{equation*}
Therefore $|x' - y|^{-n} \leq 2^{n}|y - \xi|^{-n}$, and (IV) can be estimated as follows.
\begin{align*}
|\text{(IV)}|&\leq C\rho \int_{B_{2r} \setminus (B_{\rho}(\xi) \cup B_{|\xi|/2})} |\xi - y|^{-n} |F(y)|dy\\
&\leq C\rho \int_{B_{2r} \setminus (B_{\rho}(\xi) \cup B_{|\xi|/2})}|\xi - y|^{-n}C (Cr^{\bar{\mu}})^{l-1}|y|^{1-n}dy\\
&\leq C\rho (Cr^{\bar{\mu}})^{l-1} |\xi|^{1-n} \int_{B_{2r} \setminus (B_{\rho}(\xi) \cup B_{|\xi|/2})}r^{1 - \bar{\mu}}  |\xi - y|^{-n - 1 + \bar{\mu}}dy\\
&\leq C\rho (Cr^{\bar{\mu}})^{l-1} |x|^{1-n} r^{1 - \bar{\mu}} \rho^{\bar{\mu}-1}\\
&\leq C\rho^{\bar{\mu}} (Cr^{\bar{\mu}})^{l-1} |x|^{1 - n - \bar{\mu}} r \leq C\rho^{\bar{\mu}}(Cr^{\bar{\mu}})^{l}|x|^{1 - n - \bar{\mu}}.
\end{align*}
In the third inequality above, we used the fact that $|y| \geq |\xi|/2$ on the domain of integration, and that
\begin{equation}
|\xi - y|^{-n} = |\xi - y|^{-n-1 + \bar{\mu}}|\xi - y|^{1 - \bar{\mu}} \leq r^{1 - \bar{\mu}}|\xi - y|^{-n-1 + \bar{\mu}}.
\end{equation}
Having estimated (I), (II), (III) and (IV), we conclude that (iii) holds for $u_{l, 2}$, thus finishing the induction step. The proof of Lemma \ref{estimatesonE} is now complete.
\end{proof}
We now conclude the proof of \eqref{zerothboundcentral} for $G^{(1)}$ as follows. Requiring that $r \leq R_{0} \equiv \min\{r_{0}/2, 1/4, \left( 2C \right)^{-1/\bar{\mu}} \}$, with $C$ as in the previous lemma, then by the estimate (i) in Lemma \ref{estimatesonE}, we have
\begin{equation}
\sum\limits_{l = 1}^{\infty}T_{2r}^{l}E_{\gamma}(x) \leq C\sum\limits_{l = 1}^{\infty}2^{-l}|x|^{2-n} \leq C|x|^{2-n}.
\end{equation}
Combining this with \eqref{wseriessup} and recalling \eqref{gseries}, we arrive at
\begin{equation}
|G^{(1)}_{\gamma}(x)| \leq |E_{\gamma}(x)| + C|x|^{2-n} + Cr^{1 + \bar{\mu} - n} \leq C|x|^{2-n} + Cr^{1 + \bar{\mu} - n}.
\end{equation}
The proof of \eqref{zerothboundcentral} for $G^{(1)}$ is now complete.
\end{proof}
\begin{proof}[Proof of the first-order derivative estimate \eqref{firstderivative}]
Having proven \eqref{zerothboundcentral}, we will show that \eqref{firstderivative} follows from standard $C^{1, \bar{\mu}}$-estimates for elliptic systems and a scaling argument. As in the previous proof, we assume without loss of generality that $y  = 0$, and we write $G_{\gamma}^{(1)}(\cdot)$ for $\left( G_{\gamma\alpha}(\cdot, 0) \right)_{1 \leq \alpha \leq N}$.

Fix $r \leq R_{0}$ and choose a point $z \in B_{r}$. We let $\rho = |z|/2$ and rescale $G^{(1)}_{\gamma}$ by letting $\tilde{G}(x) \equiv G^{(1)}_{\gamma}(z + \rho x)$. Since $G^{(1)}_{\gamma}$ satisfies $L^{(1)}G^{(1)}_{\gamma} = 0$ on $B_{\rho}(z)$, we infer that $\tilde{L}\tilde{G} = 0$ on $B_{1}$, where the operator $\tilde{L}$ has the form
\begin{equation}\label{scaledoperator}
\left( \tilde{L}u \right)_{\alpha} = \partial_{i}\left( \tilde{A}^{ij}_{\alpha\beta}\partial{j}u_{\beta} \right) + \tilde{C}^{i}_{\alpha\beta}\partial_{i}u_{\beta} + \tilde{D}_{\alpha\beta}u_{\beta},
\end{equation}
with coefficients given by
\begin{equation}
\left\{
\begin{array}{cl}
\tilde{A}^{ij}_{\alpha\beta}(x) &= A^{ij}_{\alpha\beta}(z + \rho x)\\
\tilde{C}^{i}_{\alpha\beta}(x) &= \rho C^{i}_{\alpha\beta}(z + \rho x)\\
\tilde{D}_{\alpha\beta}(x) &= \rho^{2}\hat{D}_{\alpha\beta}(z + \rho x).
\end{array}
\right.
\end{equation}
We see immediately from the definition that $\tilde{L}$ also satisfies (H1) to (H3) with the same parameters $\lambdabar, \Lambdabar$ and $\bar{\mu}$. Since $\tilde{L}\tilde{G} = 0$ on $B_{1}$, Schauder theory implies that $\tilde{G} \in C^{1, \bar{\mu}}_{\text{loc}}(B_{1}; \RR^{N})$. Furthermore, the following estimate holds. 
\begin{equation}\label{scaledschauder}
| D \tilde{G} |_{0; B_{1/2}} \leq C\|\tilde{G}\|_{2; B_{1}}.
\end{equation}
Scaling back and recalling the definition of $\rho$, we obtain
\begin{equation}\label{unscaledschauder}
|z| |D G^{(1)}_{\gamma}|_{0; B_{|z|/4}(z)} \leq C |z|^{-n/2}\|G^{(1)}_{\gamma}\|_{2; B_{|z|/2}(z)}.
\end{equation}
Since $|z| < r$, it follows that $B_{|z|/2}(z) \subset B_{2r}\setminus \{0\}$. Thus we may use the bound \eqref{zerothboundcentral} for $G^{(1)}_{\gamma}$ to estimate the right-hand side of \eqref{unscaledschauder} by 
\begin{equation*}
C|z|^{-n/2}\left( C|z|^{2-n} + r^{1 + \bar{\mu} - n} \right)|z|^{n/2} = C\left( |z|^{2-n} + r^{1 + \bar{\mu} - n} \right).
\end{equation*}
Plugging this back into \eqref{unscaledschauder} and dividing both sides by $|z|$, we arrive at
\begin{equation}
|D G^{(1)}_{\gamma}|_{0; B_{|z|/4}(z)} \leq C\left( |z|^{1-n} + |z|^{-1}r^{1 + \bar{\mu} - n} \right).
\end{equation}
In particular, \eqref{firstderivative} holds, and we are done.
\end{proof}

Before proceeding with the proof of Theorem \ref{greensgrowth}, let us present some further properties of $G^{(1)}$ that we will need.
\begin{prop}\label{greensdifferentiable}
\begin{enumerate}
\item[(1)] $G^{(1)}: B_{1} \times B_{1} \to \RR^{N^{2}}$ is locally $C^{1}$ on $B_{1} \times B_{1} \setminus \Delta$. Moreover, for $x \neq y$, we have
\begin{equation}\label{c1symmetry}
\partial_{i}^{(1)}G^{(1)}_{\alpha\beta}(x, y) = \partial_{i}^{(2)}G'^{(1)}_{\beta\alpha}(y, x),
\end{equation}
where $\partial_{i}^{(1)}$ (resp., $\partial_{i}^{(2)}$) denotes partial differentiation with respect to $i$-th variable in the first slot (resp., second slot), and $G'^{(1)}$ denotes the fundamental solution for $L'^{(1)}$, the adjoint of $L^{(1)}$.
\item[(2)] For $p \in [1, \infty)$ and $F  = (F_{\alpha}^{i}) \in L^{p}(B_{1}; \RR^{nN})$, let $u \in W^{1, p}_{0}(B_{1}; \RR^{N})$ be the unique solution to $L^{(1)}u = \partial_{i}F^{i}$. Then, for $H^{n}$-a.e. $y \in B_{1}$, we have
\begin{equation}\label{G1divrepresentation}
u_{\gamma}(y) = \int_{B_{1}} \partial_{i}^{(2)}G^{(1)}_{\alpha\gamma}(y, x)F_{\alpha}^{i}(x)dx.
\end{equation}
\item[(3)] For $p \in [1, \infty)$ and $f = (f_{\alpha}) \in L^{p_{\ast}}(B_{1}; \RR^{N})$, let $u \in W^{1, p}_{0}(B_{1}; \RR^{N})$ be the unique solution to $L^{(1)}u = f$. Then, for $H^{n}$-a.e. $x \in B_{1}$, we have
\begin{equation}\label{G1nondivrepresentation}
u_{\gamma}(x) = \int_{B_{1}}G^{(1)}_{\alpha\gamma}(x, y)f_{\alpha}(y)dy.
\end{equation}
\end{enumerate}
\end{prop}
\begin{proof}[Proof of Proposition \ref{greensdifferentiable}(1)]
For $(1)$, we start by noting that \eqref{c1symmetry} is an immediate consequence of Proposition \ref{greensbasic}$(3)$. Next, to see that $G^{1} \in C^{1}_{\text{loc}}(B_{1} \times B_{1}\setminus \Delta; \RR^{N^{2}})$, we will argue that for each $\rho > 0$, $G^{(1)}$ is $C^{1}$ on the set $A_{\rho} \equiv \{ (x, y) \in B_{1} \times B_{1}|\ |x - y| > \rho \}.$ To that end, let us fix $(x_{0}, y_{0}) \in A_{\rho}$, and look at the family of functions 
\begin{equation*}
\left\{ G^{(1)}(\cdot, y) \right\}_{y \in B_{\rho/2}(y_{0})}.
\end{equation*} 
Below, for convenience, we denote the restriction of $G^{(1)}(\cdot, y)$ to $B_{\rho/2}(x_{0})$ by $u_{y}$. 
\begin{claim*}
The family $\{u_{y}\}_{y \in B_{\rho/2}(y_{0})}$ is bounded in the $ |\cdot|_{1, \bar{\mu}; B_{\rho/3}(x_{0})}$-norm.
\end{claim*}
\begin{proof}[Proof of the claim]
Notice that since $|x_{0} - y_{0}| > \rho$, for all $y \in B_{\rho/2}(y_{0})$, the functions $u_{y}$ actually solve $L^{(1)}u_{y} = 0$ on $B_{\rho/2}(x_{0})$. By the interior $C^{1, \bar{\mu}}$-estimates for divergence-form elliptic systems, we infer that
\begin{equation}\label{uniformc1holder}
|u_{y}|_{1, \bar{\mu}; B_{\rho/3}(x_{0})} \leq C\|u_{y}\|_{1, 2; B_{\rho/2}(x_{0})}.
\end{equation}
Since $|x_{0} - y_{0}| > \rho$, we see from Proposition \ref{greensbasic} (2) with $q = 2$ that
\begin{align*}
\|u_{y}\|_{1, 2; B_{\rho/2}(x_{0})} &\leq \|u_{y}\|_{1, 2; B_{1} \setminus B_{\rho/2}(y_{0})} \leq C_{\rho}.
\end{align*}
 The claim is proved upon combining the above inequalities.
\end{proof}
From the claim we see that for each sequence $y_{i} \to y_{0}$, a subsequence of $u_{y_{i}}$ converges in $C^{1}(B_{\rho/3}(x); \RR^{N^{2}})$ to a limit in $C^{1, \bar{\mu}}(B_{\rho/3}(x); \RR^{N^{2}})$. By Proposition \ref{greensbasic}(4), this limit must be $u_{y_{0}}$, and hence we conclude that $u_{y}$ converges to $u_{y_{0}}$ in $C^{1}$-topology as $y \to y_{0}$. Next we let $(x_{i}, y_{i}) \to (x_{0}, y_{0})$ as $i \to \infty$. Then for $i$ sufficiently large, we have
\begin{align*}
&|\partial^{(1)}_{i} G^{(1)}(x_{i}, y_{i}) - \partial^{(1)}_{i} G^{(1)}(x_{0}, y_{0})|\\
& \leq \ |\partial^{(1)}_{i} G^{(1)}(x_{i}, y_{i}) - \partial^{(1)}_{i} G^{(1)}(x_{i}, y_{0})|+ |\partial^{(1)}_{i} G^{(1)}(x_{i}, y_{0}) - \partial^{(1)}_{i} G^{(1)}(x_{0}, y_{0})|\\
&\leq \ |u_{y_{i}} - u_{y_{0}}|_{1, 0; B_{\rho/3}(x_{0})} + |x_{i} - x_{0}|^{\bar{\mu}}[u_{y_{0}}]_{1, \bar{\mu}; B_{\rho/3}(x_{0})} \longrightarrow \ 0 \text{ as }i \to \infty.
\end{align*}
Thus we've shown that $\partial^{(1)}_{i}G^{(1)}$ is continuous on $A_{\rho}$. Applying the same argument to $G'^{(1)}$ and using the relation \eqref{c1symmetry}, we see that the same conclusion holds for $\partial_{i}^{(2)}G^{(1)}$.
In view of this and Proposition \ref{greensbasic}(4), we conclude that $G^{(1)}$ is $C^{1}$ on $A_{\rho}$ for any $\rho > 0$. The proof of Proposition \ref{greensdifferentiable}(1) is now complete.
\end{proof}

Next we turn to assertions (2) and (3) of Proposition \ref{greensdifferentiable}. For this purpose we need to introduce the mollified fundamental solutions, denoted by $G^{(1)\rho}(\cdot, y) = (G^{(1)\rho}_{\alpha\beta}(\cdot, y))$ and defined for $\rho > 0$ by letting $G^{(1)\rho}_{\alpha} = (G^{(1)\rho}_{\alpha\beta})_{1 \leq \beta \leq N}$ be the unique solution to 
\begin{equation}
L^{(1)}u = \frac{e_{\alpha}}{|B_{\rho}|}\chi_{B_{\rho}(y)}dH^{n}.
\end{equation}
The mollified version of the adjoint fundamental solution $G'^{(1)}$ is defined similarly.
In what follows, we will frequently use the following property of $G^{(1)\rho}$.
\begin{prop}[\cite{fuchs}, Corollary to Theorem 4]\label{mollifiedgreens}
For $x, y \in B_{1}$ and $0 < \rho < \dist(y, \partial B_{1})$, the following holds.
\begin{equation}\label{mollificationidentity}
G^{(1)\rho}_{\alpha\beta}(x, y) = \fint_{B_{\rho}(y)}G'^{(1)}_{\beta\alpha}(z, x)dz = \fint_{B_{\rho}(y)} G^{(1)}_{\alpha\beta}(x, z) dz,
\end{equation}
where the $\fint_{B_{\rho}(y)}$ denotes the average over $B_{\rho}(y)$.
\end{prop}
\begin{rmk}
The proofs given in \cite{fuchs} were based solely on standard $W^{1, p}$-estimates, and therefore apply to our setting.
\end{rmk}
\begin{proof}[Proof of Proposition \ref{greensdifferentiable} (2)]
To simplify notations, we fix an index $\gamma$ and denote $G^{(1)\rho}_{\gamma}(\cdot, y)$ (resp., $G^{(1)}_{\gamma}(\cdot, y)$) by $u^{\rho}_{y}$ (resp., $u_{y}$), and $G'^{(1)\rho}_{\gamma}(\cdot, y)$ (resp., $G'^{1}_{\gamma}(\cdot, y)$) by $v^{\rho}_{y}$ (resp., $v_{y}$). Let $u$ be as in the hypothesis of (2), and use $v^{\rho}_{y}$ as a test function in the equation $L^{(1)}u = \partial_{i}F^{i}$. Then we get
\begin{align*}
\int_{B_{1}}A^{ij}_{\alpha\beta}\partial_{j}u_{\beta} \partial_{i}v^{\rho}_{y, \alpha} - v^{\rho}_{y, \alpha}(C^{i}_{\alpha\beta}\partial_{i}u_{\beta} + \hat{D}_{\alpha\beta}u_{\beta}) = \int_{B_{1}} F^{i}_{\alpha}\partial_{i}v^{\rho}_{y, \alpha}.
\end{align*}
Note that the left-hand side can understood as testing the system $L'^{(1)}v^{\rho}_{y} = e_{\gamma}|B_{\rho}|^{-1}\chi_{B_{\rho}(y)}dH^{n}$ against $u$. Thus we get
\begin{equation*}
\int_{B_{1}}A^{ij}_{\alpha\beta}\partial_{j}u_{\beta} \partial_{i}v^{\rho}_{y, \alpha} - v^{\rho}_{y, \alpha}(C^{i}_{\alpha\beta}\partial_{i}u_{\beta} + \hat{D}_{\alpha\beta}u_{\beta}) = \fint_{B_{\rho}(y)} u_{\gamma}.
\end{equation*}
Putting the two above equations together, we obtain
\begin{equation}\label{averagerepresentation}
\fint_{B_{\rho}(y)} u_{\gamma}(z)dz = \int_{B_{1}}F^{i}_{\alpha}(x)\partial_{i}v^{\rho}_{y, \alpha}(x)dx.
\end{equation}
To proceed we observe that, by Proposition \ref{mollifiedgreens} applied to $G'^{(1)\rho}$, we have
\begin{equation*}
v_{y, \alpha}^{\rho}(x) = \fint_{B_{\rho}(y)}  G^{(1)}_{\alpha\gamma}(z, x)dz.
\end{equation*}
Thus, by the estimates \eqref{zerothboundcentral} (for $G^{(1)}$) and \eqref{firstderivative}, we may differentiate under the integral to get (cf. the proof of Lemma 4.1 in \cite{gt})
\begin{equation*}
\partial_{i}v^{\rho}_{y, \alpha}(x) = \fint_{B_{\rho}(y)} \partial^{(2)}_{i} G^{(1)}_{\alpha\gamma}(z, x) dz, \text{ for } x \in B_{1}.
\end{equation*}
Going back to \eqref{averagerepresentation}, we see that the right-hand side equals
\begin{equation}\label{doubleintegralrepresentation}
\int_{B_{1}} \left(  \fint_{B_{\rho}(y)} F^{i}_{\alpha}(x) \partial^{(2)}_{i}G^{(1)}_{\alpha\gamma}(z, x) dz \right) dx.
\end{equation}
By \eqref{c1symmetry} with $G'^{(1)}$ in place of $G^{(1)}$ and using Proposition \ref{greensbasic}(1), we see that $\|\partial_{i}^{(2)}G^{(1)}(\cdot, x)\|_{1; B_{1}}$ is bounded independent of $x$. Thus, since we also know that $F \in L^{p}$, we may switch the order of integration in \eqref{doubleintegralrepresentation} to get
\begin{align*}
\int_{B_{1}} \left(  \fint_{B_{\rho}(y)} F^{i}_{\alpha}(x) \partial^{(2)}_{i}G^{(1)}_{\alpha\gamma}(z, x) dz \right) dx
& = \fint_{B_{\rho}(y)} \left(  \int_{B_{1}} F^{i}_{\alpha}(x) \partial^{(2)}_{i}G^{(1)}_{\alpha\gamma}(z, x) dx \right) dz.
\end{align*}
Moreover, the inner-integral on the right-hand side is an $L^{1}$-function in $z$, and thus by the Lebesgue differentiation theorem and recalling \eqref{averagerepresentation}, for $H^{n}$-a.e. $y$ we have
\begin{equation}
\lim\limits_{\rho \to 0}\fint_{B_{\rho}(y)}u_{\gamma}(z)dz =  \int_{B_{1}} F^{i}_{\alpha}(x) \partial^{(2)}_{i}G^{(1)}_{\alpha\gamma}(y, x) dx.
\end{equation}
However, since $u \in L^{p}(B_{1}; \RR^{N})$, the left-hand side of the above identity equals $u_{\gamma}(y)$ for $H^{n}$-a.e. $y$. Hence, for $H^{n}$-a.e. $y$ there holds
\begin{equation*}
u_{\gamma}(y) = \int_{B_{1}} F^{i}_{\alpha}(x) \partial^{(2)}_{i}G^{(1)}_{\alpha\gamma}(y, x) dx.
\end{equation*}
\end{proof}
\begin{proof}[Proof of Proposition \ref{greensdifferentiable} (3)]
We will use the same notations as in the previous proof. Let $u$ be as in the hypothesis of (3) and use $v^{\rho}_{y}$ as a test function in the system $L^{(1)}u = f$. Then we get the following analogue of \eqref{averagerepresentation}.
\begin{equation}\label{averagedrepresentationzero}
\fint_{B_{\rho}(y)} u_{\gamma}(z)dz = \int_{B_{1}} f_{\alpha}(x)  v^{\rho}_{y, \alpha}(x)dx.
\end{equation}
By \eqref{mollificationidentity} with $G'^{(1)}$ in place of $G^{(1)}$, we find that
\begin{equation}
\int_{B_{1}} f_{\alpha}(x)  v^{\rho}_{y, \alpha}(x)dx = \int_{B_{1}} \left( \fint_{B_{\rho}(y)} f_{\alpha}(x)G_{\alpha\gamma}^{(1)}(z, x) dz \right) dx.
\end{equation}
As in the previous proof, the order of integration on the right-hand side can be switched, resulting in
\begin{equation}
\int_{B_{1}} f_{\alpha}(x)  v^{\rho}_{y, \alpha}(x)dx = \fint_{B_{\rho}(y)} \left( \int_{B_{1}} f_{\alpha}(x)G_{\alpha\gamma}^{(1)}(z, x) dx \right) dz.
\end{equation}
Combining this wit \eqref{averagedrepresentationzero}, letting $\rho$ go to zero and using the Lebesgue differentiation theorem, we see that for $H^{n}$-a.e. y $\in B_{1}$, 
\begin{equation}
u_{\gamma}(y) = \int_{B_{1}}f_{\alpha}(x)G^{(1)}_{\alpha\gamma}(y, x)dx.
\end{equation}
\end{proof}
\begin{rmk}
The same proof as the one given above shows that Proposition \ref{greensdifferentiable}(3) holds with $L$ in place of $L^{(1)}$ and $G$ in place of $G^{(1)}$.
\end{rmk}

At this point we are almost ready to prove Theorem \ref{greensgrowth}. The proof will be another perturbation argument, but this time based on estimates for $G^{(1)}$ instead of $E$. For that purpose, as in the proof of Theorem \ref{growthstepone}, we need to introduce a perturbation operator and show that locally it is a contraction mapping.
\begin{prop}\label{finalperturbationproperties}
Take a ball $B_{r}(x_{0}) \subset B_{1}$. For $u: B_{r}(x_{0}) \to \RR^{N}$ and $x \in B_{r}(x_{0})$, define 
\begin{align}\label{finalperturbation}
\left(T_{r}u \right)_{\gamma} (x) =\ & - \int_{B_{r}(x_{0})} \partial^{(2)}_{i}G^{(1)}_{\alpha\gamma}(x, y)B^{i}_{\alpha\beta}(y)u_{\beta}(y)dy\\
\nonumber&\ - K\int_{B_{r}(x_{0})} G^{(1)}_{\alpha\gamma} (x, y) u_{\alpha}(y)dy.
\end{align}
Then the following are true.
\begin{enumerate}
\item[(1)] If $u \in W^{1, p}(B_{r}(x_{0}); \RR^{N})$, then $v = T_{r}u$ is a weak solution in $W^{1, p}(B_{r}(x_{0}); \RR^{N})$ to
\begin{equation}\label{correctionsystem}
\left(L^{(1)}v \right)_{\alpha} = -\partial_{i}\left( B^{i}_{\alpha\beta}u_{\beta} \right) - Ku_{\alpha}.
\end{equation}

\item[(2)] Write $\|T_{r}\|_{p}$ for the operator norm of $T_{r}: W^{1, p}(B_{r}(x_{0}); \RR^{N}) \to W^{1, p}(B_{r}(x_{0}); \RR^{N})$ with respect to the norm $\|\cdot\|^{\ast}_{1, p; B_{r}(x_{0})}$. Then there exists $r_{1} = r_{1}(n, N, c_{0}, \lambdabar, \Lambdabar, \bar{\mu})$ such that whenever $r \leq r_{1}$, we have $\|T_{r}\|_{p} < 1/2$.
\end{enumerate}
\end{prop}
\begin{proof}
Assertion (1) follows straight from Proposition \ref{greensdifferentiable}(2)(3). Thus we will focus on proving assertion (2). To that end, notice that the integrals in \eqref{finalperturbation} in fact make sense for $x \in B_{1}$, and Proposition \ref{greensdifferentiable} tells us that the resulting function on $B_{1}$, still denoted $T_{r}u$, is the weak solution in $W^{1, p}_{0}(B_{1}; \RR^{N})$ to \eqref{correctionsystem}, with the right-hand side extended to be zero outside of $B_{r}(x_{0})$. Thus, the global $W^{1, p}$-estimates implies that
\begin{align}
\| D T_{r}u \|_{p; B_{1}} &\leq C\left( \Lambdabar \|u\|_{p; B_{r}(x_{0})} + K \| u \|_{p_{\ast}; B_{r}(x_{0})}\right) \\
\nonumber & \leq C(\Lambdabar + Kr) \|u\|_{p; B_{r}(x_{0})}.
\end{align}
Moreover, since $T_{r}u \in W^{1, p}_{0}(B_{1}; \RR^{N})$, by the Sobolev inequality we have $\|T_{r}u\|_{p^{\ast}; B_{1}} \leq C\| D T_{r}u \|_{p; B_{1}}$, and thus
\begin{align*}
r^{-n/p}\| T_{r}u \|_{p; B_{r}(x_{0})} &\leq r^{1 - n/p}\| T_{r}u \|_{p^{\ast}; B_{r}(x_{0})} \leq r^{1 - n/p}\| T_{r}u\|_{p^{\ast}; B_{1}}\\
\nonumber &\leq Cr^{1 - n/p}\| D T_{r}u \|_{p; B_{1}} \leq Cr r^{-n/p}\| u \|_{p; B_{r}(x_{0})}.
\end{align*}
From the above inequalities and the definition of $\|\cdot\|^{\ast}_{1, p; B_{r}(x_{0})}$, we arrive at
\begin{equation}
\|T_{r}u\|^{\ast}_{1, p;B_{r}(x_{0})} \leq Cr \| u \|^{\ast}_{1, p; B_{r}(x_{0})}.
\end{equation}
Taking $r_{1} = (2C)^{-1}$ completes the proof.
\end{proof}

\begin{proof}[Proof of Theorem \ref{greensgrowth}] 
Again we assume that $y  = 0$. Moreover, we use the same notations as in the proof of Theorem \ref{growthstepone}. Finally, when writing convolution integrals, we will often omit the Latin and Greek indices. 

Taking a radius $r \leq \min\{r_{1}/4, R_{0}/2, 1/8\}$, with $R_{0}$ given by Theorem \ref{growthstepone} and $r_{1}$ by Proposition \ref{finalperturbationproperties}, we write $G_{\gamma}$ as a perturbation of $G^{(1)}$ by letting 
\begin{equation}
w = G_{\gamma} - T_{2r}G_{\gamma} - G^{(1)}_{\gamma},
\end{equation}
where we take $x_{0} = 0$ in the definition of $T_{2r}$. Then a straightforward computation shows that $L^{(1)}w = 0$ on $B_{2r}$. Next, by our choice of $r$, we have $\| T_{2r} \|_{p} < 1/2$. Thus, as in the proof of Theorem \ref{growthstepone}, we can write
\begin{equation}\label{Gdecomposition}
G_{\gamma} = G^{(1)}_{\gamma} + \sum\limits_{l = 1}^{\infty}T_{2r}^{l}G^{(1)}_{\gamma} + \sum\limits_{l = 0}^{\infty}T_{2r}^{l}w. 
\end{equation}
\begin{claim*}
$\| w \|^{\ast}_{1, q; B_{2r}} \leq Cr^{1 - n/p}$.
\end{claim*}
\begin{proof}[Proof of the claim]
The proof is very similar to that of Lemma \ref{estimatesonw}. Therefore we will only emphasize the necessary modifications and sketch the rest of the argument.

As in the proof of Lemma \ref{estimatesonw}, we estimate separately the norm of $w$ over $B_{r}$ and $B_{2r}\setminus B_{r}$. For $\| w\|^{\ast}_{1, q; B_{r}}$, since $L^{(1)}w = 0$ on $B_{2r}$, we have
\begin{equation}
\|w\|^{\ast}_{1, q; B_{r}} \leq C\|w\|^{\ast}_{1, p; B_{2r}}\leq C\left( \| G^{(1)}_{\gamma} \|_{1, p; B_{2r}}^{\ast} + \| G_{\gamma} \|_{1, p; B_{2r}}^{\ast} + \| T_{2r}G_{\gamma} \|_{1, p; B_{2r}}^{\ast}\right).
\end{equation}
Again following \cite{fuchs}, p. 520, we get $\| G^{(1)}_{\gamma} \|_{1, p; B_{2r}}^{\ast} + \| G_{\gamma} \|_{1, p; B_{2r}}^{\ast} \leq Cr^{1-n/p}$. Since $T_{2r}$ is a bounded operator with respect to $\|\cdot\|_{1, p; B_{2r}}^{\ast}$, a similar estimate holds for $\| T_{2r}G_{\gamma} \|_{1, p; B_{2r}}^{\ast}$, and therefore we obtain $\| w \|_{1, q; B_{r}}^{\ast} \leq Cr^{1-n/p}$.

To estimate $\|w\|_{1, q; B_{2r}\setminus B_{r}}$, note that, by definition and the triangle inequality,
\begin{equation}\label{newwdecomposition}
\|w\|_{1, q; B_{2r}\setminus B_{r}}^{\ast} \leq \| G_{\gamma} \|^{\ast}_{1, q; B_{2r}\setminus B_{r}} + \| G^{(1)}_{\gamma} \|^{\ast}_{1, q; B_{2r}\setminus B_{r}} + \|T_{2r}G_{\gamma}\|^{\ast}_{1, q; B_{2r}\setminus B_{r}}.
\end{equation}
The first and second terms on the right-hand side are treated as in \cite{fuchs} (Lemma 4.1) with the help of Proposition \ref{greensbasic} and interior $W^{1, p}$-estimates, and we get 
\begin{equation}\label{neweasyannulus}
\|G^{(1)}_{\gamma}\|^{\ast}_{1, q; B_{2r} \setminus B_{r}} + \|G_{\gamma}\|^{\ast}_{1, q; B_{2r}\setminus B_{r}} \leq Cr^{1-n/p}.
\end{equation} 
The third term on the right-hand side of \eqref{newwdecomposition} requires some care, as the estimates we have for $G^{(1)}$ are not as precise as the ones we have for $E$. For convenience, we denote the two integrals in the definition of $T_{2r}G_{\gamma}$ by $T^{I}_{2r}G_{\gamma}$ and $T^{II}_{2r}G_{\gamma}$, respectively. 

For $T^{I}_{2r}G_{\gamma}$, by definition we have
\begin{align*}
T^{I}_{2r}G_{\gamma}(x) =& - \int_{B_{2r} \setminus B_{r/2}} \partial_{i}^{(2)}G^{(1)}_{\gamma}(x, y) B(y)G_{\gamma}(y)dy\\
&\ - \int_{B_{r/2}} \partial_{i}^{(2)}G^{(1)}_{\gamma}(x, y)B(y)G_{\gamma}(y)dy\\
&\ \equiv \varphi_{1}(x) + \varphi_{2}(x).
\end{align*} 
By Proposition \ref{greensdifferentiable} and elliptic estimates, applied with $F^{i}_{\alpha}(y) = B^{i}_{\alpha\beta}(y)G_{\gamma\beta}(y)\chi_{B_{2r}\setminus B_{r/2}}(y)$, we have $\|\varphi_{1}\|_{1, q; B_{1}} \leq C\Lambdabar \| G_{\gamma} \|_{1, q; B_{2r}\setminus B_{r/2}}$.
Thus we infer that
\begin{align*}
\| \varphi_{1} \|^{\ast}_{1, q; B_{2r}\setminus B_{r}} &\leq Cr^{1-n/q}\| \varphi_{1} \|_{1, q; B_{1}} \leq Cr^{1 - n/q}\| G_{\gamma} \|_{1, q; B_{2r}\setminus B_{r/2}},
\end{align*}
and the right-most term is bounded by $Cr^{1-n/p}$ by a version of \eqref{neweasyannulus}.

As for $\varphi_{2}$, note that again by Proposition \ref{greensdifferentiable}, this time applied with  $F^{i}_{\alpha}(y) = B^{i}_{\alpha\beta}(y)G_{\gamma\beta}(y)\chi_{B_{r/2}}(y)$, the function $\varphi_{2}$ is a weak solution to $L^{(1)}u = 0$ on $B_{4r}\setminus B_{r/2}$. Thus, using elliptic estimates and Propositions \ref{greensbasic} and \ref{greensdifferentiable}, we have
\begin{align*}
\|\varphi_{2}\|^{\ast}_{1, q; B_{2r}\setminus B_{r}} &\leq C\| \varphi_{2} \|^{\ast}_{1, p; B_{3r}\setminus B_{2r/3}} \leq Cr^{1 - n/p}\|\varphi_{2}\|_{1, p; B_{1}}\\
&\leq C\Lambdabar r^{1 - n/p}\| G_{\gamma} \|_{1, p; B_{r/2}} \leq Cr^{1-n/p}.
\end{align*}
Combining the bounds for $\varphi_{1}$ and $\varphi_{2}$, we arrive at
\begin{equation}
\|T_{2r}^{I}G_{\gamma}\|^{\ast}_{1, q; B_{2r}\setminus B_{r}} \leq Cr^{1 - n/p}. 
\end{equation}
By similar reasoning, we get $\| T_{2r}^{II}G_{\gamma} \|^{\ast}_{1, q; B_{2r} \setminus B_{r}} \leq Cr^{1-n/p}$ and hence $\|T_{2r}G_{\gamma}\|^{\ast}_{1, q; B_{2r}\setminus B_{r}} \leq Cr^{1 - n/p}$.

In view of \eqref{newwdecomposition} and the estimates derived above, we have $\|w\|^{\ast}_{1, q; B_{2r}\setminus B_{r}} \leq Cr^{1 - n/p}$, and the claim is proved upon recalling $\|w\|^{\ast}_{1, q; B_{r}} \leq Cr^{1 - n/p}$.
\end{proof}
From the claim and our choice of $r$, we see that 
\begin{equation}
\left\| \sum\limits_{l=0}^{\infty} T_{2r}^{l}w \right\|_{1, q; B_{2r}} \leq \sum\limits_{l = 0}^{\infty} \|T_{2r}\|_{q}^{l}\|w\|_{1, q; B_{2r}} \leq Cr^{1- n/p}.
\end{equation}
Hence, as before, the Sobolev embedding then yields 
\begin{equation}\label{finalwsup}
\left|\sum\limits_{l=0}^{\infty} T_{2r}^{l}w\right|_{0; B_{2r}} \leq Cr^{1 - n/p} = Cr^{1 + \bar{\mu} - n}.
\end{equation}
We next treat the term $\sum\limits_{l = 1}^{\infty}T_{2r}^{l}G^{(1)}_{\gamma}$. Below we write $u_{l}$ for $T_{2r}^{l}G^{(1)}_{\gamma}$. Then by definition, 
\begin{align}\label{finaliteration}
\left( u_{l}\right)_{\gamma} (x) =\ & - \int_{B_{2r}} \partial^{(2)}_{i}G^{(1)}_{\alpha\gamma}(x, y)B^{i}_{\alpha\beta}(y)u_{l-1, \beta}(y)dy\\
\nonumber&\ - K\int_{B_{2r}} G^{(1)}_{\alpha\gamma} (x, y) u_{l-1, \alpha}(y)dy.
\end{align}
\begin{claim*}
For $l = 0, 1, 2, \cdots$, and $x \in B_{2r}\setminus \{0\}$ the following estimate holds.
\begin{equation}\label{finalinduction}
|u_{l}(x)| \leq C(Cr^{\bar{\mu}})^{l}|x|^{2-n}.
\end{equation}
\end{claim*}
\begin{proof}[Proof of the claim]
For $l = 0$, $u_{l}$ reduces to $G^{(1)}_{\gamma}$ and \eqref{finalinduction} follows from Theorem \ref{growthstepone} and the fact that $r < R_{0}/2$.  Next, assuming \eqref{finalinduction} for $l-1$, by the recursive relation \eqref{finalinduction} we obtain
\begin{align*}
|(u_{l})_{\gamma}(x)| \leq&\ \Lambdabar \int_{B_{2r}} |D G^{(1)}_{\gamma}(x, y)||u_{l-1}(y)|dy+ K\int_{B_{2r}} |G^{(1)}_{\gamma}(x, y)||u_{l-1}(y)|dy\\
\leq&\ C\Lambdabar\int_{B_{2r}} (|x - y|^{1 - n} + |x - y|^{-1}R_{0}^{1 + \bar{\mu} - n})(Cr^{\bar{\mu}})^{l-1}|y|^{2-n}dy\\
& + K\int_{B_{2r}}(|x - y|^{2-n} + R_{0}^{1 + \bar{\mu} - n}) (Cr^{\bar{\mu}})^{l-1}|y|^{2-n}dy\\
\leq & C(Cr^{\bar{\mu}})^{l-1} \int_{B_{2r}} (\Lambdabar + Kr)|x - y|^{1-n} |y|^{2-n}dy,
\end{align*}
where in the last step we used the inequalities $r \leq R_{0}/2 < R_{0}$ and $|x - y| \leq 4r$. Recalling \eqref{poleconvolution}, we infer from the above string of inequalities that
\begin{equation}
|(u_{l})_{\gamma}(x)| \leq C(Cr^{\bar{\mu}})^{l-1} |x|^{3-n} \leq C(Cr^{\bar{\mu}})^{l}|x|^{2 - n},
\end{equation}
where in the second inequality above we estimated $|x|^{3-n}$ by $|x|^{2-n}r^{\bar{\mu}}$. This completes the induction step, and the claim is proved.
\end{proof}

We now conclude the proof of Theorem \ref{greensgrowth}. Choosing 
\begin{equation*}
R_{1} = \min\{R_{0}/2, 1/8, r_{1}/4, (2C)^{1/\bar{\mu}}\},
\end{equation*}
where $C$ is given by the last claim, we see from \eqref{finalinduction} that $|u_{l}(x)| \leq C2^{-l}|x|^{2-n}$. Summing from $l = 1$ to $\infty$, we get
\begin{equation}
\left|\sum\limits_{l = 1}^{\infty}u_{l}(x)\right| \leq C|x|^{2-n}. 
\end{equation}
Recalling \eqref{Gdecomposition} and \eqref{finalwsup}, we conclude that,
\begin{equation}
|G_{\gamma}(x)| \leq C|x|^{2-n} + Cr^{1 + \bar{\mu} - n} \text{ for $r \leq R_{1}$ and $x \in B_{2r}\setminus \{0\}$, }
\end{equation}
which is precisely the desired conclusion.
\end{proof}

\section{The Hodge Laplacian with a Lipschitz metric}

In this appendix we again work on the unit ball $B_{1}$ and collect some results concerning the bilinear form $\cD_{g}$ and the Hodge Laplacian when the metric $g$ lies in one of the classes from Section 2.1. We also demonstrate below how the results in the previous appendix can be applied to the Hodge Laplacian. 

We begin by noting that since Theorem 7.7.7 of \cite{morreybook} applies when $g \in \cM_{\lambda, \Lambda}$, as pointed out in Remark \ref{lipschitzhodge}, we have that $\cH_{r, \ft}(B_{1}) = \{0\}$ for $0 < r < n$, as is the case when $g$ is smooth. It follows from this absence of non-trivial harmonic forms, the Garding inequality (Theorem 7.5.1 of \cite{morreybook}), and a standard argument by contradiction that, if $g \in \cM^{+}_{\mu, \lambda, \Lambda}$, then there exists $c_{0} = c_{0}(n, \mu, \lambda, \Lambda, r) > 0$ such that
\begin{equation}\label{hodgecoercive}
\cD_{g}(\zeta, \zeta) \geq c_{0}\|\zeta\|_{1, 2; M}^{2} \text{ for all }\zeta \in W^{1, 2}_{r, \ft}(B_{1}),
\end{equation}
Consequently, we obtain the following result.
\begin{lemm}\label{hodgeinvertible}
Suppose $g \in \cM_{\lambda, \Lambda}$. Then for all $l$ in the dual of $W^{1, 2}_{r, \ft}(B_{1})$, there exists an unique $\omega \in W^{1, 2}_{r, \ft}(B_{1})$ such that
\begin{equation}\label{hodgeinverse}
\cD_{g}(\omega, \zeta) = l(\zeta).
\end{equation}
If we assume in addition that $g \in \cM^{+}_{\mu, \lambda, \Lambda}$, then $\|\omega\|_{1, 2; M} \leq C \|l\|$, where $C = C(n, \mu, \lambda, \Lambda, r)$.
\end{lemm}
\begin{rmk}\label{symmetry}
If $g \in \cM_{\mu, \lambda, \Lambda}^{+}$ and if we take $l(\cdot) = (\xi, \cdot)$ where $\xi$ satisfies $\tau^{\ast}\xi = \xi$, then by uniqueness, the solution $\omega$ to \eqref{hodgeinverse} lies in $W^{1, 2, +}_{r, \ft}$.
\end{rmk}

We proceed to describe how the results in Appendix B apply to the Hodge Laplacian associated to a metric in $\cM^{+}_{\mu, \lambda, \Lambda}$. We will only be interested in the case of $2$-forms. Below, we let $N = n(n-1)/2$, fix a basis $(e_{\alpha} = dx^{k_{\alpha}} \wedge dx^{l_{\alpha}})_{1 \leq \alpha \leq N}$ for $\bigwedge^{2}\RR^{n}$ and write two forms in terms of their components, e.g. $\varphi = \varphi_{\alpha}e_{\alpha} = \varphi_{k_{\alpha}l_{\alpha}}dx^{k_{\alpha}}\wedge dx^{l_{\alpha}}$. We begin with the following simple observation.
\begin{lemm}\label{ellipticform}
For $g \in \cM_{\lambda, \Lambda}$, $\varphi \in W^{1, 2}_{2, \ft}(B_{1})$ and $\zeta \in W^{1, 2}_{2, 0}(B_{1})$, there holds
\begin{equation}\label{goodellipticform}
\cD_{g}(\varphi, \zeta) =  \int_{B_{1}}(A^{ij}_{\alpha \beta} \partial_{j}\varphi_{\beta} + B^{i}_{\alpha\beta}\varphi_{\beta})\partial_{i}\zeta_{\alpha} - \zeta_{\alpha}(C^{i}_{\alpha\beta}\partial_{i}\varphi_{\beta} + D_{\alpha\beta}\varphi_{\beta})dx,
\end{equation}
where $B^{i}_{\alpha\beta}, C^{i}_{\alpha\beta}$ and $D_{\alpha\beta}$ are expressions involving $g$ and its Christoffel symbols, and the leading coefficients are given by
\begin{equation}\label{leading}
A^{ij}_{\alpha\beta} = \sqrt{\det(g)} g^{ij}g^{k_{\alpha}k_{\beta}}g^{l_{\alpha}l_{\beta}}.
\end{equation}
\end{lemm}
\begin{rmk}\label{fitfundamentalsolution}
In particular, the coefficients satisfy 
\begin{enumerate}
\item[(1)] $A^{ij}_{\alpha\beta} \in C^{0, 1}(B_{1})$ and $B^{i}_{\alpha\beta}, C^{i}_{\alpha\beta}, D_{\alpha\beta} \in L^{\infty}(B_{1})$, with norms bounded in terms of $n, \lambda$ and $\Lambda$. Moreover, for some constant $\lambdabar$ depending on $n, N, \lambda$ and $\Lambda$,
\begin{equation}
\lambdabar |\xi|^{2} \leq A^{ij}_{\alpha\beta}\xi_{i\alpha}\xi_{j\beta} \leq \lambdabar^{-1}|\xi|^{2}, \text{ for all }x \in B_{1} \text{ and }\xi = (\xi_{i\alpha}) \text{ in }\RR^{nN}.
\end{equation}
\item[(2)]
If $g \in \cM^{+}_{\mu, \lambda, \Lambda}$, then $A^{ij}_{\alpha\beta}\in C^{0, 1}(B_{1})$, while all the lower-order coefficients are $C^{0, \mu}$ on $B_{1}^{+}$ and $B_{1}^{-}$ (but may not be continuous across $\{x^{n} = 0\}$). Furthermore, 
\begin{equation}\label{goodellipticformhalf}
\cD_{g}(\varphi, \zeta) =  2\int_{B_{1}^{+}}(A^{ij}_{\alpha \beta} \partial_{j}\varphi_{\beta} + B^{i}_{\alpha\beta}\varphi_{\beta})\partial_{i}\zeta_{\alpha} - \zeta_{\alpha}(C^{i}_{\alpha\beta}\partial_{i}\varphi_{\beta} + D_{\alpha\beta}\varphi_{\beta})dx,
\end{equation}
for $\varphi \in W^{1, 2, +}_{2, \ft}(B_{1})$ and $\zeta \in W^{1, 2, +}_{2, 0}(B_{1})$.
\end{enumerate}
\end{rmk}
Assuming that $g \in \cM^{+}_{\mu, \lambda, \Lambda}$, then we see by \eqref{hodgecoercive} and Remark \ref{fitfundamentalsolution} that the elliptic operator associated to the bilinear form \eqref{goodellipticform} is precisely of the type considered in Appendix B, and therefore we have the following result.

\begin{lemm}
\label{hodgegreengrowth} 
Suppose $g \in \cM^{+}_{\mu, \lambda, \Lambda}$. The operator associated to $\cD_{g}$ has a fundamental solution $G: B_{1} \times B_{1} \to \RR^{N^{2}}$ (recall that $N = n(n-1)/2$) in the sense that given $p \in (1, \infty)$ and $f = (f_{\alpha})_{1 \leq \alpha \leq N} \in L^{p_{\ast}}(B_{1}; \RR^{N})$, with $p_{\ast} = np/(n + p)$, the unique solution $u \in W^{1, p}_{0}(B_{1}; \RR^{N})$ to the system 
\begin{equation}
\cD_{g}(u, \zeta) = (f, \zeta) \text{ for all }\zeta \in C^{1}_{0}(B_{1}; \RR^{N}),
\end{equation}
 is given for $H^{n}$-a.e. $x \in B_{1}$ by
\begin{equation}\label{hodgegreensrepresentation}
u_{\gamma}(x) = \int_{B_{1}} G_{\alpha\gamma}(x, y) u_{\alpha}(y)dy.
\end{equation}
Moreover, there exists constants $C_{1}, C_{2} > 0$ and $ 1/8 \geq R_{1} > 0$, depending only on $n, c_{0}, \lambda, \Lambda$ and $\mu$, such that
\begin{equation}\label{hodgezerothbound}
|G(x, y)| \leq C_{1}|x - y|^{2-n} + C_{2} R_{2}^{1 + \bar{\mu} - n},
\end{equation}
whenever $x , y \in B_{R_{1}}$ and $x \neq y$. 
\end{lemm}
\begin{rmk}
The representation formula \eqref{hodgegreensrepresentation} follows from Proposition \ref{greensdifferentiable}(3), whereas the upper bound \eqref{hodgezerothbound} of course comes from Corollary \ref{growthnearzero}.
\end{rmk}

\section{Proof of Proposition \ref{fermiproperties}}

Letting $(x^{1}, \cdots, x^{n-1}, t = x^{n})$ be Fermi coordinates centered at some $p \in \partial M$, we identify $\tilde{B}^{+}_{2\rho_{0}}(p)$ with $B^{+}_{2\rho_{0}} \subseteq \RR^{n}$ and introduce the following conventions:
\begin{itemize}
\item Unless otherwise stated, all tensor norms ($|\cdot|$) and covariant derivatives ($\nabla$) will be taken with respect to $g$. 
\item $\widetilde{\nabla}$ denotes the covariant derivative on $T_{2\rho_{0}}$ induced by $g$.
\item $h$ denotes the second fundamental form with respect to $g$ of $T_{2\rho_{0}}$ in $B^{+}_{2\rho_{0}}$.
\item $\delta$ denotes the Euclidean metric.
\item The Latin indices $i, j, \cdots$ will run from $1$ to $n-1$, whereas the Greek indices $\alpha, \beta, \cdots$ range from $1$ to $n$.
\end{itemize}
Also, in all the proofs below we denote by $C$ any constant which depends only those parameters listed in the statements, and $c_{n}$ denotes any dimensional constant. We now describe the proof of Proposition \ref{fermiproperties}, which is adapted from \cite{hamilton} and will be accomplished through a number of lemmas. 

\begin{lemm}\label{tensorODE} Let $S = S_{\alpha_{1}, \cdots, \alpha_{q}}dx^{\alpha_{1}}\cdots dx^{\alpha_{q}}$ be a smooth tensor on $B^{+}_{2\rho_{0}}$ and let $A_{0} = \sup_{|x'| \leq \rho_{0} }|S|(x', 0)$.
\begin{enumerate}
\item[(1)] If $S$ satisfies the differential inequality
\begin{equation*}
|\nabla_{t}S|(x', t) \leq C_{1}|S|^{2}(x', t) + C_{2} \text{ on }B^{+}_{2\rho_{0}}
\end{equation*}
then there exists $t_{0} \in (0, \rho_{0}]$, depending only on $C_{1}, C_{2}$ and $A_{0}$, such that
\begin{equation*}
|S|(x', t) \leq 2\left( A_{0} + C_{2}t\right) \text{ on } T_{\rho_{0}} \times [0, t_{0}] \equiv C^{+}_{\rho_{0}, t_{0}}.
\end{equation*}

\item[(2)] Let $t_{0}$ be given by part (1). If $S$ satisfies the differential inequality
\begin{equation*}
|\nabla_{t}S|(x', t) \leq C_{1}|S|(x', t) + C_{2} \text{ on } C^{+}_{\rho_{0}, t_{0}},
\end{equation*}
then in fact we have
\begin{equation*}
|S|(x', t) \leq C\left( A_{0} + t \right) \text{ on }C^{+}_{\rho_{0}, t_{0}},
\end{equation*}
where $C$ depends only on $C_{1}, C_{2}$ and $t_{0}$.
\end{enumerate}
\end{lemm}
\begin{proof}
For part (1), borrowing an idea from the proof of Corollary 4.8 in \cite{hamilton}, we fix $x' \in T_{\rho_{0}}$ and define the auxiliary function $f(t) = \sqrt{|S|^{2}(x', t) + \eta^{2}}$, where $\eta \in (0, 1)$ is some parameter to be send to zero later. Differentiating $f$ and using the given inequality, we find that
\begin{align*}
|f'| &= \frac{|\langle S, \nabla_{t}S \rangle|}{\sqrt{|T|^{2} + \eta^{2}}} \leq |\nabla_{t}S| \\
&\leq C_{1}|S|^{2} + C_{2}\\
&\leq C_{1}f^{2} + C_{2}.
\end{align*}
The fundamental theorem of Calculus then yields, for all $0 < s \leq t \leq \rho_{0}$, that
\begin{equation*}
f(s) \leq f(0) + t\left( C_{1}\sup\limits_{s \leq t}f(s)^{2} + C_{2} \right).
\end{equation*}
Letting $M(t) = \sup_{s \leq t}f(s)$ and taking supremum on the left-hand side, we obtain
\begin{align*}
M(t) &\leq f(0) + t\left( C_{1}M(t)^{2} + C_{2} \right)\\
&\leq \left( A_{0} + \eta + C_{2}t \right) + C_{1}M(t)^{2} t.
\end{align*}
We may now choose $t_{0}$ depending only on $C_{1}, C_{2}$ and $A_{0}$ such that
\begin{equation*}
1 - 4C_{1}t\left( A_{0} + 1 + C_{2}t \right) > 0 \text{ for }t \leq t_{0}.
\end{equation*}
Then by the continuity of $M(t)$ on $[0, \rho_{0}]$ we must have
\begin{equation*}
M(t) \leq \frac{1 - \sqrt{1 - 4C_{1}t\left( A_{0} + \eta + C_{2}t \right)}}{2C_{1}t} \leq 2(A_{0} + \eta + C_{2}t) \text{ for }t \leq t_{0}.
\end{equation*}
Sending $\eta$ to zero and recalling the definition of $f$, we complete the proof of part (1).

For part (2), again we fix $x' \in T_{\rho_{0}}$, introduce $f(t) = \sqrt{|S|^{2} + \eta^{2}}$ and find after differentiating that 
\begin{equation*}
|f'(t)| \leq C_{1}f(t) + C_{2} = C_{1}\left( f(t) + C_{2}/C_{1} \right)\text{ for }t \leq t_{0},
\end{equation*}
which implies by the Gronwall inequality that
\begin{equation*}
f(t) \leq e^{C_{1}t_{0}}\left( f(0) + Ct \right) \text{ for }t \leq t_{0},
\end{equation*}
with $C$ depending only on $C_{1}$ and $C_{2}$. The proof is complete upon sending $\eta$ to zero and recalling that definition of $f$.
\end{proof}

\begin{lemm}\label{hessiant}
Let $A_{0} = \sup_{T_{\rho_{0}}} |h|$ and let $B_{0} = \sup_{C^{+}_{\rho_{0}, \rho_{0}}}|R|$. Then there exists $t_{0} < \rho_{0}$ and a constant $C$, both depending only on $n, A_{0}, B_{0}$, such that 
\begin{equation}\label{hessiantbound}
|\nabla^{2}t| \leq C \text{ on }C^{+}_{\rho_{0}, t_{0}}.
\end{equation}
\end{lemm}
\begin{proof}
Since straight lines in the $t$-direction parametrize unit-speed geodesics on $M$, we have $\nabla_{t}\nabla t = 0$ on $B^{+}_{\rho_{0}}$, and thus
\begin{equation}\label{hessiankernel}
\nabla^{2}_{t, t}t = \nabla^{2}_{i, t}t = 0.
\end{equation}
To estimate $\nabla^{2}_{i, j}t$, we apply $\nabla^{2}_{i, j}$ to the equation $1 = |\nabla t|^{2}$ to obtain
\begin{align}\label{nabla3t}
 0 & = \nabla^{3}_{i, j, t}t \nabla_{t}t + g^{rs}\nabla^{2}_{i, r}t \nabla^{2}_{j, s}t\\
\nonumber &= \nabla^{3}_{i, t, j}t \nabla_{t}t + g^{rs}\nabla^{2}_{i, r}t \nabla^{2}_{j, s}t\\
\nonumber   &= \left( \nabla^{3}_{t, i, j}t + R_{itj}^{\alpha}\nabla_{\alpha}t \right) \nabla_{t}t +  g^{rs}\nabla^{2}_{i, r}t \nabla^{2}_{j, s}t\\\nonumber
   &= \nabla^{3}_{t, i, j}t + R_{itjt} +  g^{rs}\nabla^{2}_{i, r}t \nabla^{2}_{j, s}t,
\end{align}
where in the last equality we used $\nabla_{i}t = 0$ and $\nabla_{t}t = 1$. Letting $S = \nabla^{2}_{i, j}t dx^{i}dx^{j}$, we find that
\begin{equation*}
|\nabla_{t}S| \leq c_{n}B_{0} + c_{n}|S|^{2} \text{ on }B^{+}_{2\rho_{0}},
\end{equation*}
and that $\sup_{|x'| \leq \rho_{0}}|S|(x', 0) = \sup_{|x'| \leq \rho_{0}}|h|(x') \leq A_{0}$. Therefore by Lemma \ref{tensorODE}(1) we conclude that there exists a $t_{0} \leq \rho_{0}$, depending only on $n, A_{0}$ and $B_{0}$ such that 
\begin{equation}
|\nabla^{2}_{ij}tdx^{i}dx^{j}| \leq C.
\end{equation}
Combining this with \eqref{hessiankernel}, we are done.
\end{proof}

\begin{lemm}\label{nablaqt}
Let $t_{0}$ be as in Lemma \ref{hessiant}. For each $q$, let $A_{q} = \sup_{T_{\rho_{0}}}|\widetilde{\nabla}^{q}h|$ and $B_{q} = \sup_{C^{+}_{\rho_{0}, \rho_{0}}}|\nabla^{q}R|$. Then for each $q \geq 2$, there exists a constant $C$, depending only on $n, q$, $\{A_{m}\}_{0 \leq m \leq q-2}$ and $\{B_{m}\}_{0 \leq m \leq q-2}$ such that
\begin{equation}\label{nablaqtbound}
|\nabla^{q}t| \leq C \text{ on }C^{+}_{\rho_{0}, t_{0}}.
\end{equation}
\end{lemm}
\begin{proof}
The proof proceeds by induction on $q$. The case $q = 2$ is contained in Lemma \ref{hessiant}. Assuming that Lemma \ref{nablaqt} holds up to some $q - 1 \geq 2$, we want to prove it for $q$. We begin by observing that, applying $\nabla^{q}_{\alpha_{1}, \cdots, \alpha_{q}}$ to both sides of $1 = \langle \nabla t, \nabla t \rangle$, we obtain
\begin{equation}\label{nablaqeikonal}
0 = \nabla^{q + 1}_{\alpha_{1}, \cdots, \alpha_{q}, t}t + \sum\limits_{\lambda = 1}^{q}g^{\alpha\beta}\nabla^{q}_{\cdots, \widehat{\alpha_{\lambda}}, \cdots, \alpha}t\nabla^{2}_{\alpha_{\lambda}, \beta}t + \sum\limits_{3 \leq m \leq q - 1}\nabla^{m}t \ast \nabla^{q + 2 - m}t.
\end{equation}
Rearranging and switching the order of differentiation, we have
\begin{equation}\label{nablaqODE}
\nabla^{q + 1}_{t, \alpha_{1}, \cdots, \alpha_{q}}t = - \sum\limits_{\lambda = 1}^{q}g^{\alpha\beta}\nabla^{q}_{\cdots, \widehat{\alpha_{\lambda}}, \cdots, \alpha}t\nabla^{2}_{\alpha_{\lambda}, \beta}t + \Phi(\{\nabla^{m}t\}_{m \leq q - 1}, \{\nabla^{m}R\}_{m \leq q - 2}), 
\end{equation}
where we've used $\Phi(\cdots)$ to denote an expression involving only the tensor within the parentheses, their products, and various contractions (with respect to $g$) thereof. Letting $S = \nabla^{q}t$, then by the induction hypotheses and \eqref{nablaqODE}, we find that
\begin{equation*}
|\nabla_{t}S| \leq C\left( |S| + 1\right) \text{ on }C^{+}_{\rho_{0}, t_{0}}.
\end{equation*}
Therefore by Lemma \ref{tensorODE}(2), we infer that
\begin{equation}\label{nablaqtODEbound}
|\nabla^{q}t| \leq C\left( 1 + \sup\limits_{|x'| \leq \rho_{0}}|\nabla^{q}t|(x', 0) \right) \text{ on }C^{+}_{\rho_{0}, t_{0}}.
\end{equation}
To bound the initial value, we need the following result.
\begin{claim}\label{initialcondition} For each $q \geq 2$, we can express $\nabla^{q}t(x', 0)$ in terms of $\{\widetilde{\nabla}^{m}h(x')\}_{0 \leq m \leq q - 2}$ and $\{\nabla^{m}R(x', 0)\}_{0 \leq m \leq q - 3}$ (the expression does not involve the curvature tensor when $q  = 2$).
\end{claim}
The proof of Claim \ref{initialcondition} will be given after the proof of Lemma \ref{nablaqt}. Assuming the claim for now, we easily see that $\sup_{|x'| \leq 2\rho_{0}}|\nabla^{q}t|(x', 0)$ is bounded in terms of $n, q, \{A_{m}\}_{0 \leq m \leq q - 2}$ and $\{B_{m}\}_{0 \leq m \leq q - 2}$. Plugging this back into \eqref{nablaqtODEbound} completes the proof.
\end{proof}
\begin{proof}[Proof of Claim \ref{initialcondition}]
We first recall that for $q  = 2$ we have the following identities:
\begin{equation}\label{initialcondition2}
\nabla^{2}_{ij}t(x', 0) = h_{ij}(x'),\ \nabla^{2}_{i, t}t = \nabla^{2}_{t, t}t = 0.
\end{equation}
Next, notice that applying various $(q - 1)$-th order covariant derivatives to the identity $1 = \langle \nabla t, \nabla t \rangle$ and performing some routine calculations, we get the following expressions for $\nabla^{q}t$ on $\{t = 0\}$.
\begin{align*}
\nabla^{q}_{i_{1},\cdots, i_{q}}t &= \widetilde{\nabla}_{i_{1}}\left( \nabla^{q - 1}t \right)_{i_{2}, \cdots, i_{q-1}} + \sum\limits_{\lambda = 1}^{q}\nabla^{2}_{i_{1}, i_{\lambda}}t \nabla^{q - 1}_{i_{2}, \cdots, \widehat{i}_{\lambda}, t, \cdots}t.\\
\nabla^{q}_{t, i_{1},\cdots, i_{q - 1}}t &=  - \sum\limits_{\lambda = 1}^{q - 1} g^{rs}\nabla^{q - 1}_{i_{1}, \cdots, \widehat{i}_{\lambda}, \cdots, i_{q - 1}, r} t \nabla^{2}_{i_{\lambda}, s}t + \Phi(\{ \nabla^{m}t \}_{m \leq q  -2}, \{\nabla^{m}R\}_{m \leq q - 3}).\\
\nabla^{q}_{t, t, i_{1}, \cdots, i_{q - 2}}t &= -\sum\limits_{\lambda = 1}^{q-2} g^{rs}\nabla^{q-1}_{t, i_{1}, \cdots, \widehat{i}_{\lambda}, \cdots, i_{q-2}, r}t \nabla^{2}_{i_{\lambda}, s}t + \Phi(\{ \nabla^{m}t \}_{m \leq q - 2}, \{\nabla^{m}R\}_{m \leq q - 3}).\\
&\vdots
\end{align*}
Using the above relations, which determines $\nabla^{q}t(x', 0)$ in terms of $\nabla^{q- 1}t(x', 0)$ and $\{\nabla^{m}R\}_{m \leq q - 3}$, we may start from \eqref{initialcondition2} and argue inductively to complete the proof of Claim \ref{initialcondition}. The details will be omitted.
\end{proof}

We are now ready to derive the estimates asserted in Proposition \ref{fermiproperties}. The fact that these estimates hold on $T_{r_{0}}$ for $r_{0}$ sufficiently small follows from Theorem 4.10 of \cite{hamilton}, and we will extend them in the $t$-direction with the help of Lemmas \ref{hessiant} and \ref{nablaqt}. We begin with the estimate \eqref{fermicomparable}.
\begin{lemm}\label{comparable}
There exists $r_{0} < \rho_{0}$ and a constant $C$, both depending only on $n, A_{0}$ and $B_{0}$, such that 
\begin{equation}\label{closeness}
|g - \delta| \leq C\tilde{r}, \text{ on }C^{+}_{r_{0}} \equiv C^{+}_{r_{0}, r_{0}}.
\end{equation}
\begin{equation}
\label{comparison}
\frac{1}{2}\delta_{ij}\xi^{i}\xi^{j} \leq g_{ij}\xi^{i}\xi^{j} \leq 2\delta_{ij}\xi^{i}\xi^{j}\text{ on }C^{+}_{r_{0}}, \text{ for all }\xi \in \RR^{n}.
\end{equation}
\end{lemm}
\begin{proof}
We first notice that
\begin{equation}\label{deltaODE}
\nabla_{t}\delta_{ij} = -\Gamma_{ti}^r \delta_{rj} - \Gamma_{tj}^{r}\delta_{ir} = -g^{rs}\left(\delta_{js} \nabla^{2}_{i, r}t + \delta_{is}\nabla^{2}_{j, r}t \right).
\end{equation}
Thus, letting $e_{ij} = \delta_{ij} - g_{ij}$ and $e_{it} = e_{tt} \equiv 0$, it follows that
\begin{equation*}
\nabla_{t}e_{ij} = -g^{rs}\left( e_{js}\nabla^{2}_{i, r}t + e_{is}\nabla^{2}_{j, r}t \right)  - 2\nabla^{2}_{i, j}t.
\end{equation*}
Then by Lemma \ref{hessiant} and Lemma \ref{tensorODE} with $e_{ij}dx^{i}dx^{j}$ in place of $S$ and with any small radius $r_{0}$ in place of $\rho_{0}$, we see that
\begin{equation}
|e|(x', t) \leq C\left( \sup_{|x'| \leq r_{1}}|e|(x', 0) + t\right) \text{ on }C^{+}_{r_{0}, t_{0}},
\end{equation}
but Theorem 4.9 of \cite{hamilton} implies that if we choose $r_{0}$ small enough depending on $n, A_{0}, B_{0}$ (the dependence on $|h|$ enters via the Gauss equation), then $|e|(x', 0) \leq C|x'|^{2}$ for all $|x'| \leq r_{0}$, with $C$ depending on the same parameters as $r_{0}$. Therefore, we conclude that for the above choice of $r_{0}$, there holds
\begin{equation}
|e|(x', t) \leq C\left( |x'|^{2} + t \right) \text{ on }C^{+}_{r_{0}, t_{0}},
\end{equation}
from which it's easy to see that \eqref{closeness} and \eqref{comparison} hold on $C^{+}_{r_{0}}$ for small enough $r_{0}$.
\end{proof}

\begin{lemm}\label{nablaqdelta}
Let $r_{0}$ be as in the previous lemma, and the constants $A_{0}, B_{0}, \cdots$ be as in Lemma \ref{nablaqt}. For each $q \geq 0$ there exists a constant $C$ depending only on $n, q$, $\{A_{m}\}$ and $\{B_{m}\}$ such that 
\begin{equation}\label{metricq}
|\nabla^{q}\delta| \leq C \text{ on }C^{+}_{r_{0}}.
\end{equation}
\end{lemm}
\begin{proof}
The proof is again by induction on $q$. For $q = 0$ this desired estimate follows from Lemma \ref{comparable}. Next, assuming that the lemma holds up to some $q - 1 \geq 0$, we want to prove it for $q$.

We begin with some straightforward computations. Note the following more general version of \eqref{deltaODE}
\begin{equation}\label{deltaODEfull}
\nabla_{t}\delta_{\alpha\beta} = -g^{\gamma\sigma}\left( \nabla^{2}_{\alpha,\sigma}t \delta_{\gamma\beta} + \nabla^{2}_{\beta, \sigma}t\delta_{\alpha\gamma} \right).
\end{equation}
Applying $\nabla^{q}_{\alpha_{1},\cdots, \alpha_{q}}$ to this equation, we find that
\begin{align*}
\nabla^{q + 1}_{\alpha_{1}, \cdots, \alpha_{q}, t}\delta_{\alpha, \beta} =& -(\nabla^{q}_{\alpha_{1}, \cdots, \alpha_{q}}\delta_{\gamma\beta} \nabla^{2}_{\alpha, \sigma}t + \nabla^{q}_{\alpha_{1},\cdots, \alpha_{q}}\delta_{\alpha\gamma} \nabla^{2}_{\beta, \sigma}t)g^{\gamma\sigma} \\
& + \Phi(\{\nabla^{m}t\}_{m \leq q + 2}, \{ \nabla^{m}\delta \}_{m \leq q - 1})\\
\Downarrow &\\
\nabla^{q + 1}_{t, \alpha_{1}, \cdots, \alpha_{q}}\delta_{\alpha\beta} =&  -(\nabla^{q}_{\alpha_{1}, \cdots, \alpha_{q}}\delta_{\gamma\beta} \nabla^{2}_{\alpha, \sigma}t + \nabla^{q}_{\alpha_{1},\cdots, \alpha_{q}}\delta_{\alpha\gamma} \nabla^{2}_{\beta, \sigma}t)g^{\gamma\sigma}\\
& + \Phi(\{\nabla^{m}t\}_{m \leq q + 2}, \{ \nabla^{m}\delta \}_{m \leq q - 1}, \{\nabla^{m}R\}_{m \leq q  -2}).
\end{align*}
Thus, by the induction hypotheses and Lemmas \ref{hessiant}, \ref{nablaqt} and \ref{tensorODE}(2), we infer that
\begin{equation}\label{nablaqODEbound}
|\nabla^{q}\delta|(x', t) \leq C\left( \sup\limits_{|x'| \leq r_{0}}|\nabla^{q}\delta|(x', 0) + 1\right) \text{ on }C^{+}_{r_{0}}.
\end{equation}
Next, arguing inductively as in the proof of Claim \ref{initialcondition}, but starting instead with the identities
\begin{align*}
\nabla_{i}\delta_{jk}(x', 0) &= \widetilde{\nabla}_{i}\delta_{jk}(x', 0),\\
\nabla_{t}\delta_{ij}(x', 0) &= -g^{rs}\left( \delta_{is}(x', 0)h_{jr}(x')+ \delta_{js}(x', 0)h_{ir}(x') \right),\\
\nabla_{i}\delta_{tj}(x, 0) &= -g^{rs}h_{is}(x')\delta_{rj}(x', 0) + h_{ij}(x'),
\end{align*}
and passing to higher orders by differentiating the equation \eqref{deltaODEfull} instead of $\langle \nabla t,\nabla t\rangle = 1$, we may show that $\nabla^{q}\delta(x', 0)$ is expressed in terms of $\{\widetilde{\nabla}^{m}\delta(x', 0)\}_{m \leq q}$, $\{\widetilde{\nabla}^{m}h(x')\}_{m \leq q - 1}$ and $\{\nabla^{m}R(x', 0)\}_{m \leq q - 2}$, whose norms on $\{|x'| \leq r_{0}\}$ are all bounded in terms of the parameters listed in the statement of Lemma \ref{nablaqdelta}. For $h, R$ and their derivatives this follows from definition, whereas for $\{\widetilde{\nabla}^{m}\delta \}_{m \leq q-1}$ this follows from Theorem 4.10 of \cite{hamilton}. Combining this with the estimate \eqref{nablaqODEbound}, we are done.
\end{proof}

The estimate \eqref{fermiderivative} now follows from Lemma \ref{comparable} and \ref{nablaqdelta} in the same way Corollary 4.11 follows from Theorem 4.10 in \cite{hamilton}.
\bibliographystyle{amsalpha}
\bibliography{BBOfreeboundary}
\end{document}